\documentclass[11pt,letterpaper]{amsart}
\usepackage{comment}
\usepackage{mathtools}
\usepackage[dvipsnames]{xcolor}
\usepackage{xcolor}

\usepackage{bbm}
\usepackage{hyperref}
\hypersetup{
    colorlinks=true,              
    breaklinks=true,              
    urlcolor= black,              
    linkcolor= black,              
    bookmarksopen=false,
    filecolor=black,
    citecolor=blue,
    linkbordercolor=blue}

\usepackage[utf8]{inputenc}
\usepackage[T1]{fontenc}
\usepackage{textcomp}

\usepackage{chngcntr}

\usepackage{mathrsfs}
\usepackage{nicefrac}

\usepackage[margin=1in]{geometry}

\usepackage{amsmath,amsthm,amssymb, mathtools}

\usepackage{bbm}

\usepackage{graphicx}
\usepackage{faktor}

\newcommand{\U}{\mathcal{U}}

\newcommand{\PP}{\mathbf{P}}
\newcommand{\Pstd}{\PP^{(0)}}
\newcommand{\Pinfty}{\PP^\infty}

\newcommand{\hh}{h_1}
\newcommand{\vh}{h_2}

\newcommand{\R}{\mathbb{R}}

\newcommand{\weaklim}{\textnormal{weak}^*}
\newcommand{\C}{\mathcal{C}}

\newcommand{\N}{\mathbb{N}}

\newcommand{\Z}{\mathbb{Z}}

\newcommand{\vol}{\normalfont \text{Vol}}

\newcommand{\ups}{\upsilon}
\newcommand{\GL}{\normalfont \text{GL}}

\newcommand{\ka}
{\kappa}

\newcommand{\al}
{\alpha}

\newcommand{\be}
{\beta}
\newcommand{\id}
{\textup{id}}

\newcommand{\SL}{\normalfont \text{SL}}

\newcommand{\SO}{\normalfont \text{SO}}
\newcommand{\PO}{\textup{PO}}

\newcommand{\Ad}{\normalfont \text{Ad}}

\newcommand{\G}{\mathcal{G}}
\newcommand{\Gr}{\text{Gr}}

\newcommand{\e}{\epsilon}

\newcommand{\de}{\delta}
\newcommand{\De}{\Delta}

\newcommand{\La}{\Lambda}

\newcommand{\Ga}{\Gamma}
\newcommand{\ga}{\gamma}

\newcommand{\PGL}{\text{PGL}}

\newcommand{\shape}{\text{s}}

\newcommand{\tr}{\intercal}

\newcommand{\HH}{\mathcal{H}}

\newcommand{\D}{\mathcal{D}}

\newcommand{\Span}{\normalfont \text{Span}}
\newcommand{\diag}{\normalfont \text{diag}}

\newcommand{\rank}{\normalfont \text{rank}}
\newcommand{\stab}{\normalfont \text{Stab}}
\usepackage{cleveref}

\newcommand{\Lie}{\textup{Lie}}

\newtheorem{theorem}{Theorem}[section]

\newtheorem{lemma}[theorem]{Lemma}

\newtheorem{proposition}[theorem]{Proposition}

\newtheorem{corollary}[theorem]{Corollary}
\newtheorem*{example*}{Example}
\newtheorem*{conjecture*}{Conjecture}

\theoremstyle{definition}

\newtheorem{definition}[theorem]{Definition}

\theoremstyle{remark}

\newtheorem{remark}[theorem]{Remark}
\newtheorem*{remark*}{Remark}

\numberwithin{equation}{section}

\usepackage{setspace}
\usepackage{tikz-cd}

\usepackage{enumerate}

\begin{document}

\title[Limit distributions for $\SO(n,1)$ actions on $k$-lattices in $\R^{n+1}$]{Limit distributions for $\SO(n,1)$ actions on $k$-lattices in $\R^{n+1}$}
\author{Michael Bersudsky}\author{Nimish A. Shah}
\address{The Ohio State University, Columbus, OH 43210, USA}
\email{bersudsky87@gmail.com, shah@math.osu.edu}

\begin{abstract}
We study the asymptotic distribution of norm ball averages along orbits of a lattice $\Gamma \subseteq \SO(n,1)$ acting on the moduli space of pairs of orthogonal discrete subgroups of $\R^{n+1}$ up to homothety. Our main result shows that, except for   special $2$-lattices in $\R^3$ lying in hyperplanes tangent to the light cone, these measures converge to an explicit semi-invariant probability measure supported on the space of homothety classes of pairs of orthogonal lattices tangent to the light cone. 

Our main motivation is a conjecture of Sargent and Shapira~\cite{Sargent2017DynamicsOT}, which is resolved as a special case of our general result.
\end{abstract}

\maketitle
\tableofcontents
\section{Introduction}
\thispagestyle{empty} 
The moduli space of $k$-dimensional lattices of $\R^n$ for $k<n$ has attracted considerable attention in recent years. This subject arises naturally in number theory with its origins tracing back to the works of Roelcke~\cite{Roelcke56}, Maass~\cite{Maass59}, and later Schmidt~\cite{Schmidt98}, who studied  the distribution of discrete subgroups of $\Z^n$ within large covolume balls — a natural generalization of counting integral vectors in expanding Euclidean balls. More recent interest comes from \cite{emss} and \cite{AES_lat}, which refined the aforementioned works by applying techniques of homogeneous dynamics.

Beyond its number-theoretic motivations, the moduli space of $k$-lattices is a natural object of study in homogeneous dynamics. This space can be realized as a homogeneous space of $\SL(n,\R)$ and admits the structure of a fiber bundle over the Grassmannian of $k$-planes in $\R^n$, with fibers isomorphic to $\SL(k,\R)/\SL(k,\Z)$. Notably, the action of subgroups of $\SL(n,\R)$ on these two spaces exhibits distinct dynamical behaviors. For instance, the action of a one-parameter diagonal subgroup $\{a(t)\}_{t\in\R} \subseteq \SL(n,\R)$ is ergodic and mixing on $\SL(n,\R)/\SL(n,\Z)$, whereas on the Grassmannian, the same flow is  contracting: points either remain fixed or converge to a limit. A concrete example explored in this paper is the action of the diagonal subgroup in \(\SO(n,1)\) on the space of \(k\)-lattices (see \Cref{lem:difference of gs and as}). In this case, all \(k\)-dimensional subspaces contract toward a specific limit, and the orbit of almost every \(k\)-lattice accumulates densely on the collection of \(k\)-lattices contained in this limiting subspace. 

Our main motivation  is the following  conjecture stated in \cite{Sargent2017DynamicsOT}.  Let \(X_2 \cong \PGL(2,\R)/\PGL(2,\Z)\) denote the space of full-rank lattices in \(\R^2\) up to scaling. Given a rank-\(2\) discrete subgroup \(\Lambda\subseteq \R^3\), its \emph{shape} is the equivalence class \(\shape(\Lambda)\in \PO(2)\backslash X_2\) of \(\Lambda\) under scaling and orthogonal transformations of \(\R^3\). Consider \( H_\Z:=\{(x,y,z)\in\Z^3:2xz-y^2=1\}\), and for each $v\in H_\Z$, let \[\La_v:=v^\perp\cap \Z^3,\] where $v^\perp$ is the orthogonal complement of $v$. The conjecture is that the collection $\{\shape(\La_v):v\in H_\Z\}$ is dense in $\PO(2)\backslash X_2$. 

As noted in \cite{Sargent2017DynamicsOT}, the motivation for this conjecture is related to  Furstenberg's approach to resolve a long-standing conjecture on the well-approximability of cubic irrationals. 
The technique of \cite{Sargent2017DynamicsOT} relies on a stationary measure classification for random walks of a Zariski dense subgroup \(\Gamma \subseteq \SO(2,1)\) acting on the moduli space of \(2\)-lattices in \(\R^3\), building on the techniques of Benoist and Quint, see \cite{BenoistQuintBook}. However, an incomplete stationary measure classification and certain unexpected stationary measures called \(k\)-extensions  prevented a full resolution of the conjecture in that work\footnote{The conjecture itself was resolved several years ago by Uri Bader and Uri Shapira by a different method. Their work, even though announced publicly in various talks, is still unpublished. As explained to us in private communication by Uri Shapira, their proof goes through establishing a refined classification of stationary measures (similar to \cite{gorodnik2022stationary}) coupled with an instability result which allows to upgrade the classification to understanding random walks. We stress that the Bader-Shapira results hold only for 2-lattices in $\R^3$, while this paper considers joint distributions of pairs of orthogonal lattices in higher dimensions.}. We also refer the reader to the more recent work \cite{gorodnik2022stationary} studying random walks of lattices in representations of $\SL(2,\R)$ acting on general hybrids of projective and lattice type quotients and proving stronger stationary measure classification.

In this paper, we take a direct approach to resolving the conjecture by analyzing the limiting distribution of orbits of a lattice \(\Gamma \subseteq \SO(n,1)\) in the moduli space of \(k\)-lattices, as the norm of the acting matrices grows. In fact, the conjecture follows as a special case of our general result, \Cref{thm:non-degen orb}. Additionally, we describe the limiting distribution of the aforementioned exceptional orbits in \Cref{thm:degen orbits low dim}, and in higher dimensions we obtain the limiting joint distribution of orbits of pairs of orthogonal complementary lattices, see \Cref{thm:non-degen orb} and \Cref{thm:degen orbits higher dim}.
\subsection*{Concrete definitions and results}
For  natural numbers $r\leq n$, a \textit{$r$-lattice in \(\R^{n+1}\)} is a discrete subgroup of the form 
$$\La=\Span_\Z\{v_1,...,v_r\},$$where $v_1,...,v_r\in \R^{n+1}$ are linearly independent. The homothety class \([\La]\) of an $r$-lattice $\La$ is defined as the equivalence class with respect to the equivalence relation $\La\sim\La'\iff \La=\al \La',$ for $\al\in \R^\times.$

We consider the moduli space $X_{r,n+1}$ given by 
$$\{([\La_1],[\La_2]):\text{$\La_1$ and $\La_2$ are respectively $r$ and $(n+1-r)$-lattices in $\R^{n+1}$ such that } \La_1\perp\La_2\},$$
where $\perp$ represents the  Euclidean orthogonality. Note that if $r=1$ or $r=n$, the pair notation is redundant: one of the lattices is one dimensional, and a one-dimensional lattice is unique up to homothety. 

Induced by the usual linear representation, the matrix group $\SL(n+1,\R)
$ 
acts on $X_{r,n+1}$ as:
\begin{equation}\label{eq:def of action on X}
    g\cdot ([\La_1],[\La_2]):=([g\La_1],[g^*\La_2]),\;\;g\in \SL(n+1,\R),\;\;([\La_1],[\La_2])\in X_{r,n+1},
\end{equation}
where \(g^*:=(g^\tr)^{-1}\).

Consider the quadratic form \(q(x_1,...,x_{n+1}):=x_1^2+\cdots+x_n^2-x_{n+1}^2\), and denote by \(G:=\SO(n,1)^\circ\) the identity component of the special orthogonal group of $q$.
Let $\Ga\subseteq G$ be a lattice, let  $\|\cdot\|$ be a  norm on $M_{n+1}(\R)$, and denote   
$$\Ga_T:=\{\ga\in\Ga:\|\ga\|\leq T\}.$$ 
Our goal is to analyze the limiting distribution of the $\Ga$-norm ball averages on $X_{r,n+1}$:
\begin{equation}\label{eq:def avgs}
   \frac{1}{\#\Ga_T}\sum_{\ga\in\Ga_T}f(\ga\cdot([\La_1],[\La_2])),\,\forall f\in C_c(X_{r,n+1}),\text{ as $T\to \infty$}, 
\end{equation}
for any $([\La_1],[\La_2])\in X_{r,n+1}$. Our main results below show that the above averages converge to an explicit measure which depends on $([\La_1],[\La_2])$ and on the chosen norm. The limiting measures are concentrated on the collection of pairs of lattices lying in subspaces tangent to the light cone $\{\mathbf{x}\in\R^{n+1}:q(\mathbf{x})=0\}$, and we now describe the aforementioned subspaces in more detail.  

We say that a $r$-dimensional vector subspace $V$ of $\R^{n+1}$ is degenerate, if the restricted quadratic form $q|_V$ is degenerate. A subspace $V$ is degenerate if and only if $V$ is tangent to the light cone, and for our particular quadratic form $q$,  its Euclidean orthogonal complement $V^\perp$ is also degenerate. Let $\Gr_r(\R^{n+1})$ denote the Grassmannian of $r$-dimensional subspaces of $\R^{n+1}$. We  define 
\begin{equation}
    \D_{r,n+1}:=\{(P_1,P_2)\in \Gr_r(\R^{n+1})\times \Gr_{n+1-r}(\R^{n+1}):P_1 \text{ is degenerate and }P_2=P_1^\perp \}.
\end{equation}
Let $\mathbf{e}_1,...,\mathbf{e}_{n+1}$ represent the canonical basis vectors of $\R^{n
+1}$.  For $r\leq n$, we reserve a special notation $\Pinfty=(P_1^\infty,P_2^\infty)$ for the following pair of  degenerate subspaces:
\begin{gather*}
P_1^\infty:=\Span_\R\{\mathbf{e}_1+\mathbf{e}_{n+1},\mathbf{e}_2,...,\mathbf{e}_r\}\text{, and }\\
P_2^\infty:=(P_1^\infty)^\perp=
\begin{cases} 
\Span_\R\{\mathbf{e}_1-\mathbf{e}_{n+1},\mathbf{e}_{r+1},...,\mathbf{e}_n\},& \text{if } r\leq n-1 \\
\Span_\R\{\mathbf{e}_1-\mathbf{e}_{n+1}\}, & \text{if } r=n.
\end{cases}
\end{gather*}
Consider the maximal compact subgroup of $\SO(n,1)^0$ given by:
\begin{equation}\label{eq:the maximal cpct}
    K:=\begin{bsmallmatrix}
    \SO(n)&\\
    &1
\end{bsmallmatrix},
\end{equation}
where $\SO(n)$ is the special orthogonal group of the form $x_1 ^2+\dots+x_n^2.$ The subgroup $K$ acts transitively on  $\D_{r,n+1}$, so that \( \D_{r,n+1}\cong K/K_{\Pinfty},\) where $K_{\PP^\infty}=\stab_K(P_1^\infty,P_2^\infty).$ 

We will call a $r$-lattice degenerate, if it is contained in a degenerate subspace, and we define 
\begin{equation}
    \C_{r,n+1}:=\{([\La_1],[\La_2])\in X_{r,n+1}:\La_1\text{ is degenerate}\}.
\end{equation}

The collection \(\C_{r,n+1}\) has the following natural fiber-bundle structure over \(\D_{r,n+1}\). Consider the natural projection \( \pi:X_{r,n+1}\to \Gr_{r}(\R^{n+1})\times \Gr_{n+1-r}(\R^{n+1}),\) defined as:
$$\pi([\La_1],[\La_2]):=(\Span_{\R}(\La_1),\Span_\R(\La_2)).$$
For $k\in \N$ we let 
\[  X_k:=\PGL(k,\R)/\PGL(k,\Z)\cong \SL(k,\R)/\SL(k,\Z),\]
which is identified with the space of (full rank) unimodular lattices of $\R^k$. Then, the fibers of \(\pi\) are isomorphic to $X_r\times X_{n+1-r};$  namely, $\C_{r,n+1}$ is a fiber-bundle over $\D_{r,n+1}$, where the fibers are isomorphic to \(X_r\times X_{n+1-r}.\) For $\PP=(P_1,P_2)\in\D_{r,n+1}$, we denote by $X_{\PP}$ the fiber over $\PP$, i.e.:
\begin{equation}\label{eq:def of fibers of pairs of lats}
    X_\PP=\{([\La_1],[\La_2])\in X_{r,n+1}:\La_1\subseteq P_1,~\La_2\subseteq P_2\}.
\end{equation}

We equip $X_{\PP}$ with the $\PGL(r,\R)\times\PGL(n+1-r,\R)$-invariant probability measure $\mu_\PP$ as follows. Choosing a basis of $P_1$ and $P_2$, we get a representation of $\PGL(r,\R)\times\PGL(n+1-r,\R)$ on $P_1\times P_2$, which in turn identifies $X_{\PP}$ with $X_r\times X_{n+1-r}$ as a $\PGL(r,\R)\times\PGL(n+1-r,\R)$-space, and it admits a unique $\PGL(r,\R)\times\PGL(n+1-r,\R)$-invariant probability measure. This measure is independent of the choice of the basis for the subspaces $P_1$ and $P_2$.

\subsection*{Orbits of nondegenerate orthogonal lattices}  
Let $(\La_1,\La_2)$ be a nondegenerate pair. Without loss of generality, throughout the paper, we will assume that $\La_1$ is such that $q|_{\Span_{\R}(\La_1)}$ is positive definite. This is convenient for stating the results in a uniform manner, and there is no loss of generality in that choice. 

The limiting measure of \eqref{eq:def avgs} for a non-degenerate pair will be obtained by integrating the measures $\mu_\PP$ with respect to a smooth measure on $\D_{r,n+1}\cong K/K_{\Pinfty}$. We will now describe the density on $\D_{r,n+1}$ of the aforementioned measure, which depends on the starting point and the norm. 

Consider the subgroup
\begin{equation*}
    H=\begin{bmatrix}
        \SO(r)&\mathbf{0}\\
        \mathbf{0}& \SO(n-r,1)^\circ
    \end{bmatrix} \subseteq G,
\end{equation*}
and let $dh$ be a Haar measure on $H$. Let $dk$ and $d\kappa$ denote the Haar probability measures on $K$ and $K_{\Pinfty}$, respectively.

For a positive-definite subspace $P_0\in \Gr_r(\R^{n+1})$, we define the density function $w_{P_0}\in C(\D_{r,n+1})$  as follows. By Witt's theorem, there exists a $g_0\in G$ such that \(P_0=g_0\cdot\Span_\R\{\mathbf{e}_1,...,\mathbf{e}_r\}\). We define for all $k_0\in K$,
\begin{gather}\label{eq:wP0}
w_{P_0}(k_0K_{\Pinfty}):=
\begin{cases}
\frac{\int_{K_{\Pinfty}}\int_{H}{\|k_0\kappa a(\infty) hg_0^{-1}\|^{-(n-1)}} dh d\kappa}
{\int_K\int_{H}{\|ka(\infty)hg_0^{-1}\|^{-(n-1)}}dhdk}, & \text{if } r>1,\\[10pt]
\frac{\int_{K_{\Pinfty}}\int_{H}\bigl({\|k_0\kappa a(\infty) hg_0^{-1}\|^{-(n-1)}}+{\|k_0\kappa a(-\infty) hg_0^{-1}\|^{-(n-1)}}\bigr) dh d\kappa}
{\int_K\int_{H}\bigl({\|k a(\infty) hg_0^{-1}\|^{-(n-1)}}+{\|k a(-\infty) hg_0^{-1}\|^{-(n-1)}}\bigr)dhdk}, &\text{if } r=1,
\end{cases} \\[7pt]
   \text{where} \quad
    a(\infty):=\begin{bsmallmatrix}
        \frac{1}{2}&\mathbf{0}&\frac{1}{2}\\
        \mathbf{0}&\mathbf{0}&\mathbf{0}\\ 
        \frac{1}{2}&\mathbf{0}&\frac{1}{2}
    \end{bsmallmatrix} \quad \text{ and } \quad
a(-\infty):=\begin{bsmallmatrix}
        \frac{1}{2}&\mathbf{0}&-\frac{1}{2}\\
        \mathbf{0}&\mathbf{0}&\mathbf{0}\\
        -\frac{1}{2}&\mathbf{0}&\frac{1}{2}
    \end{bsmallmatrix}\label{eq:DefOfainftyIntro}.
\end{gather}

To see that the above expressions are well defined, we refer to~\Cref{lem:integral of Ah inverse estimate} proving convergence of the integrals. In addition, the expression defining $w_{P_0}$  is independent of the choice of $g_0$ because $H$ is the subgroup of $G$ preserving $\Span_\R\{\mathbf{e}_1,...,\mathbf{e}_r\}$, and $dh$ is right-$H$ invariant. 

Observe that  if $\|\cdot\|$ is left $K$-invariant, then $w_{P_0}$ is the constant function $1$, for every $P_0$.

\begin{theorem}\label{thm:non-degen orb}
    Let $\Ga\subseteq G$ be a lattice, and let $([\La_1],[\La_2])\in X_{r,n+1}$ where $\Span_\R(\La_1)$ is positive definite. We denote $P_0:=\Span_{\R}(\La_1).$ Then, for all $f\in C_c(X_{r,n+1}),$
    \begin{equation}
        \lim_{T\to\infty}\frac{1}{\#\Ga_T}\sum_{\ga\in \Ga_T}f(\ga\cdot ([\La_1],[\La_2]))=\int_{K/K_{\Pinfty}}\mu_{k\cdot \Pinfty}(f) w_{P_0}(kK_{\Pinfty})d(kK_{\Pinfty}),
    \end{equation}
 where $d(kK_{\Pinfty})$ is the $K$-invariant probability on $K/K_{\Pinfty}$.
\end{theorem}
We will now verify the conjecture of Sargent--Shapira~\emph{\cite{Sargent2017DynamicsOT}} affirmatively using \Cref{thm:non-degen orb}.
\begin{proof}[Proof of the Sargent--Shapira conjecture] Since for a $2$-lattice $\La\subseteq \R^3$, the homothety class of any lattice in its orthogonal line is unique, the pair notation for $X_{2,3}$ is redundant; that is, we write \[X_{2,3}=\{[\La]:\La\textup{ is a \(2\)-lattice in $\R^3$}\}.\]
For a plane \(P\in \Gr_2(\R^3)\), let $X_P=\{[\La]\in X_{2,3}: \La \subseteq P\}$, and let $\mu_P$ be the $\PGL(2,\R)$-invariant probability on $X_P$. 

We note the identification \(\SO(3)\backslash X_{2,3}\cong \PO(2)\backslash X_2\) obtained by identifying the class $\SO(3)[\La]$ with the shapes of lattices in the horizontal plane $\R^2\times 0$. Hence, the map sending a \(2\)-lattice in $\R^3$ to its shape is the natural quotient map $\shape:X_{2,3}\to \SO(3)\backslash X_{2,3}$.  With this identification, the pushed measure $\shape_*\mu_P$ is the natural measure $\mu_{\PO(2)\backslash X_2}$  on $\PO(2)\backslash X_2$, for every $P\in \Gr_2(\R^3)$. Therefore, given a lattice $\tilde\Ga\subseteq \SO(2,1)^\circ$ and a non-degenerate \(2\)-lattice $\La\subseteq \R^3$, we obtain from \Cref{thm:non-degen orb}
\begin{equation}\label{eq:AvgsInShapSarProof}
    \lim_{T\to\infty}\frac{1}{\#\tilde\Ga_T}\sum_{\tilde\ga\in \tilde\Ga_T}f(\shape([\tilde\ga\La]))=\mu_{\PO(2)\backslash X_2}(f),\;\forall f\in C_c(\PO(2)\backslash X_2).
\end{equation}

Let $Q(x,y,z)=2xz-y^2$ and let $\SO(Q)$ be the special orthogonal group preserving $Q$. Consider $\Ga:=\SO(Q)^\circ \cap \SL(3,\Z)$, and note that $\Ga^*=\Ga$, where $\Ga^*:=(\Ga^\tr)^{-1}$. For $v\in \Z^3\smallsetminus 0,$ we have that
$$\ga\La_{v}=\La_{\gamma^* v},~\text{ for }\ga\in \Ga,$$
where $\La_v:=v^\perp \cap \Z^3$.  Thus, to prove that the shapes $\shape(\La_v)$ for $v\in H_\Z:=\{x\in\Z^3:Q(x)=1\}$ are dense in \(\PO(2)\backslash X_2\), it is sufficient to show that the shapes of the subcollection  $\Ga\La_{v_0}$ are dense, where $v_0:=(1,1,1)\in H_\Z$. 

To apply \Cref{thm:non-degen orb}, we may choose \(M\in \textup{O}(3)\) such that
$M^{-1}\SO(Q)^\circ M=\SO(2,1)^\circ$, and we observe that the plane
$\Span_\R(M^{-1}\Lambda_{v_0})$ is nondegenerate for the form \(x^2+y^2-z^2\). Thus, we may apply \Cref{thm:non-degen orb} for $\tilde\Ga=M^{-1}\Ga M$ and $\La=M^{-1} \La_{v_0}$. Let $x\in \PO(2)\backslash X_2$ and let $\U\subseteq \PO(2)\backslash X_2$ be a compact neighborhood of $x$.  Choose $f\in C_c(\PO(2)\backslash X_2)$ supported in $\overline{\U}$ and positive on $\U$. By \eqref{eq:AvgsInShapSarProof},
\[0<\mu_{\PO(2)\backslash X_2}(f)=\lim_{T\to\infty}\frac{1}{\#\tilde\Ga_T}\sum_{\tilde\ga\in \tilde\Ga_T} f(\shape([\tilde\ga (M^{-1}\La_{v_0})])=\lim_{T\to\infty}\frac{1}{\#\tilde\Ga_T}\sum_{\ga\in \Ga,\; \|M^{-1}\ga M\|\leq T} f(\shape([\ga \La_{v_0}])).\]
In particular, there exists $\ga \in \Ga$ such that $\shape[\ga \La_{v_0}]\in \U$.
\end{proof}

\subsection*{Orbits of degenerate lattices}
We now present our results on the limits of the norm-ball averages \eqref{eq:def avgs} for a degenerate pair $(\La_1,\La_2)$. We split the discussion into the higher-dimensional case $n\geq 3$ and the low-dimensional case $n=2$, due to a special arithmetic phenomenon occurring only when $n=2$. 

When $n\geq 3$, the limits of \eqref{eq:def avgs} for $([\La_1],[\La_2])\in\C_{r,n+1}$ are described by a limiting measure of the same form as in  \Cref{thm:non-degen orb}, where the only difference is the density function on $\D_{r,n+1}$.
\subsubsection*{Orbits of degenerate lattices for $n\geq3$}

Consider the unipotent subgroup $U\subseteq G$ given by
\begin{equation}
    U:=\left\{u(\mathbf{x}):=\exp\left(\begin{bsmallmatrix}
        0&\mathbf{x}&0\\
        -\mathbf{x}^{\tr}&\mathbf{0}&\mathbf{x}^{\tr}\\
        0&\mathbf{x}&0
    \end{bsmallmatrix}\right):\mathbf{x}\in\R^{n-1}\right\}.
\end{equation}
Given a degenerate $r$-dimensional subspace $P_0\in \Gr_r(\R^{n+1})$,  there exists $g_0\in G$ such that $$ P_0=g_0\Span_\R\{\mathbf{e}_1-\mathbf{e}_{n+1},\mathbf{e}_2,\dots,\mathbf{e}_r\}.$$ 
We define a density function $w^\infty_{P_0}\in C(\D_{r,n+1})$ by: 
\begin{equation}
w^\infty_{P_0}(k_0K_{\Pinfty}):=
\frac{\int_{K_{\Pinfty}}\int_{\R^{n-1}}{\|k_0\kappa a(\infty) u(\mathbf{x})g_0^{-1}\|^{-(n-1)}}\, d\mathbf{x} d\kappa}
{\int_K\int_{\R^{n-1}}{\|ka(\infty)u(\mathbf{x})g_0^{-1}\|^{-(n-1)}}\, d\mathbf{x}dk},\:\textup{for }k_0\in K.
\end{equation}
Here $a(\infty)$ is given in \eqref{eq:DefOfainftyIntro}. For the convergence of the above integral, see \Cref{lem:integral of Aux inverse estimate}. The definition $w^\infty_{P_0}$ is independent of the choice of $g_0$; indeed, choosing a different $g_0'$ such that $g'_0\Span_\R\{\mathbf{e}_1-\mathbf{e}_{n+1},\mathbf{e}_2,\dots,\mathbf{e}_r\}=P_0$ changes the above expressions by the same scalar multiple, which cancels in the quotient.

We note that if $\|\cdot\|$ is left $K$-invariant, then $w^\infty_{P_0}\equiv 1$, for every $P_0\in \Gr_r(\R^{n+1})$.

\begin{theorem}\label{thm:degen orbits higher dim}
    Let $n\geq3$. Let $\Ga\subseteq G$ be a lattice, and let $([\La_1],[\La_2])\in \C_{r,n+1}$, i.e.  $\Span_\R(\La_1)$ is degenerate. We denote $P_0:=\Span_{\R}(\La_1).$ Then, for all $f\in C_c(X_{r,n+1}),$
    \begin{equation}
        \lim_{T\to\infty}\frac{1}{\#\Ga_T}\sum_{\ga\in \Ga_T}f(\ga\cdot ([\La_1],[\La_2]))=\int_{K/K_{\Pinfty}}\mu_{k\cdot \Pinfty}(f) w^\infty_{P_0}(kK_{\Pinfty})d(kK_{\Pinfty}),
    \end{equation}
 where $d(kK_{\Pinfty})$ is the $K$-invariant probability on $K/K_{\Pinfty}$.
\end{theorem}

\subsubsection*{Special orbits of degenerate $2$-lattices in $\R^3$}

In this case, the pair notation is redundant. In particular, the collection $\D_{2,3}\subseteq \Gr_2(\R^3)$ is identified with the circle of planes tangent to the light cone, and the collection $\C_{2,3}$ of degenerate $2$-lattices is a bundle over that circle.

Consider the diagonal subgroup $A\subseteq G$ and its expanding unipotent subgroup $U\subset G$, given by
\[
A:=\left\{a(s):=\exp\begin{bsmallmatrix}
        0& 0 &s\\
        0& 0 & 0\\
        s & 0 &0
    \end{bsmallmatrix}:s\in\R \right\},
\quad
U:=\left\{u(t):= \exp\left(\begin{bsmallmatrix}
        0&t&0\\
        -t&0&t\\
        0&t&0
\end{bsmallmatrix}\right):t\in\R\right\}.
\]

We denote \[
\de(s):=\begin{bmatrix}
    e^{\frac{s}{2}} & 0 \\
    0 & e^{-\frac{s}{2}}
\end{bmatrix},
\quad
\ups(t):=\begin{bmatrix}
    1 & t \\
    0 & 1
\end{bmatrix},
\]and we let \(\Phi:\SL(2,\R)\to G\) be the special isogeny (surjective map with finite kernel) characterized by
\begin{equation}\label{eq:CharacterisationOfIsogeny}
    \Phi(\de(s))=a(s),\quad \Phi(\ups(t))=u(t).
\end{equation}

In the following, we define a family of representations of $\SL(2,\R)$ acting on each degenerate plane, which align in a certain way with the $\SL(2,\R)$ representation on \(\R^3\) given by \(\Phi\). 

First consider the standard degenerate plane \(P_0^\infty:=\Span_{\R}\{\mathbf e_1+\mathbf e_3,\mathbf e_2\}\). Denote
\[
v^+:=\mathbf e_1+\mathbf e_3,\quad v^0:=\mathbf e_2.
\]
By representing linear maps on $P_0^\infty$ with respect to the ordered basis $(v^+,v^0)$, we get an isomorphism
\[
\rho^\infty:\SL(2,\R)\to \SL(P_0^\infty);
\]
that is,
\[\rho^\infty\begin{bsmallmatrix}
    a&b\\c&d
\end{bsmallmatrix}v^+=av^++cv^0,\quad\rho^\infty\begin{bsmallmatrix}
    a&b\\c&d
\end{bsmallmatrix}v^0=bv^++dv^0.\]

Observe that the vectors $v^+,v^0$ are eigenvectors of $a(s)$, and the parabolic subgroup $AU\subseteq G$ preserves $P_0^\infty$ and acts on $X_{P_0^\infty}$. Explicitly:
\begin{align}\label{eq:a u action}
    a(s)v^+=e^sv^+,&\quad a(s)v^0=v^0,\\
    u(t)v^+=v^+,&\quad u(t)v^0=tv^++v^0.\nonumber
\end{align}
Then, by \eqref{eq:a u action}, we get the following correspondence:
\begin{equation}\label{eq:alignment}
   \Phi(\de(s))\cdot[\La]=[\rho^\infty(\de(s))\La],
   \quad
   \Phi(\ups(t))\cdot [\La]=[\rho^\infty(\ups(t))\La],\quad \forall\La\subseteq P_0^\infty.
\end{equation}

Next, to describe the correspondence for the remaining degenerate planes, set \(P^\infty_\theta:=k_\theta P_0^\infty,\) where 
\[k_\theta:=\begin{bsmallmatrix}
    \cos(\theta)&-\sin(\theta)&0\\
    \sin(\theta)&\cos(\theta)&0\\
    0&0&1
\end{bsmallmatrix}\in G.
\]

We define the representation \(\rho^\infty_\theta:\SL(2,\R)\to \SL(P^\infty_\theta)\) by
\begin{equation}
   \label{eq:DefOfRhoInftyTheta}
\rho^\infty_\theta(g):=k_\theta \rho^\infty(g) k_\theta^{-1},\quad g\in\SL(2,\R).
\end{equation}
Notice that the parabolic subgroup \(k_\theta AU k_\theta^{-1}\) preserves $P^\infty_\theta$, and similarly to \eqref{eq:alignment}, we have
\begin{equation}\label{eq:TwistedAlignment}
k_\theta\Phi(\de(s))k_\theta^{-1}\cdot[\La]=[\rho_\theta^\infty(\de(s))\La],
\quad
k_\theta\Phi(\ups(t))k_\theta^{-1}\cdot [\La]=[\rho_\theta^\infty(\ups(t))\La],\quad \forall\La\subseteq P_\theta^\infty.   
\end{equation}

We will now define the special points.

\begin{definition}\label{def:special points}
     Let $\Ga\subseteq \SO(2,1)^\circ$ be a lattice, and for $\theta\in[0,2\pi)$, define
\begin{equation}\label{eq:DefGammaThetaPhi}
  \Ga^\Phi_\theta:=\{g\in\SL(2,\R):k_\theta \Phi(g)k_\theta^{-1}\in \Ga\}.
\end{equation} 
We say that a $2$-lattice $\La\subseteq P^\infty_\theta$ is $\Ga$-special, if the orbit \(\{[\rho_\theta^\infty(g)\La]:g\in \Ga^\Phi_\theta\}\) is finite, say
     \[
     \{[\rho_\theta^\infty(g)\La]:g\in \Ga^\Phi_\theta\}=:\{[\La_1],\dots,[\La_m]\}.
     \]
     We refer to this finite set as the $(\Ga,\La)$-packet.
\end{definition}

\begin{remark}\label{rem:StabIslatticeForSpecialOrbits}
    Let
    \[
    \textup{Stab}_\theta([\La]):=\{g\in\SL(2,\R):[\rho_\theta^\infty(g)\La]=[\La]\}.
    \]
    Note that both $\textup{Stab}_\theta([\La])$ and $\Ga^\Phi_\theta$ are lattices in $\SL(2,\R)$. Then $\La$ is $\Ga$-special if and only if $\textup{Stab}_\theta([\La])$ is commensurable with $\Ga^\Phi_\theta$, equivalently if and only if $\textup{Stab}_\theta([\La])\cap \Ga^\Phi_\theta$ is a lattice in $\SL(2,\R)$.
\end{remark}

\begin{remark}
    If $\Lambda$ is $\Gamma$-special, then a conjugate of $\Gamma$ is commensurable with $\Phi(\SL(2,\Z))$. 
\end{remark}

The following lemma shows that the $\Ga$-orbit of a special point is trapped in a multi-section. Denote:
\[\ka_\vartheta:=\begin{bsmallmatrix}
    \cos(\vartheta)&-\sin(\vartheta)\\
    \sin(\vartheta)&\cos(\vartheta)
\end{bsmallmatrix},\]
and note that $\Phi(\ka_{\theta/2})=k_\theta$.
\begin{lemma}\label{lem:containment in multi-section}
    Let $\Ga\subseteq G$ be a lattice and let $\La\subset P^\infty_{\theta_0}$ be a $\Ga$-special $2$-lattice. Then
    \[
    \Ga\cdot[\La]\subseteq \left\{
    k_{\theta}\bigl[\rho_{\theta_0}^\infty(\ka_{-\theta/2})\La_i\bigr]
    :\ 1\le i\le m,\ \theta\in[0,2\pi)
    \right\},
    \]
    where $\{[\La_i]: 1\le i\le m\}$ is the $(\Ga,\La)$-packet.
\end{lemma}
\begin{remark}
    Note that \(\rho_{\theta_0}^\infty(\ka_{-\theta/2})\) is not a Euclidean rotation. Namely, the shapes of the lattices \(\rho_{\theta_0}^\infty(\ka_{-\theta/2})\La_i\) vary when varying $\theta$.
\end{remark}
\begin{proof}
Take $\ga\in\Ga$. Since $\ga P^\infty_{\theta_0}$ is degenerate, there exists $\theta\in[0,2\pi)$ such that \(\ga P^\infty_{\theta_0}=P^\infty_{\theta+\theta_0}.\) Hence
\[
b_\ga:=k_{\theta+\theta_0}^{-1}\ga k_{\theta_0}\in AU.
\]

Choose $g_\ga=\de(s)\ups(t)\in\SL(2,\R)$ such that $\Phi(g_\ga)=b_\ga$, and note that
\[
\ga=k_\theta\left(k_{\theta_0}\Phi(g_\ga) k_{\theta_0}^{-1}\right). 
\]
Then, by \eqref{eq:TwistedAlignment}, \(\ga\cdot[\La]=k_\theta \cdot\left[\rho_{\theta_0}(g_\ga) \La\right]\).

Using that \(\Phi(\ka_{\theta/2})=k_\theta\), we also see that \(\ga=k_{\theta_0}\Phi(\ka_{\theta/2}g_\ga) k_{\theta_0}^{-1}\). Therefore, $\ka_{\theta/2}g_\ga\in\Ga^\Phi_{\theta_0}$, and by definition of the packet, there exists $i$ such that
\[
[\La_i]=[\rho_{\theta_0}^\infty(\ka_{\theta/2}g_\ga)\La].
\]
We may now conclude that:
\[
\ga\cdot[\La]
=
k_\theta\bigl[\rho_{\theta_0}^\infty(\ka_{-\theta/2}(\ka_{\theta/2}g_\ga))\La\bigr]
=
k_\theta\bigl[\rho_{\theta_0}^\infty(\ka_{-\theta/2})\La_i\bigr],
\]
which proves the claim.
\end{proof}

Our main theorem below proves that orbits of special points become dense in their multi-section and are distributed according to an explicit measure. This limiting measure is an $m$-extension, a notion introduced in \cite[Definition 1.5]{Sargent2017DynamicsOT}.  If a lattice is not $\Ga$-special, then the corresponding limiting distribution is as in \Cref{thm:degen orbits higher dim}.
\begin{theorem}\label{thm:degen orbits low dim}
    Let $\Ga\subseteq G$ be a lattice, and let $[\La]\in X_{2,3}$ be a degenerate $2$-lattice. Choose $k_{\theta_0}\in K$ such that $\La\subseteq k_{\theta_0}P^\infty$. Let
    \[
    w_{\theta_0}(\vartheta):=\frac{\int_{\R}\frac{1}{\|k_{\vartheta} a(\infty) u(x)k_{\theta_0}^{-1}\|}\, dx}{\int_0^{2\pi}\int_{\R}\frac{1}{\|k_{\theta} a(\infty) u(x)k_{\theta_0}^{-1}\|}\, dx\,d\theta}.
    \]
    Then:
    \begin{enumerate}
        \item\label{enu:special case nondeg} If $\La$ is $\Ga$-special, then for all $f\in C_c(X_{2,3})$,
    \begin{equation}
        \lim_{T\to\infty}\frac{1}{\#\Ga_T}\sum_{\ga\in \Ga_T}f(\ga\cdot [\La])
        =
        \frac{1}{m}\sum_{i=1}^m\int_0^{2\pi}
        f\!\left(k_{\theta}\bigl[\rho_{\theta_0}^\infty(\ka_{-\theta/2})\La_i\bigr]\right)
        w_{\theta_0}(\theta+\theta_0)\,d\theta,
    \end{equation}
    where $\{[\La_1],\dots,[\La_m]\}$ is the $(\Ga,\La)$-packet.
    \item If $\La$ is not $\Ga$-special, then
    \begin{equation*}
        \lim_{T\to\infty}\frac{1}{\#\Ga_T}\sum_{\ga\in \Ga_T}f(\ga\cdot [\La])
        =
        \int_0^{2\pi}\mu_{k_{\theta} P^\infty}(f)\, w_{\theta_0}(\theta)\,d\theta.
    \end{equation*}
    \end{enumerate}
\end{theorem}

\section{Sums to integrals}
This section describes a general method used to obtain our main result. It is a slight adaptation of \cite[Proposition 3.1]{Gorodnik2004DistributionOL}, and the proof follows the same argument closely.

In the following, $G\leq \SL(n,\R)$ is a closed subgroup, $\De\leq G$ is a lattice, and  $\|\cdot\|$ is a norm. Let $dg$ be the Haar measure on $G$ such that $\vol(G/\De)=1$. For a subset \(L \subseteq G\) and \(T>0\), write
\[
L_T:=\{l\in L:\|l\|\le T\}.
\]

We assume: 
\begin{itemize}
    \item For every $\e>0$ there exists $\de>0$ and $T_0>0$ such that for all $T>T_0$\begin{equation}\label{eq:uniform continuity of norm balls}
    \frac{\vol(G_{T(1+\de)})}{\vol(G_{T})}\leq 1+\e.
\end{equation}
\item  It holds that\begin{equation}\label{eq:counting estimate of lattice pts in gp}
     \lim_{\tau\to\infty}\frac{\#\De_\tau}{\vol(G_\tau)}=1.
    \end{equation}
\end{itemize}
These assumptions hold in considerable generality; in particular, they hold for simple groups \(G\), see for example  \cite{GorodnikNevo2012}.

Suppose that $G$ acts continuously on a manifold $X$, and fix $x_0 \in X$. For $T>0$, define a probability measure $\mu_T$ on $G/\De \times X$ by
\begin{equation*}
\mu_T(\varphi \otimes f)
:=
\frac{1}{\operatorname{vol}(G_T)}
\int_{G_T} \varphi(g\De)\, f(gx_0)\, dg,
\end{equation*}
for $\varphi \in C_c(G/\De)$ and $f \in C_c(X)$.

Let $\{O_\de\}_{\de \in (0,1]}$ be a nested family of symmetric neighborhoods of the identity in $G$ such that $O_\de \downarrow \{e\}$ as $\de \to 0$. For each $\de \in (0,1]$, choose a nonnegative bump function $\chi_\de \in C_c(G)$ supported in $O_\de$ and normalized by \(\int_G \chi_\de(g)\, dg = 1.\) Define
\begin{equation*}
\varphi_\de(g\De) := \sum_{\gamma \in \De} \chi_\de(g\gamma^{-1}).
\end{equation*}
Then $\varphi_\de \in C_c(G/\De)$ and \(\int_{G/\De} \varphi_\de = 1.\)

\begin{proposition}\label{prop:reduction general method}
Fix $x_0 \in X$, and suppose that there exists a measure $\nu$ on $G/\De \times X$ such that
\begin{equation}\label{eq:joint equidistribution assumption}
\mu_T \xrightarrow[T\to\infty]{} \nu
\end{equation}
in the weak-* topology. Then, for every $f \in C_c(X)$,
\begin{equation}
\lim_{T\to\infty}
\frac{1}{\#\De_T}
\sum_{\gamma \in \De_T} f(\gamma x_0)
=
\lim_{\de\to 0} \nu(\varphi_\de \otimes f).
\end{equation}
\end{proposition}

\begin{proof}
It is enough to treat the case $f \ge 0$, since the general case follows by writing $f = f^+ - f^-$.

Fix $\varepsilon > 0$. Since $f$ is compactly supported, it is left uniformly continuous with respect to the $G$-action. Therefore, for $\de > 0$ sufficiently small,
\begin{equation}\label{eq:uniform continuity replacement}
|f(ugx_0) - f(gx_0)| \le \varepsilon
\quad
\text{for all } u \in O_\de \text{ and } g \in G.
\end{equation}
We also choose $\de$ so that
\begin{equation}\label{eq:TranslateOFBallBydelta}
O_\de G_T \subseteq G_{(1+\de)T}
\end{equation}
for all sufficiently large $T$.

We first prove the upper bound. If $\gamma \in \De_T$, then the support of $g \mapsto \chi_\de(g\gamma^{-1})$ is contained in $O_\de \gamma \subseteq G_{(1+\de)T}$, so
\begin{equation}
1
=
\int_G \chi_\de(g\gamma^{-1})\, dg
=
\int_{G_{(1+\de)T}} \chi_\de(g\gamma^{-1})\, dg.
\end{equation}
Hence
\begin{equation}\label{eq:upper first step}
\sum_{\gamma \in \De_T} f(\gamma x_0)
\le
\sum_{\gamma \in \De}
\int_{G_{(1+\de)T}} \chi_\de(g\gamma^{-1})\, f(\gamma x_0)\, dg.
\end{equation}

Moreover, if $\|\gamma\| > (1+\de)^2 T$, then
\begin{equation}
\int_{G_{(1+\de)T}} \chi_\de(g\gamma^{-1})\, f(\gamma x_0)\, dg = 0.
\end{equation}
Indeed, if the integrand is nonzero for some $g \in G_{(1+\de)T}$, then $g\gamma^{-1} \in O_\de$, so $\gamma = ug$ for some $u \in O_\de^{-1}= O_\de$. By \eqref{eq:TranslateOFBallBydelta}, this implies $\gamma \in G_{(1+\de)^2T}$, a contradiction.

Next, for every measurable $B \subseteq G$ and every $\gamma \in \De$, \eqref{eq:uniform continuity replacement} gives
\begin{equation}
\left|
\int_B \chi_\de(g\gamma^{-1})\bigl(f(gx_0) - f(\gamma x_0)\bigr)\, dg
\right|
\le
\varepsilon \int_B \chi_\de(g\gamma^{-1})\, dg
\le
\varepsilon.
\end{equation}
Summing over $\gamma$, and using the vanishing just proved, we obtain
\begin{align}
&\sum_{\gamma \in \De}\int_{G_{(1+\de)T}}
 \chi_\de(g\gamma^{-1})
f(\gamma x_0)\, dg
=\sum_{\gamma \in \De_{(1+\de)^2T}}\int_{G_{(1+\de)T}}
 \chi_\de(g\gamma^{-1})
f(\gamma x_0)\, dg\nonumber\\
&=\int_{G_{(1+\de)T}}
\left(\sum_{\gamma \in \De} \chi_\de(g\gamma^{-1})\right)
f(gx_0)\, dg  + O(\varepsilon)\, \#\De_{(1+\de)^2T}.\label{eq:replace f gamma by f g}
\end{align}
Since $\sum_{\gamma \in \De} \chi_\de(g\gamma^{-1}) = \varphi_\de(g\De)$, combining \eqref{eq:upper first step} with \eqref{eq:replace f gamma by f g} and using  \eqref{eq:counting estimate of lattice pts in gp} yields
\begin{equation}\label{eq:key bound from above}
\frac{1}{\operatorname{vol}(G_T)}
\sum_{\gamma \in \De_T} f(\gamma x_0)
\le
\frac{1}{\operatorname{vol}(G_T)}
\int_{G_{(1+\de)T}} \varphi_\de(g\De)\, f(gx_0)\, dg
+
O(\varepsilon).
\end{equation}
We now prove the lower bound. Similarly to above, if $\|\gamma\| > T$, then for every $u \in O_\de=O_\de^{-1}$ we have $\|u\gamma\| > T/(1+\de)$, and therefore
\begin{equation}
\int_{G_{T/(1+\de)}} \chi_\de(g\gamma^{-1})\, dg = 0.
\end{equation}
It follows that
\begin{equation}\label{eq:DeltaTSumLowerBound}
\sum_{\gamma \in \De}
\int_{G_{T/(1+\de)}} \chi_\de(g\gamma^{-1})\, f(\gamma x_0)\, dg=\sum_{\gamma \in \De_T}
\int_{G_{T/(1+\de)}} \chi_\de(g\gamma^{-1})\, f(\gamma x_0)\, dg
\le
\sum_{\gamma \in \De_T} f(\gamma x_0).
\end{equation}
Using \eqref{eq:uniform continuity replacement} exactly as in the upper-bound argument, but now integrating over $G_{T/(1+\de)}$, we obtain
\begin{align}
\sum_{\gamma \in \De}
\int_{G_{T/(1+\de)}} \chi_\de(g\gamma^{-1})\, f(\gamma x_0)\, dg
=
\int_{G_{T/(1+\de)}}
\varphi_\de(g\De)
f(gx_0)\, dg + O(\varepsilon)\,\#\De_T.
\label{eq:lower replacement}
\end{align}
Therefore, by \eqref{eq:DeltaTSumLowerBound} and \eqref{eq:lower replacement}
\begin{equation}\label{eq:key bound from below}
\frac{1}{\operatorname{vol}(G_T)}
\int_{G_{T/(1+\de)}} \varphi_\de(g\De)\, f(gx_0)\, dg
+
O(\varepsilon)
\le
\frac{1}{\operatorname{vol}(G_T)}
\sum_{\gamma \in \De_T} f(\gamma x_0).
\end{equation}

Define
\begin{equation}
L^-(\de)
:=
\liminf_{T\to\infty}
\frac{\operatorname{vol}(G_{T/(1+\de)})}{\operatorname{vol}(G_T)},
\quad
L^+(\de)
:=
\limsup_{T\to\infty}
\frac{\operatorname{vol}(G_{(1+\de)T})}{\operatorname{vol}(G_T)}.
\end{equation}
Using \eqref{eq:key bound from above}, \eqref{eq:key bound from below}, and the weak-* convergence in \eqref{eq:joint equidistribution assumption}, we get
\begin{align}
L^-(\de)\, \nu(\varphi_\de \otimes f) + O(\varepsilon)
&\le
\liminf_{T\to\infty}
\frac{1}{\operatorname{vol}(G_T)}
\sum_{\gamma \in \De_T} f(\gamma x_0) \nonumber \\
&\le
\limsup_{T\to\infty}
\frac{1}{\operatorname{vol}(G_T)}
\sum_{\gamma \in \De_T} f(\gamma x_0) \nonumber \\
&\le
L^+(\de)\, \nu(\varphi_\de \otimes f) + O(\varepsilon). \label{eq:final squeeze}
\end{align}
By \eqref{eq:uniform continuity of norm balls}, we have $L^\pm(\de) \to 1$ as $\de \to 0$. Since $\varepsilon > 0$ is arbitrary, \eqref{eq:final squeeze} implies
\begin{equation}
\lim_{T\to\infty}
\frac{1}{\operatorname{vol}(G_T)}
\sum_{\gamma \in \De_T} f(\gamma x_0)
=
\lim_{\de\to 0} \nu(\varphi_\de \otimes f).
\end{equation}
Since $\lim_{T\to\infty}\frac{\#\De_T}{\vol(G_T)}=1$, this completes the proof.
\end{proof}

\section{Volume computations}\label{sec:vol estiamtes}
We let $n\geq 2$ and denote $G:=\SO(n,1)^\circ$. In this section, we compute the growth of the norm balls \(G_T:=\{g\in G:\|g\|\le T\}\) and isolate the part of the Haar integral that contributes to the limiting densities appearing in Theorems~\ref{thm:non-degen orb}, \ref{thm:degen orbits higher dim}, and \ref{thm:degen orbits low dim}. To obtain the volume estimates, we use decompositions of \(G\) of the form \(KAH\), where \(A\) is a fixed Cartan subgroup of \(G\), \(K\) is the corresponding maximal compact subgroup, and \(H\) is either a symmetric subgroup or the expanding horospherical subgroup. We begin by introducing these subgroups and recording the relevant Haar-measure decompositions.

For $i,j\in\N\cup\{0\}$ with $i+j=n-1$, let
\begin{equation}\label{eq:DefHij}
    H_{i,j}=
    \begin{bmatrix}
        \SO(i+1) & \mathbf{0}\\
        \mathbf{0} & \SO(j,1)^\circ
    \end{bmatrix}.
\end{equation}
Here $\SO(0,1)^\circ$ and $\SO(1)$ are understood to be the trivial group $\{\id\}$. 

Our diagonalizable subgroup \(A\subseteq G\) is given by
\begin{equation}\label{eq:defining the diagonal subgroup}
    A=\left\{a(s):=\exp\left(\begin{bsmallmatrix}
        \mathbf{0} & 1\\
        1 & \mathbf{0}
    \end{bsmallmatrix}s\right):s\in\R\right\},
\end{equation}
where
\begin{equation}\label{eq:expression for a(s)}
    \exp\left(\begin{bsmallmatrix}
        \mathbf{0} & 1\\
        1 & \mathbf{0}
    \end{bsmallmatrix}s\right)
    =
    \begin{bsmallmatrix}
        \cosh(s) & \mathbf{0} & \sinh(s)\\
        \mathbf{0} & I_{n-1} & \mathbf{0}\\
        \sinh(s) & \mathbf{0} & \cosh(s)
    \end{bsmallmatrix},
\end{equation}
and its expanding horosphere is
\begin{equation}\label{eq:DefOfU}
    U=\left\{u(\mathbf{x}):=\exp\left(\begin{bsmallmatrix}
        0 & \mathbf{x} & 0\\
        -\mathbf{x}^{\tr} & \mathbf{0} & \mathbf{x}^{\tr}\\
        0 & \mathbf{x} & 0
    \end{bsmallmatrix}\right):\mathbf{x}\in\R^{n-1}\right\},
\end{equation}
where
\begin{equation}\label{eq:u_x expression}
   \exp\left(\begin{bsmallmatrix}
        0 & \mathbf{x} & 0\\
        -\mathbf{x}^{\tr} & \mathbf{0} & \mathbf{x}^{\tr}\\
        0 & \mathbf{x} & 0
    \end{bsmallmatrix}\right)
    =
    \begin{bsmallmatrix}
        1-\frac{\|\mathbf{x}\|^2}{2} & \mathbf{x} & \frac{\|\mathbf{x}\|^2}{2}\\
        -\mathbf{x}^{\tr} & I_{n-1} & \mathbf{x}^{\tr}\\
        -\frac{\|\mathbf{x}\|^2}{2} & \mathbf{x} & 1+\frac{\|\mathbf{x}\|^2}{2}
    \end{bsmallmatrix}.
\end{equation}

The Haar measure on $G$ admits the following decompositions; see \cite[Chapter 8]{Schlichtkrull_book} and \cite[Section 3]{shah_oh_gor_satake_2009} for further remarks. Let $K$ be the maximal compact subgroup given by \eqref{eq:the maximal cpct}. When $i=0$, we have
\begin{align}
    \int_G f(g)\,dg
    &=
    \int_K\int_{H_{0,n-1}}\int_0^\infty
    \left(
        f(ka(s)h)+
        f\left(ka(s)\text{\tiny$\begin{bsmallmatrix}
            -I_2&\\
            &I_{n-1}
        \end{bsmallmatrix}$}h\right)
    \right)
    \cosh(s)^{n-1}\,ds\,dh\,dk
    \nonumber\\
    &=
    \int_K\int_{H_{0,n-1}}\int_0^\infty
    \left(
        f(ka(s)h)+
        f\left(ka(s)h\text{\tiny$\begin{bsmallmatrix}
            -I_2&\\
            &I_{n-1}
        \end{bsmallmatrix}$}\right)
    \right)
    \cosh(s)^{n-1}\,ds\,dh\,dk.
    \label{eq:Haar measure in KAHij for i zero}
\end{align}
For $i\neq 0$ and $i+j=n-1$,
\begin{equation}\label{eq:Haar measure in KAHij for i nonzero}
    \int_G f(g)\,dg
    =
    \int_K\int_{H_{i,j}}\int_0^\infty
    f(ka(s)h)\sinh(s)^i\cosh(s)^j\,ds\,dh\,dk.
\end{equation}
Moreover,
\begin{equation}\label{eq:Haar measure in KAU}
    \int_G f(g)\,dg
    =
    \int_K\int_{\R^{n-1}}\int_{-\infty}^\infty
    f(ka(s)u(\mathbf{x}))e^{(n-1)s}\,ds\,d\mathbf{x}\,dk.
\end{equation}

Consider
\begin{equation}\label{eq:defining A}
    a(\infty):=
    \begin{bsmallmatrix}
        1/2 & \mathbf{0} & 1/2\\
        \mathbf{0} & \mathbf{0} & \mathbf{0}\\
        1/2 & \mathbf{0} & 1/2
    \end{bsmallmatrix}.
\end{equation}
The matrix $a(\infty)$ will be useful below, since $a(s)=e^s a(\infty)+O(1)$ as $s\to+\infty$.

\begin{lemma}\label{lem:useful estimates on Hij}
The following estimates hold:
\begin{enumerate}
    \item\label{enu:lower bound on norm of a(s)h}
    For $h\in H_{i,j},\;s\in\R$, we have
    $e^{|s|}\|h\| \ll \|ka(s)h\|$ uniformly in $k\in K$.

    \item\label{enu:norm of h is et}
    Suppose that $h\in H_{i,j}$, where $j\neq 0$. Write $h$ in the following Cartan form\emph{:}
    \begin{equation*}
        h=
        \begin{bsmallmatrix}
            c&\\
            &c_1a_j(t)c_2
        \end{bsmallmatrix},
        \quad
        c\in\SO(i+1),
        \quad
        c_1,c_2\in
        \begin{bsmallmatrix}
            \SO(j)&\\
            &1
        \end{bsmallmatrix},
    \end{equation*}
    where
    \begin{equation*}
        a_j(t)=
        \begin{bsmallmatrix}
            \cosh(t) & \mathbf{0} & \sinh(t)\\
            \mathbf{0} & I_{j-1} & \mathbf{0}\\
            \sinh(t) & \mathbf{0} & \cosh(t)
        \end{bsmallmatrix}.
    \end{equation*}
    Then
    \begin{equation}
        \|h\|\asymp \|a_j(t)\|\asymp e^{|t|}.
    \end{equation}

    \item\label{enu:lower bound on norm of a(s)ux}
    For $u(\mathbf{x})\in U,\;s\in\R$, we have
    $e^s\|\mathbf{x}\|^2+e^{|s|}\ll \|a(s)u(\mathbf{x})\|$.
    Moreover, if $s\geq 0$, then
    $e^s\|u(\mathbf{x})\| \ll \|a(s)u(\mathbf{x})\|$.

    \item\label{enu:norm of ux is x2and1}
    Let $u(\mathbf{x})\in U$. Then
    $\|u(\mathbf{x})\|\asymp 1+\|\mathbf{x}\|^2$.

    \item\label{enu:h over Ah is bdd}
    The ratio $\|h\|/\|a(\infty)h\|$ is uniformly bounded for $h\in H_{i,j}$, and the ratio $\|u(\mathbf{x})\|/\|a(\infty)u(\mathbf{x})\|$ is uniformly bounded for $u(\mathbf{x})\in U$.
\end{enumerate}
\end{lemma}

\begin{proof}
By equivalence of norms, there is no loss of generality in assuming that $\|\cdot\|$ is the Hilbert--Schmidt norm, namely $\|g\|=\sqrt{\text{Trace}(g^\tr g)}$, which is bi-$K$-invariant.
In particular, $$\|k a(s) h\|=\|a(s)h\|,~\forall k\in K,\, \forall a(s)\in A,\, \forall h\in H.$$ 

We now prove \eqref{enu:lower bound on norm of a(s)h} and \eqref{enu:norm of h is et}. First consider $H=H_{n-1,0}$, namely $H=K$. By bi-$K$-invariance, and since $\|h\|=\|I_{n+1}\|=\sqrt{n+1}$ 
\[
\|ka(s)h\|=\|a(s)\|\asymp e^{|s|}\asymp \|h\|e^{|s|},
\]
which gives \eqref{enu:lower bound on norm of a(s)h} in this case.

Now assume $H=H_{i,j}$ with $j\neq 0$.  Writing $h$ in Cartan form as in \eqref{enu:norm of h is et}, and using again bi-$K$-invariance,
\[
\|h\|
=
\left\|
\begin{bmatrix}
I_{i+1}&\\
&a_j(t)
\end{bmatrix}
\right\|.
\]
Since
\[
a_j(t)=
\exp\!\left(
\begin{bmatrix}
\mathbf{0}&1\\
1&\mathbf{0}
\end{bmatrix}t
\right)
=
\begin{bsmallmatrix}
\cosh(t)&0&\sinh(t)\\
0&I_{j-1}&0\\
\sinh(t)&0&\cosh(t)
\end{bsmallmatrix},
\]
we get
\[
\|h\|^2
=
(i+1)+(j-1)+2\cosh(2t)
=
n-1+2\cosh(2t)
\asymp \cosh^2(t).
\]
Hence \(\|h\|\asymp \cosh(t)\asymp e^{|t|},\) which proves \eqref{enu:norm of h is et}.

To prove \eqref{enu:lower bound on norm of a(s)h}, note that the bottom right entry of $a(s)h$ is $\cosh(s)\cosh(t)$. Therefore
\[
\|a(s)h\|
\ge \cosh(s)\cosh(t)
\gg e^{|s|}\cosh(t)
\asymp e^{|s|}\|h\|,
\]
and \eqref{enu:lower bound on norm of a(s)h} follows.

We now verify \eqref{enu:lower bound on norm of a(s)ux}. Note that the bottom right entry of $a(s)u(\mathbf{x})$ equals
\(\frac{e^s}{2}\|\mathbf{x}\|^2+\cosh(s).\) Thus
\[
\|a(s)u(\mathbf{x})\|
\ge \frac{e^s}{2}\|\mathbf{x}\|^2+\cosh(s)
\gg e^s\|\mathbf{x}\|^2+e^{|s|},
\]
which is exactly \eqref{enu:lower bound on norm of a(s)ux}. If $s\ge 0$, then by \eqref{enu:norm of ux is x2and1},
\[
\|a(s)u(\mathbf{x})\|
\gg e^s\bigl(1+\|\mathbf{x}\|^2\bigr)
\asymp e^s\|u(\mathbf{x})\|,
\]
which proves the second assertion in \textup{(3)}.

To prove \eqref{enu:norm of ux is x2and1}, observe first that the bottom right entry of $u(\mathbf{x})$ is \(1+\frac12\|\mathbf{x}\|^2,\) so
\[
\|u(\mathbf{x})\|\gg 1+\|\mathbf{x}\|^2.
\]
On the other hand, from the explicit expression \eqref{eq:u_x expression} for $u(\mathbf{x})$, each matrix entry is $O(1+\|\mathbf{x}\|^2)$, and therefore
\[
\|u(\mathbf{x})\|\ll 1+\|\mathbf{x}\|^2.
\]
Hence
\[
\|u(\mathbf{x})\|\asymp 1+\|\mathbf{x}\|^2,
\]
proving \eqref{enu:norm of ux is x2and1}.

Finally, we verify \eqref{enu:h over Ah is bdd}. If $H=H_{n-1,0}$, then $H$ is compact and, again by orthogonal invariance,
\[\|h\|/\|a(\infty)h\|=\sqrt{n+1}/\|a(\infty)\|.\]
If $H=H_{i,j}$ with $j\neq 0$, then the bottom right entry of $a(\infty)h$ is $\frac12\cosh(t)$, and therefore
\[
\|a(\infty)h\|\gg \cosh(t)\asymp \|h\|.
\]
Likewise, the bottom right entry of $a(\infty)u(\mathbf{x})$ is $\frac12\bigl(1+\|\mathbf{x}\|^2\bigr)$, so
\[
\|a(\infty)u(\mathbf{x})\|
\gg 1+\|\mathbf{x}\|^2
\asymp \|u(\mathbf{x})\|.
\]
This proves \eqref{enu:h over Ah is bdd}.
\end{proof}

\begin{lemma}\label{lem:integral of Ah inverse estimate}
Let $J$ denote either $I_{n+1}$ or $a(\infty)$. Let $H=H_{i,j}$, and denote by $dh$ a Haar measure on $H$.    If $n-2<\al$, then $\int_H\frac{1}{\|Jh\|^\al} dh$ converges with the following tail estimate:
     $$\int_{T\leq \|h\|}\frac{1}{\|Jh\|^\al} dh\ll T^{(n-2)-\al}.$$
Otherwise, for $n-2 \geq \al$:
    \begin{equation}
        \int_{\|h\|\leq T}\frac{1}{\|Jh\|^\al} dh\ll\begin{cases}
        \log(T)& \al=n-2, \\
        T^{n-2-\alpha}& \al<n-2,
    \end{cases}
\end{equation}

\end{lemma}

\begin{proof}
If $j=0$, then $H_{n-1,0}=K$ is compact, so the claims are immediate. Thus assume $j>0$.

In the decomposition of $h\in H_{i,j}$ as
\[
h=
\begin{bmatrix}
c&\\
&c_1a_j(t)c_2
\end{bmatrix},
~ t\geq 0,\quad c\in \SO(i+1), \quad c_1,c_2\in \begin{bsmallmatrix}\SO(j)&\\&1\end{bsmallmatrix},
\]
the integration element of the $H_{i,j}$-Haar measure is given by $dh=dc\,dc_1\,dc_2\sinh(t)^{j-1}dt$; cf. \eqref{eq:Haar measure in KAHij for i zero}. Then, by compactness, and by Lemma \ref{lem:useful estimates on Hij}, 
\[
\int_{\|h\|\le T}\frac{1}{\|Jh\|^\alpha}\,dh
\ll
\int_{e^t\le cT} e^{-\alpha t}\sinh(t)^{j-1}\,dt
\ll
\int_{e^t\le cT} e^{(j-1-\alpha)t}\,dt.
\]

Since $j-1\le n-2$, this gives
\[
\int_{\|h\|\le T}\frac{1}{\|Jh\|^\alpha}\,dh
\ll
\int_{e^t\le cT} e^{((n-2)-\alpha)t}\,dt
\ll
\begin{cases}
1,& \alpha>n-2,\\[1mm]
\log T,& \alpha=n-2,\\[1mm]
T^{\,n-2-\alpha},& \alpha<n-2.
\end{cases}
\]
This proves the stated bound when $\alpha\le n-2$, and also shows convergence of $\int_H \|Jh\|^{-\alpha}\,dh$ when $\alpha>n-2$.

For the tail, again using $\|Jh\|\gg e^t$ and $\|h\|\asymp e^t$,
\[
\int_{T\le \|h\|}\frac{1}{\|Jh\|^\alpha}\,dh
\ll
\int_{e^t\ge c^{-1}T} e^{-\alpha t}\sinh(t)^{j-1}\,dt
\ll
\int_{e^t\ge c^{-1}T} e^{((n-2)-\alpha)t}\,dt
\ll
T^{(n-2)-\alpha},
\]
provided $\alpha>n-2$. This proves the lemma.
\end{proof}

We get the following analogous estimates for $U$.    
\begin{lemma}\label{lem:integral of Aux inverse estimate}
Let $J$ denote either $I_{n+1}$ or $a(\infty)$, and let $d\mathbf{x}$ be a Haar measure on $U\cong \R^{n-1}$.    If $\frac{n-1}{2}<\al$, then:
     $$\int_{T\leq \|u(\mathbf{x})\|}\frac{1}{\|Ju(\mathbf{x})\|^\al} d\mathbf{x}\ll T^{\frac{n-1}{2}-\al},$$
and otherwise:
    \begin{equation}
        \int_{\|u(\mathbf{x})\|\leq T}\frac{1}{\|Ju(\mathbf{x})\|^\al} d\mathbf{x}\ll\begin{cases}
        \log(T)& \al=\frac{n-1}{2}, \\
       T^{\frac{n-1}{2}-\al}& \al<\frac{n-1}{2},
    \end{cases}
\end{equation}
for $J=I_{n+1}$ or $J=a(\infty)$.

\end{lemma}
\begin{proof}
By Lemma \ref{lem:useful estimates on Hij}, there exists some $c>0$ such that:
\begin{align*}
      \int_{\|u(\mathbf{x})\|\leq T}\frac{1}{\|Ju(\mathbf{x})\|^\al} d\mathbf{x}\ll& \int_{\|\mathbf{x}\|\leq \sqrt{cT-1}} \frac{1}{(\|\mathbf{x}\|^2+1)^\al}d\mathbf{x}\\
      \ll& 1+\int_1^{\sqrt{T}}r^{(n-2)-2\al}dr.
\end{align*}
The tail estimate follows similarly.
\end{proof}
For $T>0,~k\in K,~h\in H$ consider
\begin{equation}\label{eq:def of bTkh}
    b_T(k,h)=\{0\leq s :\|k a(s) h\|\leq T\}.
\end{equation}
The following asymptotic estimate will be key in describing the density in the limiting measure appearing in the main theorems.

\begin{lemma}\label{lem:assymp of bT integral for gen case}
Fix $k\in K$. Then
\begin{equation}\label{eq:asymp-bT-general}
\lim_{T\to\infty}\frac{1}{T^{n-1}}\int_H\int_{b_T(k,h)}\omega(s)\,ds\,dh
=
\begin{cases}
 \frac{1}{(n-1)2^{n-1}}\int_H \frac{1}{\|k a(\infty)h\|^{n-1}}\,dh, & H=H_{i,j},\\[1.2em]
 \frac{1}{n-1}\int_H \frac{1}{\|k a(\infty)h\|^{n-1}}\,dh, & H=U,
\end{cases}
\end{equation}
where $\omega(s)=\sinh(s)^i\cosh(s)^j$ when $H=H_{i,j}$, and
$\omega(s)=e^{(n-1)s}$ when $H=U$.
\end{lemma}

\begin{proof} By Lemma \ref{lem:useful estimates on Hij}\eqref{enu:lower bound on norm of a(s)h}
and \eqref{enu:lower bound on norm of a(s)ux}, there exists $c_0>0$ such that
$b_T(k,h)=\emptyset$ whenever $\|h\|\ge c_0T$. Hence
\[
\frac{1}{T^{n-1}}\int_H\int_{b_T(k,h)}\omega(s)\,ds\,dh
=
\frac{1}{T^{n-1}}\int_{\|h\|\le c_0T}\int_{b_T(k,h)}\omega(s)\,ds\,dh.
\]

Also, we observe that 
\[
\bigl|\|ka(s)h\|-e^s\|ka(\infty)h\|\bigr|
\le \|k(a(s)-e^sa(\infty))h\|
\ll \|h\|,
\]
which implies, together with Lemma \ref{lem:useful estimates on Hij}\eqref{enu:h over Ah is bdd},
\[
\|ka(s)h\|
=
e^s\|ka(\infty)h\|\left(1+O\!\left(\frac{\|h\|}{e^s\|ka(\infty)h\|}\right)\right)
=
e^s\|ka(\infty)h\|\bigl(1+O(e^{-s})\bigr).
\]
Fix $\epsilon\in(0,1)$, and choose $s_0>0$ so that for all $s\ge s_0$,
\[
\|ka(s)h\|=e^s\|ka(\infty)h\|\bigl(1+O(\epsilon)\bigr).
\]
Therefore, for some absolute $C>0$,
\begin{align}
\left\{s\ge s_0:\ e^s\le \frac{T}{(1+C\epsilon)\|ka(\infty)h\|}\right\}
&\subseteq b_T(k,h)\label{eq:b_Tinclusion} \\
&\subseteq
[0,s_0]\cup
\left\{s\ge s_0:\ e^s\le \frac{T}{(1-C\epsilon)\|ka(\infty)h\|}\right\}\nonumber.
\end{align}

The contribution of $[0,s_0]$ is negligible. Indeed,
\[
\int_{\|h\|\le c_0T}\int_0^{s_0}\omega(s)\,ds\,dh
\ll_{s_0}\int_{\|h\|\le c_0T}dh.
\]
If $H=H_{i,j}$, then by Lemma \ref{lem:integral of Ah inverse estimate} with $\alpha=0$,
\[
\int_{\|h\|\le c_0T}dh\ll
\begin{cases}
\log T,& n=2,\\
T^{n-2},& n\ge 3,
\end{cases}
\]
while if $H=U$, Lemma \ref{lem:integral of Aux inverse estimate} with $\alpha=0$ gives
\[
\int_{\|h\|\le c_0T}dh\ll T^{\frac{n-1}{2}}.
\]
Hence in both cases
\[
\frac{1}{T^{n-1}}\int_{\|h\|\le c_0T}\int_0^{s_0}\omega(s)\,ds\,dh=o(1).
\]

We proceed to establish the estimate in the two cases $H=H_{i,j}$ and $U$ separately.

\noindent\emph{Case 1: $H=H_{i,j}$.}
Since $i+j=n-1$,
\[
\omega(s)=\sinh(s)^i\cosh(s)^j=\frac{1}{2^{n-1}}e^{(n-1)s}+r_n(s),
\]
where $r_n(s)\ll e^{(n-3)s}$.

Using the previous sandwich \eqref{eq:b_Tinclusion} for $b_T(k,h)$ and the negligible contribution of $[0,s_0]$, we get
\begin{align*}
\frac{1}{T^{n-1}}\int_H\int_{b_T(k,h)}\omega(s)\,ds\,dh
&=
\frac{(1+O(\epsilon))^{n-1}}{(n-1)2^{n-1}}
\int_{\|h\|\le c_0T}\frac{1}{\|ka(\infty)h\|^{n-1}}\,dh \\
&\quad + E_T + o(1),
\end{align*}
where the error $E_T$ comes from $r_n(s)$. 

We proceed to estimate $E_T$ using the inclusion \eqref{eq:b_Tinclusion} and Lemmata \ref{lem:integral of Ah inverse estimate}--\ref{lem:integral of Aux inverse estimate}. For $n=2$,
\[
E_T\ll \frac{1}{T}\int_{\|h\|\le c_0T}dh\ll \frac{\log T}{T}=o(1).
\]
For $n=3$,
\[
E_T\ll \frac{\log T}{T^2}\int_{\|h\|\le c_0T}dh\ll \frac{\log T}{T}=o(1).
\]
For $n\ge 4$,
\[
E_T\ll \frac{1}{T^2}\int_{\|h\|\le c_0T}\frac{1}{\|ka(\infty)h\|^{n-3}}\,dh
=O(T^{-1})=o(1),
\]
by Lemma \ref{lem:integral of Ah inverse estimate}, since $n-3<n-2$.

Thus
\[
\frac{1}{T^{n-1}}\int_H\int_{b_T(k,h)}\omega(s)\,ds\,dh
=
\frac{(1+O(\epsilon))^{n-1}}{(n-1)2^{n-1}}
\int_{\|h\|\le c_0T}\frac{1}{\|ka(\infty)h\|^{n-1}}\,dh
+o(1).
\]
Since $n-1>n-2$, Lemma \ref{lem:integral of Ah inverse estimate} gives
\[
\lim_{T\to\infty}\int_{\|h\|\le c_0T}\frac{1}{\|ka(\infty)h\|^{n-1}}\,dh
=
\int_H\frac{1}{\|ka(\infty)h\|^{n-1}}\,dh.
\]
Letting $T\to\infty$ and then $\epsilon\to 0$ yields
\[
\lim_{T\to\infty}\frac{1}{T^{n-1}}\int_H\int_{b_T(k,h)}\omega(s)\,ds\,dh
=
\frac{1}{(n-1)2^{n-1}}
\int_H\frac{1}{\|ka(\infty)h\|^{n-1}}\,dh.
\]

\medskip
\noindent\emph{Case 2: $H=U$.}
Here $\omega(s)=e^{(n-1)s}$ exactly, so the same argument gives
\[
\frac{1}{T^{n-1}}\int_H\int_{b_T(k,h)}\omega(s)\,ds\,dh
=
\frac{(1+O(\epsilon))^{n-1}}{n-1}
\int_{\|h\|\le c_0T}\frac{1}{\|ka(\infty)h\|^{n-1}}\,dh
+o(1).
\]
Since $n-1>\frac{n-1}{2}$, Lemma \ref{lem:integral of Aux inverse estimate} implies
\[
\lim_{T\to\infty}\int_{\|h\|\le c_0T}\frac{1}{\|ka(\infty)h\|^{n-1}}\,dh
=
\int_H\frac{1}{\|ka(\infty)h\|^{n-1}}\,dh.
\]
Letting $T\to\infty$ and then $\epsilon\to 0$ finishes the proof.
\end{proof}
\begin{corollary}\label{cor:volume of G balls}
Let
\[
G_T:=\{g\in G:\|g\|\le T\},
\quad
\textup{J}:=\begin{bmatrix}
-1&\\
&I_{n}
\end{bmatrix},
\quad
a(-\infty):=Ja(\infty)J.
\]

If $H=H_{0,n-1}$, then
\begin{equation}\label{eq:vol BT H0n-1 case}
\lim_{T\to\infty}\frac{\vol(G_T)}{T^{n-1}}
=
\frac{1}{(n-1)2^{n-1}}
\int_K\int_H
\left(
\frac{1}{\|ka(\infty)h\|^{n-1}}
+
\frac{1}{\|ka(-\infty)h\|^{n-1}}
\right)\,dh\,dk.
\end{equation}

If $H=H_{i,j}$ with $i\neq 0$, then
\begin{equation}\label{eq:vol BT Hij case}
\lim_{T\to\infty}\frac{\vol(G_T)}{T^{n-1}}
=
\frac{1}{(n-1)2^{n-1}}
\int_K\int_H \frac{1}{\|ka(\infty)h\|^{n-1}}\,dh\,dk.
\end{equation}

If $H=U$, then
\begin{equation}\label{eq:vol BT U case}
\lim_{T\to\infty}\frac{\vol(G_T)}{T^{n-1}}
=
\frac{1}{n-1}
\int_K\int_{\mathbb R^{n-1}}
\frac{1}{\|ka(\infty)u(\mathbf x)\|^{n-1}}\,d\mathbf x\,dk.
\end{equation}
\end{corollary}

\begin{proof}
We prove \eqref{eq:vol BT H0n-1 case} and \eqref{eq:vol BT U case}. The proof of
\eqref{eq:vol BT Hij case} is similar.

\emph{Case 1: $H=H_{0,n-1}$.}
For $k\in K$ and $h\in H$, define
\[
b_T^+(k,h):=\{s\ge 0:\|ka(s)h\|\le T\},
\quad
b_T^-(k,h):=\{s<0:\|ka(s)h\|\le T\}.
\]
Since $a(-s)=\textup{J}a(s)\textup{J}$, and since \(\textup{J}h=h\textup{J}\), we have
\[b_T^-(k,h):=\{s<0:\|ka(s)h\|\le T\}=\{s< 0:\|Jk^\textup{J}a(-s)h\textup{J}\|\le T\},\]
where \(k^\textup{J}=\textup{J}k\textup{J}\). Then, after the change of variable $t=-s$ we have
\[
\int_{b_T^-(k,h)}\cosh(s)^{n-1}\,ds
=
\int_{\widetilde b_T(k^\textup{J},h)}\cosh(t)^{n-1}\,dt,
\]
where \(\widetilde b_T(k^\textup{J},h):=\{t\ge 0:\|k^\textup{J}a(t)h\|_\textup{J}\le T\}\), with \(\|g\|_\textup{J}:=\|\textup{J}g\textup{J}\|\).

Hence, by the Haar measure formula \eqref{eq:Haar measure in KAHij for i zero} for $H_{0,n-1}$,
\[
\frac{\vol(G_T)}{T^{n-1}}
=
\frac{1}{T^{n-1}}
\int_K\int_H
\left(
\int_{b_T^+(k,h)}+\int_{\widetilde b_T(k^\textup{J},h)}
\right)\cosh(s)^{n-1}\,ds\,dh\,dk.
\]

Apply Lemma~\ref{lem:assymp of bT integral for gen case} to the first inner integral with respect to the norm $\|\cdot\|$, and to the second inner integral with respect to the norm $\|\cdot\|_{\mathrm J}$, and then dominated convergence in $k$, we obtain
\[
\lim_{T\to\infty}\frac{\vol(G_T)}{T^{n-1}}
=
\frac{1}{(n-1)2^{n-1}}
\int_K\int_H
\left(
\frac{1}{\|ka(\infty)h\|^{n-1}}
+
\frac{1}{\|k^\textup{J}a(\infty)h\|_\textup{J}^{n-1}}
\right)\,dh\,dk.
\]
For the second term above, note that 
\[\frac{1}{\|k^\textup{J}a(\infty)h\|^{n-1}_\textup{J}}=\frac{1}{\|\textup{J}k^\textup{J}a(\infty)h\textup{J}\|^{n-1}} = \frac{1}{\|ka(-\infty)h\|^{n-1}},\]
which proves \eqref{eq:vol BT H0n-1 case}.

\emph{Case 2: $H=U$.}
By the Haar measure formula for $U$,
\[
\frac{\vol(G_T)}{T^{n-1}}
=
\frac{1}{T^{n-1}}
\int_K\int_{\mathbb R^{n-1}}
\left(
\int_{b_T^+(k,u(\mathbf x))}+\int_{b_T^-(k,u(\mathbf x))}
\right)e^{(n-1)s}\,ds\,d\mathbf x\,dk.
\]

We claim that
\begin{equation}\label{eq:neg U negligible}
\frac{1}{T^{n-1}}
\int_K\int_{\mathbb R^{n-1}}
\int_{b_T^-(k,u(\mathbf x))}e^{(n-1)s}\,ds\,d\mathbf x\,dk
=
O\!\left(\frac{\log T}{T^{\frac{n-1}{2}}}\right).
\end{equation}
Indeed, by Lemma \ref{lem:useful estimates on Hij}, \eqref{enu:lower bound on norm of a(s)ux},
there exists $c_0>0$ such that
\[
b_T^-(k,u(\mathbf x))
\subseteq
\{s<0: e^s\|\mathbf x\|^2+e^{-s}\le c_0T\}.
\]
Therefore
\begin{align*}
\int_{\mathbb R^{n-1}}\int_{b_T^-(k,u(\mathbf x))}e^{(n-1)s}\,ds\,d\mathbf x
&\ll
\int_{-\log(c_0T)}^0
\left(\frac{c_0T-e^{-s}}{e^s}\right)^{\frac{n-1}{2}}
e^{(n-1)s}\,ds \\
&=
\int_{-\log(c_0T)}^0
(c_0Te^s-1)^{\frac{n-1}{2}}\,ds \\
&\ll
T^{\frac{n-1}{2}}\log T,
\end{align*}
which proves \eqref{eq:neg U negligible}.

Hence the negative-$s$ contribution is negligible after division by $T^{n-1}$, and Lemma
\ref{lem:assymp of bT integral for gen case} gives
\begin{align*}
\lim_{T\to\infty}\frac{\vol(G_T)}{T^{n-1}}
&=
\lim_{T\to\infty}
\frac{1}{T^{n-1}}
\int_K\int_{\mathbb R^{n-1}}
\int_{b_T^+(k,u(\mathbf x))}e^{(n-1)s}\,ds\,d\mathbf x\,dk \\
&=
\frac{1}{n-1}
\int_K\int_{\mathbb R^{n-1}}
\frac{1}{\|ka(\infty)u(\mathbf x)\|^{n-1}}\,d\mathbf x\,dk.
\end{align*}
This proves \eqref{eq:vol BT U case}.
\end{proof}
Next, for $L>0$, $k\in K$, and $h\in H$, define
\[
d_{T,L}(k,h):=\{0\le s\le \log T-L:\ \|ka(s)h\|\le T\}.
\]

\begin{lemma}\label{lem:the set dtL is epsilon negilble}
Let $H$ be either $H_{i,j}$ or $U$. For every $\epsilon>0$ there exist
$L_0=L_0(\epsilon)$ and $T_0=T_0(\epsilon)$ such that, for all
$L\ge L_0$, $T\ge T_0$, and all $k\in K$,
\[
\frac{1}{T^{n-1}}\int_H\int_{d_{T,L}(k,h)}\omega(s)\,ds\,dh\le \epsilon,
\]
where $\omega(s)=\sinh(s)^i\cosh(s)^j$ if $H=H_{i,j}$, and
$\omega(s)=e^{(n-1)s}$ if $H=U$.
\end{lemma}

\begin{proof}
By Lemma \ref{lem:useful estimates on Hij}, \eqref{enu:lower bound on norm of a(s)h},
and \eqref{enu:lower bound on norm of a(s)ux}, there exists $c_0>0$ such that
\[
d_{T,L}(k,h)\subseteq
\left\{0\le s\le \log T-L:\ e^s\le \frac{c_0T}{\|h\|}\right\}.
\]
Since $\omega(s)\ll e^{(n-1)s}$, we get
\begin{align*}
\frac{1}{T^{n-1}}\int_H\int_{d_{T,L}(k,h)}\omega(s)\,ds\,dh
&\ll
\frac{1}{T^{n-1}}
\int_{\|h\|\le c_0e^L}\int_0^{\log T-L} e^{(n-1)s}\,ds\,dh \\
&\quad+
\frac{1}{T^{n-1}}
\int_{c_0e^L\le \|h\|\le c_0T}
\int_0^{\log(c_0T/\|h\|)} e^{(n-1)s}\,ds\,dh.
\end{align*}
For the first term,
\[
\frac{1}{T^{n-1}}
\int_{\|h\|\le c_0e^L}\int_0^{\log T-L} e^{(n-1)s}\,ds\,dh
\ll
e^{-(n-1)L}\int_{\|h\|\le c_0e^L}dh=o_L(1),
\]
by Lemma \ref{lem:integral of Ah inverse estimate} or
Lemma \ref{lem:integral of Aux inverse estimate} with $\alpha=0$.

For the second term,
\begin{align*}
\frac{1}{T^{n-1}}
\int_{c_0e^L\le \|h\|\le c_0T}
\int_0^{\log(c_0T/\|h\|)} e^{(n-1)s}\,ds\,dh
&\ll
\int_{\|h\|\ge c_0e^L}\frac{1}{\|h\|^{n-1}}\,dh
+
\frac{1}{T^{n-1}}\int_{\|h\|\le c_0T}dh \\
&= o_L(1)+o_T(1),
\end{align*}
by Lemma \ref{lem:integral of Ah inverse estimate} or
Lemma \ref{lem:integral of Aux inverse estimate} with $\alpha=n-1$ for the first term
and $\alpha=0$ for the second.
Now choose $L_0$ so that the $o_L(1)$ terms are $<\epsilon/2$ for $L\ge L_0$, and then
choose $T_0$ so that the $o_T(1)$ term is $<\epsilon/2$ for $T\ge T_0$.
\end{proof}

We denote
\begin{equation}\label{eq:definition of bTL}    
b_{T,L}(k,h):=b_T(k,h)\setminus d_{T,L}(k,h)
=\{s\ge \log T-L:\ \|ka(s)h\|\le T\}.
\end{equation}

Note that  there exists $c_0>0$ such that
if $\|h\|\ge c_0e^L$, then $b_{T,L}(k,h)=\emptyset$  by Lemma \ref{lem:useful estimates on Hij}. Thus, after restricting to the late range $s\ge \log T-L$, the $H$-integration is automatically
confined to the bounded region $\|h\|\ll e^L$, uniformly in $T$. This is used later in Section \ref{sec:PartOfUnity} in the proof
deriving the limiting distribution, via Corollary \ref{cor:H component can be bounded by paying epsilon} below, to
truncate the $h$-integration to a compact range before applying the equidistribution argument.

\begin{corollary}\label{cor:H component can be bounded by paying epsilon}
Let $f:G\to\mathbb R$ be bounded. For every $\epsilon>0$ there exist
$L_0=L_0(\epsilon,f)$ and $T_0=T_0(\epsilon,f)$ such that, for all
$L\ge L_0$, $T\ge T_0$, and all $k\in K$,
\[
\left|
\frac{1}{\vol(G_T)}\int_H\int_{b_T(k,h)} f(ka(s)h)\omega(s)\,ds\,dh
-
\frac{1}{\vol(G_T)}\int_H\int_{b_{T,L}(k,h)} f(ka(s)h)\omega(s)\,ds\,dh
\right|
\le \epsilon.
\]
Moreover, there exists $c_0>0$ such that
$b_{T,L}(k,h)=\emptyset$ whenever $\|h\|\ge c_0e^L$.
\end{corollary}

\begin{proof}
Since $b_T(k,h)\setminus b_{T,L}(k,h)=d_{T,L}(k,h)$,
\begin{align*}
&\left|
\frac{1}{\vol(G_T)}\int_H\int_{b_T(k,h)} f(ka(s)h)\omega(s)\,ds\,dh
-
\frac{1}{\vol(G_T)}\int_H\int_{b_{T,L}(k,h)} f(ka(s)h)\omega(s)\,ds\,dh
\right| \\
&\le
\frac{\|f\|_\infty}{\vol(G_T)}
\int_H\int_{d_{T,L}(k,h)}\omega(s)\,ds\,dh.
\end{align*}
By Corollary \ref{cor:volume of G balls}, $\vol(G_T)\asymp T^{n-1}$, and by
Lemma \ref{lem:the set dtL is epsilon negilble}, \(\frac{1}{\vol(G_T)}
\int_H\int_{d_{T,L}(k,h)}\omega(s)\,ds\,dh\) can be made arbitrarily small by taking first $L$ and then $T$ sufficiently large.
This proves the first claim. The second is exactly the observation above.
\end{proof}

\section{Action on orthogonal lattices}
Let $n\geq 2$ and $r<n+1$ be natural numbers and denote $X:=X_{r,n+1}$.  A challenge in analyzing the $G$-action \eqref{eq:def of action on X} on $X$ is that the lattices not
only change shape, but also move between different ambient subspaces. In our
setting, a convenient way to separate these two effects is to use a $KAH$
decomposition. The subgroup $H$ is chosen to preserve the ambient subspaces of
the lattices, and therefore acts on the lattices inside these subspaces by
explicit transformations. The essential point is then to understand the action
of the diagonal part $A$. We show that asymptotically this action can be
written in the form $R\delta$, where $\delta$ acts on the lattices inside
their subspaces by explicit diagonal matrices in $\PGL(k,\R)$, while $R$
moves the ambient subspaces to the distinguished degenerate pair
$\mathbf{P}^\infty$; see Lemma~\ref{lem:description of gs}.


Combined with Corollary~\ref{cor:H component can be bounded by paying epsilon},
this reduces the original equidistribution problem to the well-studied problem
of equidistribution of expanding translates of bounded manifolds treated in
Section~\ref{sec:equid results}. Those results yield equidistribution for the
lattices contained in $\mathbf{P}^\infty$, and the $K$-component then moves
this limiting distribution to the lattices contained in the other degenerate
pairs.

\subsection{The case of $H_{i,j}$}

Fix nonnegative integers $i,j$ such that $i+j=n-1$. Consider the pair of nondegenerate orthogonal complementary subspaces
$P^{(0)}_1,P^{(0)}_2\subseteq \R^{n+1}$ defined by
\begin{align}
    P^{(0)}_1&:=\Span_\R\{\mathbf{e}_1,\dots,\mathbf{e}_{i+1}\},\nonumber\\
    P^{(0)}_2&:=\Span_\R\{\mathbf{e}_{i+2},\dots,\mathbf{e}_{n+1}\},\label{eq:def of ver and hor vspaces}\\
    \Pstd&:=(P^{(0)}_1,P^{(0)}_2).\nonumber
\end{align}

Consider \(X_{\Pstd}:=\{([\Lambda_1],[\Lambda_2])\in X:\Lambda_1\subseteq P^{(0)}_1\},\) and observe that the subgroup $H_{i,j}\subseteq \SO(n,1)^\circ$ acts on $X_{\Pstd}$ with respect to the  action induced by the $G$-action  \eqref{eq:def of action on X}.
More specifically, for $h\in H_{i,j}$, write
\[
h=
\begin{bmatrix}
h_1 & 0\\
0 & h_2
\end{bmatrix},
~
h_1\in \SO(i+1),~ h_2\in \SO(j,1)^\circ.
\]
Then,
\begin{equation}
    h\cdot([\Lambda_1],[\Lambda_2])
    =
    ([h_1\Lambda_1],[h_2^*\Lambda_2]),\quad h\in H_{i,j}.
   \end{equation}

It will be convenient to work with the following matrix representation. Let
\[
\rho^{(0)}:\GL(i+1,\R)\times \GL(j+1,\R)\to \GL(P^{(0)}_1)\times\GL(P^{(0)}_2)
\]
be the representation obtained by expressing linear maps
$P^{(0)}_1\to P^{(0)}_1$ and $P^{(0)}_2\to P^{(0)}_2$
in the ordered bases
\[
B_1:=(\mathbf{e}_1,\dots,\mathbf{e}_{i+1}),
\quad
B_2:=(\mathbf{e}_{i+2},\dots,\mathbf{e}_{n+1}).
\]

For $k\in\N$, denote
\[
X_{k}
:=\PGL(k,\R)/\PGL(k,\Z) .
\]
The representation $\rho^{(0)}$ induces an action of
$\PGL(i+1,\R)\times \PGL(j+1,\R)$ and hence an identification
\[
X_{\Pstd}\cong X_{i+1}\times X_{j+1},
\]
under which the identity coset corresponds to \(\bigl([\Span_{\Z}\{\mathbf{e}_1,\dots,\mathbf{e}_{i+1}\}],
[\Span_{\Z}\{\mathbf{e}_{i+2},\dots,\mathbf{e}_{n+1}\}]\bigr).\)

For $([\Lambda_1],[\Lambda_2])\in X_{\Pstd}$, if
$(y_1,y_2)\in X_{i+1}\times X_{j+1}$ is such that \(\rho^{(0)}(y_1,y_2)=([\Lambda_1],[\Lambda_2]),\) then
\begin{equation}\label{eq:action-Hij-on-XP0}
    \begin{bsmallmatrix}
h_1 & 0\\
0 & h_2
\end{bsmallmatrix}\cdot([\Lambda_1],[\Lambda_2])
    =
    \rho^{(0)}\bigl(h_1y_1,h_2^*y_2\bigr).
\end{equation}

We now define matrices $g_1(s)$ and $g_2(s)$, for $s>0$, that describe the
asymptotic action of $a(s)$ on $X_{\Pstd}$.

Let $v^+,v^-\in\R^{n+1}$ be given by
\begin{equation}\label{eq:def of eigenvecs vpm}
    v^+:=\begin{psmallmatrix}
        1\\
        \mathbf{0}\\
        1
    \end{psmallmatrix},
    ~~
    v^-:=\begin{psmallmatrix}
        1\\
        \mathbf{0}\\
        -1
    \end{psmallmatrix}.
\end{equation}
Then
\begin{align}
    a(s)v^+=e^s v^+,\quad
    a(s)v^-=e^{-s}v^-,
    \quad
    a(s)\mathbf{e}_q=\mathbf{e}_q\quad (2\le q\le n).
    \label{eq:a(s) eigenval eqn}
\end{align}

Since $(\mathbf{e}_1,\mathbf{e}_2,\dots,\mathbf{e}_n,v^-)$ and
$(v^+,\mathbf{e}_2,\dots,\mathbf{e}_n,\mathbf{e}_{n+1})$
are bases of $\R^{n+1}$, we define $g_1(s),g_2(s)\in \GL(n+1,\R)$ by
\begin{equation}\label{eq:def g1s}
    g_1(s)\mathbf{e}_1:=\frac{e^s}{2}v^+,\quad
    g_1(s)\mathbf{e}_q:=\mathbf{e}_q\ (2\le q\le n),\quad
    g_1(s)v^-:=e^{-s}v^-,
\end{equation}
and
\begin{equation}\label{eq:def g2s}
    g_2(s)v^+:=e^{-s}v^+,\quad
    g_2(s)\mathbf{e}_q:=\mathbf{e}_q\ (2\le q\le n),\quad
    g_2(s)\mathbf{e}_{n+1}:=-\frac{e^s}{2}v^-.
\end{equation}

\begin{lemma}\label{lem:difference of gs and as}
It holds that
\[
\lim_{s\to\infty}a(s)(g_1(s))^{-1}=\id,
\quad
\lim_{s\to\infty}a(s)^*(g_2(s))^{-1}=\id.
\]
\end{lemma}

\begin{proof}
It is enough to check the convergence on the basis
$(v^+,\mathbf{e}_2,\dots,\mathbf{e}_n,v^-)$. 

We first prove that $a(s)(g_1(s))^{-1}\to \id$. For $2\le q\le n$,
\[
a(s)(g_1(s))^{-1}\mathbf{e}_q=\mathbf{e}_q.
\]
Also,
\[
a(s)(g_1(s))^{-1}v^-=a(s)(e^s v^-)=v^-.
\]
Finally, since $2\mathbf{e}_1=v^++v^-$, we get
\[
a(s)(g_1(s))^{-1}v^+
=
a(s)(2e^{-s}\mathbf{e}_1)
=
e^{-s}a(s)(v^++v^-)
=
v^+ + e^{-2s}v^-
\longrightarrow v^+.
\]
Hence $a(s)(g_1(s))^{-1}\to \id$.

Similarly, for the second limit, for $2\le q\le n$,
\[
a(s)^*(g_2(s))^{-1}\mathbf{e}_q=\mathbf{e}_q,
\]
and
\[
a(s)^*(g_2(s))^{-1}v^+=a(s)^*(e^s v^+)=v^+.
\]
Since $2\mathbf{e}_{n+1}=v^+-v^-$, we have
\[
(g_2(s))^{-1}v^-=-2e^{-s}\mathbf{e}_{n+1}=e^{-s}(v^- - v^+).
\]
Therefore,
\[
a(s)^*(g_2(s))^{-1}v^-
=
e^{-s}\bigl(a(s)^*v^- - a(s)^*v^+\bigr)
=
v^- - e^{-2s}v^+
\longrightarrow v^-.
\]
This proves $a(s)^*(g_2(s))^{-1}\to \id$.
\end{proof}

Choose $R_1\in \SL(n+1,\R)$ such that
\begin{equation}\label{eq:def of R1}
    R_1\mathbf{e}_1=\frac{1}{2}v^+,~~
    R_1\mathbf{e}_q=\mathbf{e}_q \textup{ for }2\le q\le n,
\end{equation}
and choose $R_2\in \SL(n+1,\R)$ such that
\begin{equation}\label{eq:def of R2}
    R_2\mathbf{e}_q=\mathbf{e}_q \textup{ for }2\le q\le n,
    ~~
    R_2\mathbf{e}_{n+1}=-\frac{1}{2}v^-.
\end{equation}
We denote
\[R([\La_1],[\La_2]):=([R_1\La_1],[R_2\La_2]),\;([\La_1],[\La_2])\in X_{r,n+1}.\]

Consider the complementary orthogonal  degenerate subspaces
\begin{align}
    P_1^{\infty}&:=\Span_\R\{v^+,\mathbf e_2,\dots,\mathbf e_{i+1}\},\nonumber\\
    P_2^{\infty}&:=\Span_\R\{\mathbf e_{i+2},\dots,\mathbf e_n,v^-\},\label{eq:def of Pinfty}\\
    \Pinfty&:=(P_1^{\infty},P_2^{\infty}).\nonumber
\end{align}
We notice that both $R_i$ and $g_i(s)$ map $P_i^{(0)}$ into $P_i^{\infty}$ for $i=1,2$. In particular, $R$ maps $X_{\Pstd}$ bijectively onto $X_{\Pinfty}$.

Denote
\begin{equation}\label{eq:defining d}
   d_1(s):=\begin{bsmallmatrix}
        e^{\frac{is}{i+1}}&0\\
        0&e^{-\frac{s}{i+1}}I_i
    \end{bsmallmatrix}\in \SL(i+1,\R), \quad d_2(s):=\begin{bsmallmatrix}
        e^{-\frac{s}{j+1}}I_j&0\\
        0&e^{\frac{js}{j+1}}
    \end{bsmallmatrix}\in \SL(j+1,\R).
\end{equation}

\begin{lemma}\label{lem:description of gs}
Let $([\Lambda_1],[\Lambda_2])\in X_{\Pstd}$ and let
$(y_1,y_2)\in X_{i+1}\times X_{j+1}$ such that \(\rho^{(0)}(y_1,y_2)=([\Lambda_1],[\Lambda_2]).\) Then
\begin{equation}\label{eq:description g1g2}
    ([g_1(s)\Lambda_1],[g_2(s)\Lambda_2])
    =
    R\Bigl(\rho^{(0)}\bigl(d_1(s)y_1,d_2(s)y_2\bigr)\Bigr).
\end{equation}
\end{lemma}
\begin{proof}
Since $g_i(s)$ and $R_i$ map $P_i^{(0)}$ to $P_i^{\infty}$ for $i=1,2$, we obtain that
\[
A_1(s):=\left.R_1^{-1}g_1(s)\right|_{P^{(0)}_1}\in \GL(P^{(0)}_1),
\quad
A_2(s):=\left.R_2^{-1}g_2(s)\right|_{P^{(0)}_2}\in \GL(P^{(0)}_2).
\]
To prove the claim, we compute the matrices representing $A_1(s)$ and $A_2(s)$ in the ordered bases \((\mathbf e_1,\dots,\mathbf e_{i+1})\) and \((\mathbf e_{i+2},\dots,\mathbf e_{n+1})\)  respectively.

For $A_1(s)$, we have
\[
A_1(s)\mathbf e_1
=
R_1^{-1}\!\left(\frac{e^s}{2}v^+\right)
=
e^s\mathbf e_1,
\quad
A_1(s)\mathbf e_q=\mathbf e_q\quad(2\le q\le i+1).
\]
Hence
\[
[A_1(s)]_{B_1}=\diag(e^s,1,\dots,1)=e^{\frac{s}{i+1}}d_1(s).
\]

For $A_2(s)$, we have
\[
A_2(s)\mathbf e_q=\mathbf e_q\quad(i+2\le q\le n),
\quad
A_2(s)\mathbf e_{n+1}
=
R_2^{-1}\!\left(-\frac{e^s}{2}v^-\right)
=
e^s\mathbf e_{n+1}.
\]
Hence
\[
[A_2(s)]_{B_2}=\diag(1,\dots,1,e^s)=e^{\frac{s}{j+1}}d_2(s).
\]

Therefore, for $([\Lambda_1],[\Lambda_2])\in X_{\Pstd}$ and \((y_1,y_2)\in X_{i+1}\times X_{j+1}\) such that  $\rho^{(0)}(y_1,y_2)=([\Lambda_1],[\Lambda_2])$, we obtain that
\[
([A_1(s)\La_1],[A_2(s)\La_2])=\rho^{(0)}(e^{\frac{s}{i+1}}d_1(s)y_1,e^{\frac{s}{j+1}}d_2(s)y_2)=\rho^{(0)}(d_1(s)y_1,d_2(s)y_2).
\]

Applying $R$, we obtain \eqref{eq:description g1g2}.
\end{proof}
For $x\in G/\Ga$, $y=(y_1,y_2)\in X_{i+1}\times X_{j+1}$, $k\in K$, and $T>0$,
we define the following measure on $G/\Ga\times X_{i+1}\times X_{j+1}$ by
\begin{equation}\label{eq:def of msr on XiXj Hij case}
   \mu_{x,y,k,T}(F):=
   \frac{1}{\vol(G_T)}
   \int_{H_{i,j}}\int_{b_T(k,h)}
   F\bigl(a(s)hx,d_1(s)h_1y_1,d_2(s)h_2^*y_2\bigr)\omega(s)\,ds\,dh,
\end{equation}
where $b_T(k,h)$ is given in \eqref{eq:def of bTkh}, and \(\omega(s)=\sinh(s)^i\cosh(s)^j.\)

Define $\Psi:G/\Ga\times X_{i+1}\times X_{j+1}\to G/\Ga\times X_{\Pinfty}$ by:
\begin{align}
    \Psi(x,y)&:=\bigl(x,R\rho^{(0)}(y)\bigr),\quad x\in G/\Ga,\;y\in X_{i+1}\times X_{j+1}, \label{eq:def of Psi}
\end{align}
and, for $k\in K$, define
\begin{equation}\label{eq:def kPsi}
    k\Psi(x,y):=\bigl(kx,k\cdot (R\rho^{(0)}(y))\bigr),
\end{equation}
where in the second coordinate $k$ acts as in \eqref{eq:def of action on X}.

\begin{lemma}\label{lem:replacing by averages in GmodD times expading in XiXj}
Fix $x_0\in G/\Ga$, let $([\Lambda_1],[\Lambda_2])\in X_{\Pstd}$, and choose
 $y_0\in X_{i+1}\times X_{j+1}$ with \(\rho^{(0)}(y_0)=([\Lambda_1],[\Lambda_2]).\) 
 
 Suppose that for every $k\in K$,
\begin{equation}\label{eq:requirement for equid in nondeg case}
    \underset{T\to\infty}{\weaklim\lim}\ \mu_{x_0,y_0,k,T}=\mu_{k,\infty}.
\end{equation}
Then, for every $F\in C_c(G/\Ga\times X)$,
\begin{align*}
    \lim_{T\to\infty}\frac{1}{\vol(G_T)}
    \int_{H_{i,j}}\int_{b_T(k,h)}
    F\bigl(ka(s)hx_0,ka(s)h\cdot([\Lambda_1],[\Lambda_2])\bigr)\omega(s)\,ds\,dh
    =
    (k\Psi)_*\mu_{k,\infty}(F).
\end{align*}
\end{lemma}

The above lemma follows as a direct corollary of Lemmata~\ref{lem:difference of gs and as}--\ref{lem:description of gs} as follows. By applying
Lemma~\ref{lem:description of gs} to the point
\(
h\cdot([\Lambda_1],[\Lambda_2])
=
([h_1\Lambda_1],[h_2^*\Lambda_2])
\),
and then using Lemma~\ref{lem:difference of gs and as} to replace
\(
([g_1(s)h_1\Lambda_1],[g_2(s)h_2^*\Lambda_2])
\)
by
\(
a(s)h\cdot([\Lambda_1],[\Lambda_2])
\)
inside the test function. The contribution of bounded values of \(s\)
is negligible after normalization by \(\vol(G_T)\), exactly as in the proof of Lemma~\ref{lem:assymp of bT integral for gen case}.

\subsection{The case of $U$}

The discussion here is analogous to the previous subsection treating the case
$H=H_{i,j}$.
Recall the pair of orthogonal degenerate subspaces
$\Pinfty=(P_1^\infty,P_2^\infty)$ from \eqref{eq:def of Pinfty}.
We consider the representation
\begin{equation}\label{eq:def rho infinity}
    \rho^\infty:\GL(i+1,\R)\times \GL(j+1,\R)\to \GL(P_1^\infty)\times\GL(P_2^\infty),
\end{equation}
defined by expressing linear transformations
$P_1^\infty\to P_1^\infty$ and $P_2^\infty\to P_2^\infty$
in the ordered bases
\[
B_1^\infty:=(v^+,\mathbf{e}_2,\dots,\mathbf{e}_{i+1}),
\quad
B_2^\infty:=(v^-,\mathbf{e}_{i+2},\dots,\mathbf{e}_n),
\]
respectively.

For $\mathbf{x}=(\boldsymbol a,\boldsymbol b)\in\R^{n-1}=\R^i\times\R^j$, we define
\begin{equation}\label{eq:defining upsilon}
\upsilon(\mathbf{x})
=
\upsilon(\boldsymbol a,\boldsymbol b)
:=
\left(
\begin{bmatrix}
1&\boldsymbol a\\
\mathbf 0&I_i
\end{bmatrix},
\begin{bmatrix}
1&\boldsymbol b\\
\mathbf 0&I_j
\end{bmatrix}
\right)
\in \GL(i+1,\R)\times\GL(j+1,\R).
\end{equation}
We also define
\begin{equation}\label{eq:diagonal pushing in unip case}
\delta(s)
:=
\left(
\begin{bsmallmatrix}
e^{\frac{is}{i+1}}&0\\
0&e^{-\frac{s}{i+1}}I_i
\end{bsmallmatrix},
\begin{bsmallmatrix}
e^{\frac{js}{j+1}}&0\\
0&e^{-\frac{s}{j+1}}I_j
\end{bsmallmatrix}
\right),
\quad s>0.
\end{equation}

Before proving the next lemma, recall that
\begin{equation}\label{eq:ux-on-lightlike-basis}
u(\mathbf x)v^+=v^+,\quad
u(\mathbf x)\mathbf e_q=x_qv^+ + \mathbf e_q
\quad(2\le q\le n).
\end{equation}
Also, using the explicit formula for $u(\mathbf x)$, one checks that
\begin{equation}\label{eq:uxstar-on-lightlike-basis}
u(\mathbf x)^*v^-=v^-,
\quad
u(\mathbf x)^*\mathbf e_q=x_qv^-+\mathbf e_q
\quad(2\le q\le n).
\end{equation}

\begin{lemma}\label{lem:description of AU action}
Let $([\Lambda_1],[\Lambda_2])\in X_{\Pinfty}$ and let
$y\in X_{i+1}\times X_{j+1}$ such that \(\rho^\infty(y)=([\Lambda_1],[\Lambda_2]).\) Then, for every $\mathbf x=(\boldsymbol a,\boldsymbol b)\in\R^i\times\R^j$ and every $s>0$,
\begin{equation}\label{eq:AU-action-via-rho-infty}
    a(s)u(\mathbf x)\cdot([\Lambda_1],[\Lambda_2])
    =
    \rho^\infty\bigl((\delta(s)\upsilon(\mathbf x))\cdot y\bigr).
\end{equation}
\end{lemma}

\begin{proof}
Since $a(s)u(\mathbf x)$ preserves $P_1^\infty$, and
$(a(s)u(\mathbf x))^*$ preserves $P_2^\infty$, it is enough to compute the
representing matrices of
\[
a(s)u(\mathbf x)\big|_{P_1^\infty}
\quad\text{and}\quad
(a(s)u(\mathbf x))^*\big|_{P_2^\infty}
\]
in the bases $B_1^\infty$ and $B_2^\infty$.

Using \eqref{eq:a(s) eigenval eqn} and \eqref{eq:ux-on-lightlike-basis}, we get \(a(s)u(\mathbf x)v^+=e^sv^+,\) and, 
\[
a(s)u(\mathbf x)\mathbf e_q
=
a(s)(x_qv^+ + \mathbf e_q)
=
e^sx_qv^+ + \mathbf e_q,\quad \textup{for } 2\le q\le i+1.
\]
Therefore the representing matrix of $a(s)u(\mathbf x)\big|_{P_1^\infty}$
in the basis $B_1^\infty$ is
\[
\begin{bmatrix}
e^s& e^s\boldsymbol a\\
0&I_i
\end{bmatrix}
=
\begin{bmatrix}
e^s&0\\
0&I_i
\end{bmatrix}
\begin{bmatrix}
1&\boldsymbol a\\
0&I_i
\end{bmatrix}
=
e^{\frac{s}{i+1}}
\begin{bsmallmatrix}
e^{\frac{is}{i+1}}&0\\
0&e^{-\frac{s}{i+1}}I_i
\end{bsmallmatrix}
\begin{bmatrix}
1&\boldsymbol a\\
0&I_i
\end{bmatrix}.
\]

Similarly, using \eqref{eq:uxstar-on-lightlike-basis} and the fact that
$a(s)^*v^-=e^sv^-$ and $a(s)^*\mathbf e_q=\mathbf e_q$ for $2\le q\le n$, we get
\[
(a(s)u(\mathbf x))^*v^-=e^sv^-,
\]
and, for $i+2\le q\le n$,
\[
(a(s)u(\mathbf x))^*\mathbf e_q
=
a(s)^*(x_qv^-+\mathbf e_q)
=
e^sx_qv^-+\mathbf e_q.
\]
Therefore the representing matrix of $(a(s)u(\mathbf x))^*\big|_{P_2^\infty}$
in the basis $B_2^\infty$ is
\[
\begin{bmatrix}
e^s& e^s\boldsymbol b\\
0&I_j
\end{bmatrix}
=
\begin{bmatrix}
e^s&0\\
0&I_j
\end{bmatrix}
\begin{bmatrix}
1&\boldsymbol b\\
0&I_j
\end{bmatrix}
=
e^{\frac{s}{j+1}}
\begin{bsmallmatrix}
e^{\frac{js}{j+1}}&0\\
0&e^{-\frac{s}{j+1}}I_j
\end{bsmallmatrix}
\begin{bmatrix}
1&\boldsymbol b\\
0&I_j
\end{bmatrix}.
\]

Since scalar homotheties are trivial in the projective actions on
$X_{i+1}$ and $X_{j+1}$, the extra scalar factors
$e^{\frac{s}{i+1}}$ and $e^{\frac{s}{j+1}}$ are absorbed.
This proves \eqref{eq:AU-action-via-rho-infty}.
\end{proof}

For $x\in G/\Ga$, $y=(y_1,y_2)\in X_{i+1}\times X_{j+1}$, $k\in K$, and $T>0$,
define a measure on $G/\Ga\times X_{i+1}\times X_{j+1}$ by
\begin{equation}\label{eq:def of msr on XiXj U case}
   \nu_{x,y,k,T}(F)
   :=
   \frac{1}{\vol(G_T)}
   \int_{\R^{n-1}}\int_{b_T(k,u(\mathbf x))}
   F\bigl(a(s)u(\mathbf x)x,\delta(s)\upsilon(\mathbf x)y\bigr)e^{(n-1)s}\,ds\,d\mathbf x,
\end{equation}
where $b_T(k,u(\mathbf x))$ is given in \eqref{eq:def of bTkh}.

If $n=2$, then either $i=0$ or $j=0$, so one of $X_{i+1},X_{j+1}$ is a singleton.
In that case, we identify $X_{i+1}\times X_{j+1}$ with $X_2$, and
\eqref{eq:def of msr on XiXj U case} becomes
\begin{equation}\label{eq:def of msr on XiXj U case n2 case}
   \nu_{x,y,k,T}(F)
   :=
   \frac{1}{\vol(G_T)}
   \int_{\R}\int_{b_T(k,u(t))}
   F\bigl(a(s)u(t)x,\delta(s)\upsilon(t)y\bigr)e^s\,ds\,dt,
\end{equation}
where \(
\upsilon(t)=
\begin{bsmallmatrix}
1&t\\
0&1
\end{bsmallmatrix}.
\)

We consider the map $\Psi^\infty:G/\Ga\times X_{i+1}\times X_{j+1}\to G/\Ga\times X_{\Pinfty}$,
\begin{align}
    \Psi^\infty(x,y)&:=(x,\rho^\infty(y))\label{eq:def of Psiinfty},\;x\in G/\Ga,\;y\in X_{i+1}\times X_{j+1},
\end{align}
and, for $k\in K$, we define
\begin{equation}\label{eq:def kPsiinfty}
    k\Psi^\infty(x,y):=(kx,k\cdot\rho^\infty(y)),\;x\in G/\Ga,\;y\in X_{i+1}\times X_{j+1},
\end{equation}
where in the second coordinate $k$ acts as in \eqref{eq:def of action on X}.

We have the following statement, analogous to
Lemma~\ref{lem:replacing by averages in GmodD times expading in XiXj}.
It is an immediate consequence of Lemma~\ref{lem:description of AU action}.

\begin{lemma}\label{lem:replacing by averages in GmodD times expading in XiXj, unipotent case}
Fix $x_0\in G/\Ga$, let $([\Lambda_1],[\Lambda_2])\in X_{\Pinfty}$, and let $y_0\in X_{i+1}\times X_{j+1}$ such that \(\rho^\infty(y_0)=([\Lambda_1],[\Lambda_2]).\) Suppose that for every $k\in K$,
\begin{equation}\label{eq:requirement for equid in degnrt case}
       \underset{T\to\infty}{\weaklim\lim}\ \nu_{x_0,y_0,k,T}=\nu_{k,\infty}.
\end{equation}

Then, for all $F\in C_c(G/\Ga\times X)$,
\begin{align*}
    &\lim_{T\to\infty}
    \frac{1}{\vol(G_T)}
    \int_{\R^{n-1}}\int_{b_T(k,u(\mathbf x))}
    F\bigl(ka(s)u(\mathbf x)x_0,ka(s)u(\mathbf x)\cdot([\Lambda_1],[\Lambda_2])\bigr)
    e^{(n-1)s}\,ds\,d\mathbf x
    \\&=
    (k\Psi^\infty)_*\nu_{k,\infty}(F).
\end{align*}
\end{lemma}

\section{Equidistribution statements}\label{sec:equid results}

The goal of this section is to prove the equidistribution statements needed to establish the convergence of \eqref{eq:requirement for equid in nondeg case} and \eqref{eq:requirement for equid in degnrt case}.

We divide the argument into two cases, \(n \geq 3\) and the \(n=2\) case.

For \(n\geq 3\), we will show that every weak-* limit has Haar marginals in each factor and is invariant under a nontrivial unipotent element whose projection to each factor acts ergodically. We will then use the fact that, in this setting, there are no nontrivial algebraic joinings. It will follow that the limiting measure is the full product measure.

The case \(n=2\) is more delicate because of the exceptional isogeny
\(\SL(2,\R)\to \SO(2,1)^\circ\), which allows additional algebraic joinings and therefore requires a separate analysis.

\subsection{The \(n\geq 3\)-case}

Let \(n \geq 3\), and fix nonnegative integers \(i,j\) such that \(i+j=n-1\). Set \(\G := \SO(n,1)^\circ\times \SL(i+1,\R) \times \SL(j+1,\R),\) and consider the subgroup
\[
\mathcal H:=\left\{ (h,\hh,\vh^*):h=\begin{bsmallmatrix}
    \hh & \\
    && \vh 
\end{bsmallmatrix}\in H_{i,j}\right\}\subseteq \G,
\]
where \(H_{i,j}\cong \SO(i+1)\times\SO(j,1)^\circ\) is as in \eqref{eq:DefHij}. We also define
\[
\U:=\{(u(\mathbf{x}),\ups(\mathbf{x})):\mathbf{x}\in\R^{n-1}\},
\]
where \(u(\mathbf{x})\) is given by \eqref{eq:u_x expression} and \(\ups(\mathbf{x})\in \SL(i+1,\R)\times \SL(j+1,\R)\) is defined in \eqref{eq:defining upsilon}.

Next, we consider the diagonalizable one-parameter subgroups of \(\G\) given by
\begin{equation}\label{eq:DefalH}
    \al_{\mathcal H}(t):=(a(t),d_1(t),d_2(t)), \; t\in\R,
\end{equation}
where \(a(t), d_1(t), d_2(t)\) are as in \eqref{eq:expression for a(s)}, \eqref{eq:defining d}, and
\begin{equation}\label{eq:DefalU}
    \al_{\U}(t):=(a(t),\de(t)), \; t\in\R,
\end{equation}
where \(\de(t)\in \SL(i+1,\R)\times \SL(j+1,\R)\) is as in \eqref{eq:diagonal pushing in unip case}.

\begin{proposition}\label{prop:HigherDimensionalExpendingPieces}
Let \(\mathcal L\) denote either \(\mathcal H\) or \(\U\), and let \(\al_{\mathcal L}\) be the corresponding diagonalizable subgroup defined above. Let \(\tilde\Gamma:=\Gamma\times\SL(i+1,\Z)\times\SL(j+1,\Z)\), where \(\Gamma\subseteq\SO(n,1)^\circ\) is a lattice, and fix \(x\in \G/\tilde\Gamma\). Let \(w\in C_c(\mathcal L)\) be nonnegative, and define
\begin{equation*}
    \mu^{\mathcal L}_s(F):=\int_{\mathcal L}F(\alpha_{\mathcal L}(s)lx)\,w(l)\,dl,
    \; F\in C_c(\G/\tilde\Gamma),
\end{equation*}
where \(dl\) denotes a Haar measure on \(\mathcal L\). Then
\[
\lim_{s\to\infty}\mu^{\mathcal L}_s(F)
=
\mu_{\G/\tilde\Gamma}(F)\int_{\mathcal L}w(l)\,dl,
\]
for all \(F\in C_c(\G/\tilde\Gamma)\).
\end{proposition}

\subsubsection{Generalities}

To prove Proposition~\ref{prop:HigherDimensionalExpendingPieces}, we will use the following general facts about equidistribution on homogeneous spaces.

Throughout this subsection, let \(\mathcal G\subseteq \mathrm{SL}(m,\mathbb R)\) be a connected semisimple Lie group, let \(\Gamma\subseteq \mathcal G\) be a lattice, and let \(\mathcal H\subseteq \mathcal G\) be a closed connected subgroup. We denote by \(\mathfrak g\) and \(\mathfrak h\) their Lie algebras.

Let \(\{\alpha(t)\}_{t\in\mathbb R}\subseteq \mathcal G\) be a nontrivial one-parameter subgroup consisting of diagonalizable elements. Then \(\mathfrak g\) admits a decomposition \(\mathfrak g=\mathfrak g^{-}\oplus \mathfrak g^{0}\oplus \mathfrak g^{+}\),  where \(\mathfrak g^{0}\) is the fixed-point subspace of \(\Ad(\alpha(t))\), and \(\mathfrak g^{+}\) and \(\mathfrak g^{-}\) are the sums of the eigenspaces corresponding, respectively, to the positive and negative weights of the one-parameter group \(\Ad(\alpha(t))\). We write \(\pi^{+}:\mathfrak g\to\mathfrak g^{+}\) for the canonical projection associated with this direct sum decomposition.

\begin{theorem}\label{thm:margulis banana trick}
Assume further that \(\mathcal G\) is simple and that \(\pi^+(\mathfrak h)=\mathfrak g^+\). Let \(x_0\in \mathcal G/\Gamma\), and let \(w\in C_c(\mathcal H)\) be nonnegative. Then
\begin{equation}\label{eq:margulis banana equidistribution}
     \lim_{s\to\infty}\int_{\mathcal H}f(\alpha(s)hx_0)\,w(h)\,dh
     =
     \mu_{\mathcal G/\Gamma}(f)\mu_{\mathcal H}(w)
\end{equation}
for all \(f\in C_c(\mathcal G/\Gamma)\). Here, \(\mu_{\mathcal G/\Gamma}\) denotes the \(\mathcal G\)-invariant probability measure on \(\mathcal G/\Gamma\), and \(\mu_{\mathcal H}\) denotes a Haar measure on \(\mathcal H\).
\end{theorem}

\begin{proof}
The result follows from mixing of the \(\{\alpha(t)\}\)-action with respect to \(\mu_{\mathcal G/\Gamma}\), together with the classical ``banana trick'' due to Margulis; see \cite{EskinMcMullen}.
\end{proof}

To prove Proposition~\ref{prop:HigherDimensionalExpendingPieces}, we will use the above theorem to obtain equidistribution in each simple factor. We will supplement this with the following lemma, which ensures that every limiting measure is invariant under a suitable unipotent element.

\begin{lemma}\label{lem: unipotent invariance general}
Let \(w\in C_c(\mathcal H)\) be nonnegative, and consider the measures
\begin{equation}\label{eq:measures of s, for unip invar}
     \mu_s(f):=\int_{\mathcal H}f(\alpha(s)hx_0)\,w(h)\,dh,
     \quad f\in C_c(\mathcal G/\Gamma).
\end{equation}
Suppose that \(\boldsymbol{v}\in \mathfrak g^+\) is an eigenvector for \(\{\Ad(\alpha(t))\}_{t\in\R}\) and that \(\boldsymbol{v}\in \pi^+(\mathfrak h)\). Then every weak-* limit of \(\mu_s\) as \(s\to\infty\) is invariant under left translation by \(\exp(\boldsymbol{v})\).
\end{lemma}

\begin{proof}
Suppose that \(\Ad(\alpha(t))\boldsymbol{v}=e^{\chi(t)}\boldsymbol{v},\) where \(\chi(t)\to\infty\) as \(t\to\infty\), and assume that there exists \(\boldsymbol{h}\in \mathfrak h\) such that \(\boldsymbol{v}=\pi^+(\boldsymbol{h})\). Since \(\ker(\pi^+)=\mathfrak g^0\oplus\mathfrak g^{-}\), there exists \(\boldsymbol{u}\in \mathfrak g^0\oplus\mathfrak g^{-}\) such that \(\boldsymbol{h}=\boldsymbol{v}+\boldsymbol{u}.\)

Denote \(\boldsymbol{h}_t:=e^{-\chi(t)}\boldsymbol{h}.\) Then
\[
\Ad(\alpha(t))\boldsymbol{h}_t=\boldsymbol{v}+O(e^{-\chi(t)}).
\]

Now suppose that \(\mu_{s_i}\to \nu\) in the weak-* topology. For \(f\in C_c(\mathcal G/\Gamma)\), using the left invariance of Haar measure on \(\mathcal H\), we obtain
\begin{align*}
\nu(f)
&=\lim_{i\to\infty}\int_{\mathcal H}f(\alpha(s_i)hx_0)\,w(h)\,dh\\
&=\lim_{i\to\infty}\int_{\mathcal H}
f(\alpha(s_i)\exp(\boldsymbol{h}_{s_i})hx_0)\,
w(\exp(\boldsymbol{h}_{s_i})h)\,dh\\
&=\lim_{i\to\infty}\int_{\mathcal H}
f(\exp(\Ad(\alpha(s_i))\boldsymbol{h}_{s_i})\alpha(s_i)hx_0)\,
w(\exp(\boldsymbol{h}_{s_i})h)\,dh\\
&=\lim_{i\to\infty}\int_{\mathcal H}
f(\exp(\boldsymbol{v}+O(e^{-\chi(s_i)}))\alpha(s_i)hx_0)\,
w(\exp(O(e^{-\chi(s_i)}))h)\,dh\\
&=\lim_{i\to\infty}\int_{\mathcal H}
f(\exp(\boldsymbol{v})\alpha(s_i)hx_0)\,w(h)\,dh\\
&=(\exp(\boldsymbol{v})_*\nu)(f),
\end{align*}
where in the second-to-last line we use the uniform continuity of \(f\) and \(w\). This proves the claimed invariance.
\end{proof}

The following proposition shows that, once the limiting measure has Haar marginals and is invariant under a suitable unipotent element, full product equidistribution follows provided there are no proper algebraic joinings.

\begin{proposition}\label{prop:if projections are full and no iso then equid is full}
Suppose that \(\G=\G_1\times\cdots\times \G_m\) and \(\Gamma=\Gamma_1 \times \cdots \times \Gamma_m\), where each \(\G_l\) is a connected Lie group and each \(\Gamma_l\) is a lattice in \(\G_l\). For each \(l\), let \(\pi_l:\G\to\G_l\) and \(\bar\pi_l:\G/\Gamma\to \G_l/\Gamma_l\) denote the natural projections. Suppose that \(u\in\G\) is an \(\Ad\)-unipotent element such that \(\pi_l(u)\) acts ergodically with respect to the \(\G_l\)-invariant probability measure \(\mu_{\G_l/\Gamma_l}\) on \(\G_l/\Gamma_l\).

Let \(\mu\) be a \(u\)-invariant Borel probability measure on \(\G/\Gamma\) such that
\[
(\bar\pi_l)_*(\mu)=\mu_{\G_l/\Gamma_l}
\]
for each \(l\). Assume further that there is no proper subgroup \(L\subseteq \G\) such that \(L\Gamma/\Gamma\) is closed and \(\pi_l(L)=\G_l\) for every \(l\). Then \(\mu\) is \(\G\)-invariant.
\end{proposition}

\begin{proof}
Consider the ergodic decomposition of \(\mu\) with respect to the \(u\)-action:
\[
\mu=\int_Y \mu_y\,d\bar\mu(y),
\]
where each \(\mu_y\) is a \(u\)-ergodic \(u\)-invariant Borel probability measure on \(\G/\Gamma\), and \((Y,\bar\mu)\) is a Borel probability space on which \(u\) acts trivially.

Fix \(1\le l\le m\). Since \(\bar\pi_l\) is equivariant with respect to the action of \(u\) on \(\G/\Gamma\) and the action of \(\pi_l(u)\) on \(\G_l/\Gamma_l\), the measure \((\bar\pi_l)_*(\mu_y)\) is \(\pi_l(u)\)-ergodic and \(\pi_l(u)\)-invariant for \(\bar\mu\)-almost every \(y\in Y\). Hence
\[
(\bar\pi_l)_*\mu=\int_Y (\bar\pi_l)_*(\mu_y)\,d\bar\mu(y)
\]
is an ergodic decomposition of \((\bar\pi_l)_*\mu\).

By hypothesis, \((\bar\pi_l)_*\mu=\mu_{\G_l/\Gamma_l}\), and this measure is \(\pi_l(u)\)-ergodic. Therefore, for \(\bar\mu\)-almost every \(y\in Y\),
\begin{equation}\label{eq:muy-Gl}
(\bar\pi_l)_*(\mu_y)=\mu_{\G_l/\Gamma_l}.
\end{equation}

Now let \(y\in Y\) be such that \(\mu_y\) is \(u\)-ergodic and \eqref{eq:muy-Gl} holds for every \(1\le l\le m\). By Ratner's measure-classification theorem, there exists a closed connected subgroup \(L\subseteq \G\) such that \(\mu_y\) is \(L\)-invariant and supported on a closed orbit \(Lg\Gamma/\Gamma\) for some \(g\in \G\). For each \(l\), the projection
\[
\pi_l(Lg\Gamma/\Gamma)=\pi_l(L)\pi_l(g)\Gamma_l/\Gamma_l
\]
has full \(\mu_{\G_l/\Gamma_l}\)-measure by \eqref{eq:muy-Gl}. This is only possible if \(\pi_l(L)=\G_l\), since otherwise \(\pi_l(L)\) is a proper lower-dimensional subgroup and the projected orbit has zero Haar measure.

Finally, let \(L_1=g^{-1}Lg\). Then \(L_1\Gamma/\Gamma\) is closed, and \(\pi_l(L_1)=\G_l\) for every \(l\). By assumption, this forces \(L_1=\G\), and hence \(L=\G\). Therefore \(\mu_y\) is the \(\G\)-invariant probability measure \(\mu_{\G/\Gamma}\).

Thus \(\mu_y=\mu_{\G/\Gamma}\) for \(\bar\mu\)-almost every \(y\), and consequently \(\mu=\mu_{\G/\Gamma}\).
\end{proof}
\subsubsection{Proving Proposition \ref{prop:HigherDimensionalExpendingPieces}}

Let \(n\geq 3\) and let \(i,j\) be nonnegative integers such that \(i+j=n-1\). Set
\begin{equation}\label{eq:SimpleGpsGi}
    \G_1:=\SO(n,1)^\circ,\quad \G_2:=\SL(i+1,\R),\quad \G_3:=\SL(j+1,\R),
\end{equation}
and let \(\mathfrak g_l\) denote the corresponding Lie algebras.

\medskip
\noindent\textbf{The subgroup \(\mathcal H\).} Let
\begin{equation}\label{eq:DefAli}
    \al_1(t):=a(t),\quad \al_2(t):=d_1(t),\quad \al_3(t):=d_2(t),
\end{equation}
where \(a(t),d_1(t),d_2(t)\) are as in \eqref{eq:expression for a(s)} and \eqref{eq:defining d}. Let \(\mathfrak g_l=\mathfrak g_l^-\oplus \mathfrak g_l^0\oplus \mathfrak g_l^+\) be the corresponding weight-space decomposition for \(\Ad(\al_l(t))\), and let \(\pi_l^+:\mathfrak g_l\to \mathfrak g_l^+\) be the associated projection. Then
\begin{equation}\label{eq:positive_spaces_H}
    \mathfrak g_1^+=\mathfrak u,\quad
    \mathfrak g_2^+=\left\{\begin{bsmallmatrix}
        0&\boldsymbol{x}\\
        0&0
    \end{bsmallmatrix}:\boldsymbol{x}\in\R^i\right\},\quad
    \mathfrak g_3^+=\left\{\begin{bsmallmatrix}
        0&0\\
        \boldsymbol{y}&0
    \end{bsmallmatrix}:\boldsymbol{y}\in\R^j\right\},
\end{equation}
where \(\mathfrak u=\Lie(U)\).

Now denote
\begin{equation}\label{eq:SimpleGpsHi}
    \HH_1:=H_{i,j},\quad \HH_2:=\SO(i+1),\quad \HH_3:=\SO(j,1)^\circ,
\end{equation}
and let \(\mathfrak h_l:=\Lie(\HH_l)\).

\begin{lemma}\label{lem:H-factorwise-surjectivity}
For each \(l=1,2,3\) we have \(\pi_l^+(\mathfrak h_l)=\mathfrak g_l^+.\)
\end{lemma}

\begin{proof}
For \(\boldsymbol{a}\in\R^i\) and \(\boldsymbol{b}\in\R^j\), consider
\begin{align}
    \boldsymbol{v}_1:=\begin{bsmallmatrix}
        0&(\boldsymbol{a},\boldsymbol{b})&0\\
        -(\boldsymbol{a},\boldsymbol{b})^{\tr}&\mathbf{0}&(\boldsymbol{a},\boldsymbol{b})^{\tr}\\
        0&(\boldsymbol{a},\boldsymbol{b})&0
    \end{bsmallmatrix},\;\; \boldsymbol{v}_2:=\begin{bmatrix}
        0&2\boldsymbol{a}\\
        0&0
    \end{bmatrix}\;\; \boldsymbol{v}_3:=\begin{bmatrix}
        0&0\\
        -2\boldsymbol{b}&0
    \end{bmatrix}. \label{eq:PositveWeightVecsi}
\end{align}
Let
\begin{align}
    \boldsymbol{h}_1:=2\begin{bsmallmatrix}
        0&(\boldsymbol{a},\mathbf{0})&0\\
        -(\boldsymbol{a},\mathbf{0})^{\tr}&\mathbf{0}&(\mathbf{0},\boldsymbol{b})^{\tr}\\
        0&(\mathbf{0},\boldsymbol{b})&0
    \end{bsmallmatrix},\;\;
    \boldsymbol{h}_2:=2\begin{bmatrix}
        0&\boldsymbol{a}\\
        -\boldsymbol{a}^{\tr}&0
    \end{bmatrix},\;\;
    \boldsymbol{h}_3&:=2\begin{bmatrix}
        0&-\boldsymbol{b}^{\tr}\\
        -\boldsymbol{b}&0
    \end{bmatrix}.\label{eq:VecsInhi}
\end{align}
Then \(\boldsymbol{h}_l\in\mathfrak h_l\) for \(l=1,2,3\), and
\[
\Ad(\al_l(t))(\boldsymbol{h}_l-\boldsymbol{v}_l)=e^{-t}(\boldsymbol{h}_l-\boldsymbol{v}_l),
\quad l=1,2,3.
\]
Hence \(\boldsymbol{h}_l-\boldsymbol{v}_l\in\ker(\pi_l^+)\), so \(\pi_l^+(\boldsymbol{h}_l)=\boldsymbol{v}_l,\; l=1,2,3.\)

As \(\boldsymbol{a}\) and \(\boldsymbol{b}\) vary, the vectors \(\boldsymbol{v}_1,\boldsymbol{v}_2,\boldsymbol{v}_3\) run through all of \(\mathfrak g_1^+,\mathfrak g_2^+,\mathfrak g_3^+\). This proves the lemma.
\end{proof}

Let
\[
\HH:=\left\{(h,\hh,\vh^*):h=\begin{bsmallmatrix}
    \hh & \\
    && \vh
\end{bsmallmatrix}\in H_{i,j}\right\}\subseteq \HH_1\times\HH_2\times\HH_3,
\]
and let \(\mathfrak h:=\Lie(\HH)\). Set
\[
\al(t):=(\al_1(t),\al_2(t),\al_3(t)),
\]
and let \(\pi^+:\mathfrak g\to \mathfrak g^+\) be the positive-weight projection for \(\Ad(\al(t))\).

\begin{lemma}\label{lem:lemma implying unip invariance for Hij averages}
For \(\boldsymbol{a}\in\R^i\) and \(\boldsymbol{b}\in\R^j\), set
\[
\boldsymbol{v}^+_{\boldsymbol{a},\boldsymbol{b}}
=
\left(
\begin{bsmallmatrix}
        0&(\boldsymbol{a},\boldsymbol{b})&0\\
        -(\boldsymbol{a},\boldsymbol{b})^{\tr}&\mathbf{0}&(\boldsymbol{a},\boldsymbol{b})^{\tr}\\
        0&(\boldsymbol{a},\boldsymbol{b})&0
\end{bsmallmatrix},
\begin{bmatrix}
        0&2\boldsymbol{a}\\
        0&0
\end{bmatrix},
\begin{bmatrix}
        0&0\\
        -2\boldsymbol{b}&0
\end{bmatrix}
\right).
\]
Then
\[
\Ad(\al(s))\boldsymbol{v}^+_{\boldsymbol{a},\boldsymbol{b}}=e^s\boldsymbol{v}^+_{\boldsymbol{a},\boldsymbol{b}},
\quad
\boldsymbol{v}^+_{\boldsymbol{a},\boldsymbol{b}}\in \pi^+(\mathfrak h).
\]
\end{lemma}

\begin{proof}
Let \(\boldsymbol{v}_1,\boldsymbol{v}_2,\boldsymbol{v}_3\) and \(\boldsymbol{h}_1,\boldsymbol{h}_2,\boldsymbol{h}_3\) be as in
\eqref{eq:PositveWeightVecsi} and \eqref{eq:VecsInhi}, and set
\[
\boldsymbol{h}:=(\boldsymbol{h}_1,\boldsymbol{h}_2,\boldsymbol{h}_3)\in \mathfrak h.
\]
By Lemma~\ref{lem:H-factorwise-surjectivity}, \(\pi_l^+(\boldsymbol{h}_l)=\boldsymbol{v}_l,\; l=1,2,3\).

Hence 
\[
\pi^+(\boldsymbol{h})
=
\bigl(\pi_1^+(\boldsymbol{h}_1),\pi_2^+(\boldsymbol{h}_2),\pi_3^+(\boldsymbol{h}_3)\bigr)
=
(\boldsymbol{v}_1,\boldsymbol{v}_2,\boldsymbol{v}_3)
=
\boldsymbol{v}^+_{\boldsymbol{a},\boldsymbol{b}}.
\]
The eigenvalue relation is immediate from the definitions.
\end{proof}

\medskip
\noindent\textbf{The subgroup \(\U\).} Write
\[
\de(t)=\bigl(\de_1(t),\de_2(t)\bigr)\in \SL(i+1,\R)\times\SL(j+1,\R),
\]
where \(\de(t)\) is as in \eqref{eq:diagonal pushing in unip case}, and set
\begin{equation}\label{eq:DefBetai}
    \be_1(t):=a(t),\quad \be_2(t):=\de_1(t),\quad \be_3(t):=\de_2(t).
\end{equation}
Let \(\mathfrak g_l=\mathfrak g_l^-\oplus\mathfrak g_l^0\oplus\mathfrak g_l^+\) be the weight decomposition for \(\Ad(\be_l(t))\), and let \(\varpi_l^+:\mathfrak g_l\to\mathfrak g_l^+\) be the associated projection. Then
\begin{equation}\label{eq:positive_spaces_U}
\mathfrak g_1^+=\mathfrak u,\quad
\mathfrak g_2^+=\left\{
\begin{bmatrix}
0&\boldsymbol{a}\\
0&0
\end{bmatrix}:\boldsymbol{a}\in\R^i\right\},\quad
\mathfrak g_3^+=\left\{
\begin{bmatrix}
0&\boldsymbol{b}\\
0&0
\end{bmatrix}:\boldsymbol{b}\in\R^j\right\}.
\end{equation}

Let
\[
\U_1:=U\subseteq \G_1,\quad
\U_2:=\left\{\begin{bmatrix}
1&\boldsymbol{a}\\
0&I_i
\end{bmatrix}:\boldsymbol{a}\in\R^i\right\}\subseteq \G_2,\quad
\U_3:=\left\{\begin{bmatrix}
1&\boldsymbol{b}\\
0&I_j
\end{bmatrix}:\boldsymbol{b}\in\R^j\right\}\subseteq \G_3,
\]
and let \(\mathfrak u_l:=\Lie(\U_l)\).

\begin{lemma}\label{lem:U-factorwise-surjectivity}
For each \(l=1,2,3\) we have \(\varpi_l^+(\mathfrak u_l)=\mathfrak g_l^+.\)
\end{lemma}

\begin{proof}
For \(l=1\), we have \(\mathfrak u_1=\mathfrak u=\mathfrak g_1^+\). For \(l=2,3\), the groups \(\U_l\) are precisely the expanding horospherical subgroups of \(\be_l(t)\), so \(\mathfrak u_l=\mathfrak g_l^+\). The claim follows.
\end{proof}

Let \(\mathfrak u_\Delta:=\Lie(\U)\subseteq \mathfrak g_1\oplus\mathfrak g_2\oplus\mathfrak g_3\), and write \(\boldsymbol{x}=(\boldsymbol{a},\boldsymbol{b})\in \R^i\times \R^j.\) Define
\[
\boldsymbol{w}^+_{\boldsymbol{a},\boldsymbol{b}}
:=
\left(
\begin{bsmallmatrix}
0&(\boldsymbol{a},\boldsymbol{b})&0\\
-(\boldsymbol{a},\boldsymbol{b})^{\tr}&\mathbf 0&(\boldsymbol{a},\boldsymbol{b})^{\tr}\\
0&(\boldsymbol{a},\boldsymbol{b})&0
\end{bsmallmatrix},
\begin{bmatrix}
0&\boldsymbol{a}\\
0&0
\end{bmatrix},
\begin{bmatrix}
0&\boldsymbol{b}\\
0&0
\end{bmatrix}
\right).
\]

\begin{lemma}\label{lem:U-product-unipotent}
For every \(\boldsymbol{a}\in\R^i\) and \(\boldsymbol{b}\in\R^j\),
\[
\Ad\bigl((a(s),\de(s))\bigr)\boldsymbol{w}^+_{\boldsymbol{a},\boldsymbol{b}}
=
e^s\boldsymbol{w}^+_{\boldsymbol{a},\boldsymbol{b}},
\quad
\boldsymbol{w}^+_{\boldsymbol{a},\boldsymbol{b}}\in \pi^+(\mathfrak u_\Delta),
\]
where \(\pi^+\) denotes the positive-weight projection for \(\Ad((a(s),\de(s)))\).
\end{lemma}

\begin{proof}
The vector \(\boldsymbol{w}^+_{\boldsymbol{a},\boldsymbol{b}}\) is tangent at the identity to the one-parameter subgroup
\[
t\mapsto \bigl(u(t(\boldsymbol{a},\boldsymbol{b})),\ups(t(\boldsymbol{a},\boldsymbol{b}))\bigr)\subseteq \U,
\]
hence \(\boldsymbol{w}^+_{\boldsymbol{a},\boldsymbol{b}}\in \mathfrak u_\Delta\). By \eqref{eq:positive_spaces_U},
\(\mathfrak u_\Delta\subseteq \mathfrak g^+\), and therefore \(\pi^+(\boldsymbol{w}^+_{\boldsymbol{a},\boldsymbol{b}})=\boldsymbol{w}^+_{\boldsymbol{a},\boldsymbol{b}}.\)

The eigenvalue relation follows from the definitions of \(a(s)\), \(\de_1(s)\), and \(\de_2(s)\).
\end{proof}

\begin{proof}[Proof of Proposition \ref{prop:HigherDimensionalExpendingPieces} for \(n\geq 3\)]
Let $\Ga\subseteq\SO(n,1)^\circ$ be a lattice, and set
\[
\Gamma_1:=\Gamma,\quad
\Gamma_2:=\SL(i+1,\Z),\quad
\Gamma_3:=\SL(j+1,\Z),
\]
so that \(\tilde\Gamma=\Gamma_1\times\Gamma_2\times\Gamma_3.\)

We first verify that there are no proper algebraic joinings. Set
\[
R:=\PGL(i+1,\R)\times\PGL(j+1,\R),
\]
and let \(\pi_{\mathrm{left}}:\G_1\times R\to \G_1\) and \(\pi_{\mathrm{right}}:\G_1\times R\to R\) be the natural projections.

We claim that if \(L\subseteq \G_1\times R\) is a closed connected subgroup such that both \(\pi_{\mathrm{left}}|_L\) and \(\pi_{\mathrm{right}}|_L\) are onto,  then \(L=\G_1\times R\).

Indeed, let \(\mathfrak l\subseteq \mathfrak g_1\oplus \mathfrak r\) be the Lie algebra of \(L\), where
\[
\mathfrak g_1=\mathfrak{so}(n,1),
\quad
\mathfrak r
=
\mathfrak{pgl}(i+1,\R)\oplus \mathfrak{pgl}(j+1,\R)
\cong
\mathfrak{sl}_{i+1}(\R)\oplus \mathfrak{sl}_{j+1}(\R).
\]
Since both projections are surjective, the differentials \(d\pi_{\mathrm{left}}|_{\mathfrak l}:\mathfrak l\to \mathfrak g_1\) and \(d\pi_{\mathrm{right}}|_{\mathfrak l}:\mathfrak l\to \mathfrak r\) are surjective. Consider
\[
\mathfrak n:=\ker(d\pi_{\mathrm{right}}|_{\mathfrak l})\subseteq \mathfrak g_1\oplus\{0\}.
\]
Identifying \(\mathfrak n\) with its projection to \(\mathfrak g_1\), we see that \(\mathfrak n\) is an ideal in \(\mathfrak g_1\). Since \(\mathfrak g_1=\mathfrak{so}(n,1)\) is simple for \(n\geq 3\), either \(\mathfrak n=0\) or \(\mathfrak n=\mathfrak g_1\oplus\{0\}\).

If \(\mathfrak n=\mathfrak g_1\oplus\{0\}\), then \(\mathfrak g_1\oplus\{0\}\subseteq \mathfrak l\), and surjectivity of
\(d\pi_{\mathrm{right}}\) gives \(\mathfrak l=\mathfrak g_1\oplus\mathfrak r\), hence \(L=\G_1\times R\).

Assume now that \(\mathfrak n=0\). Then \(d\pi_{\mathrm{right}}|_{\mathfrak l}\) is an isomorphism, so \(\mathfrak l\) is the graph of a surjective Lie algebra homomorphism \(\varphi:\mathfrak r\to \mathfrak g_1\). Writing
\[
\mathfrak r=\mathfrak r_1\oplus\mathfrak r_2,
\quad
\mathfrak r_1\cong\mathfrak{sl}_{i+1}(\R),
\quad
\mathfrak r_2\cong\mathfrak{sl}_{j+1}(\R),
\]
we have \([\mathfrak r_1,\mathfrak r_2]=0\), so \(\varphi(\mathfrak r_1)\) and \(\varphi(\mathfrak r_2)\) are commuting ideals in \(\mathfrak g_1\). Since \(\mathfrak g_1\) is simple, one of them is \(0\) and the other is all of \(\mathfrak g_1\). Thus for some \(m\in\{i+1,j+1\}\) there exists a surjective Lie algebra homomorphism \(\mathfrak{sl}_m(\R)\twoheadrightarrow \mathfrak{so}(n,1).\)

As both Lie algebras are simple, this must be an isomorphism. But \(\rank_\R(\mathfrak{sl}_m(\R))=m-1\) and \(\rank_\R(\mathfrak{so}(n,1))=1\), so this is possible only when \(m=2\), which would force
\(\mathfrak{sl}_2(\R)\cong \mathfrak{so}(2,1)\), contrary to \(n\geq 3\). This proves the claim.

Now fix \(\mathcal L\in\{\mathcal H,\U\}\), and let \(\nu\) be a weak-* limit of \(\mu_s^{\mathcal L}\). We verify the hypotheses of Proposition~\ref{prop:if projections are full and no iso then equid is full} uniformly in \(\mathcal L\).

For \(k=1,2,3\), define
\[
\mathcal L_k:=
\begin{cases}
\HH_k,& \mathcal L=\mathcal H,\\[2mm]
\U_k,& \mathcal L=\U,
\end{cases}
\quad
\al^{\mathcal{L}}_k(t):=
\begin{cases}
\al_k(t),& \mathcal L=\mathcal H,\\[2mm]
\be_k(t),& \mathcal L=\U.
\end{cases}
\]
Then the \(k\)-th marginal of \(\mu_s^{\mathcal L}\) is of the form
\[
f\longmapsto \int_{\mathcal L_k} f(\al^{\mathcal{L}}_k(s)\ell x_k)\,w_k(\ell)\,d\ell,
\]
for some nonnegative \(w_k\in C_c(\mathcal L_i)\). By Theorem~\ref{thm:margulis banana trick}, together with
Lemma~\ref{lem:H-factorwise-surjectivity} when \(\mathcal L=\mathcal H\) and
Lemma~\ref{lem:U-factorwise-surjectivity} when \(\mathcal L=\U\), it follows that \((\bar\pi_k)_*\nu=\mu_{\G_k/\Gamma_l}\) for \((\bar\pi_k)_*\nu=\mu_{\G_k/\Gamma_l},
k=1,2,3,\) with the obvious convention in the trivial-factor cases \(i=0\) or \(j=0\).

Next, by Lemma~\ref{lem: unipotent invariance general}, together with
Lemma~\ref{lem:lemma implying unip invariance for Hij averages} when \(\mathcal L=\mathcal H\) and
Lemma~\ref{lem:U-product-unipotent} when \(\mathcal L=\U\), the measure \(\nu\) is invariant under a nontrivial unipotent element whose projection to each nontrivial factor is nontrivial. By \cite{Moore1966}, these projections act ergodically on the corresponding factor spaces.

Thus, all the hypotheses of Proposition~\ref{prop:if projections are full and no iso then equid is full} are now satisfied, so \(\nu=\mu_{\G/\tilde\Gamma}\).
\end{proof}

\subsection{The case of $n=2$}
We first treat the unipotent case, and then the two possibilities $H_{1,0}$ and $H_{0,1}$.

\medskip

\noindent\textbf{The unipotent case.}
Let
\begin{equation}\label{eq:deltasandupss}
    \de(s):=\begin{bmatrix}
e^{s/2}&0\\
0&e^{-s/2}
\end{bmatrix},
\quad
\ups(t):=\begin{bmatrix}
1&t\\
0&1
\end{bmatrix}.
\end{equation}
Let $G:=\SO(2,1)^\circ$. Then, there exists an isogeny $\Phi:\SL(2,\R)\to G$ such that:
\[
a(s)=\Phi(\de(s))
=
\exp\!\begin{bsmallmatrix}
0&0&s\\
0&0&0\\
s&0&0
\end{bsmallmatrix},
\quad
u(t)=\Phi(\ups(t))
=
\exp\!\begin{bsmallmatrix}
0&t&0\\
-t&0&t\\
0&t&0
\end{bsmallmatrix}.
\]
Let $X_2:=\SL(2,\R)/\SL(2,\Z)$, let $\Gamma\subseteq G$ be a lattice, and let $w\in C_c(\R)$ be nonnegative. Given $(x_0,y_0)\in G/\Gamma\times X_2$, for each $s\in\R$, let $\mu_s^U$ be the measure defined by  
\[
\mu_s^U(F):=\int_{\R}F(a(s)u(t)x_0,\de(s)\ups(t)y_0)\,w(t)\,dt,
\quad
\forall\, F\in C_c(G/\Gamma\times X_2).
\]
Set $\widetilde\Gamma:=\Phi^{-1}(\Gamma)$, and let \(\bar\Phi:\SL(2,\R)/\widetilde\Gamma\to G/\Gamma\) be the map induced by $\Phi$. If $\tilde x_0\in \SL(2,\R)/\widetilde\Gamma$ projects to $x_0$, then
\[
\mu_s^U(F)=
\int_{\R}
F\bigl(\bar\Phi(\de(s)\ups(t)\tilde x_0),\de(s)\ups(t)y_0\bigr)\,w(t)\,dt.
\]

\newcommand{\dsl}{\SL^\Delta(2,\R)}
\begin{definition}\label{def:commensurability}
Let
\[
\dsl:=\{(\Phi(g),g):g\in\SL(2,\R)\}\subseteq G\times \SL(2,\R).
\]
We say that $(x_0,y_0)\in G/\Gamma\times X_2$ is \emph{commensurable} if the $\dsl$-orbit of $(x_0,y_0)$ is closed.
\end{definition}

\begin{proposition}\label{prop:unipotent equidistribution low dim}
Let $w\in C_c(\R)$ be nonnegative, and let $(x_0,y_0)\in G/\Gamma\times X_2$. Then, for every
$F\in C_c(G/\Gamma\times X_2)$, the following hold:
\begin{enumerate}
    \item If $(x_0,y_0)$ is not commensurable, then
    \begin{equation}\label{eq:full}
        \lim_{s\to\infty}\mu_s^U(F)
        =
        \mu_{G/\Gamma}\otimes\mu_{X_2}(F)\int_{\R}w(t)\,dt.
    \end{equation}
    \item If $(x_0,y_0)$ is commensurable, then
    \begin{equation}\label{eq:L}
        \lim_{s\to\infty}\mu_s^U(F)
        =
        \mu_{\dsl(x_0,y_0)}(F)\int_{\R}w(t)\,dt,
    \end{equation}
    where $\mu_{\dsl(x_0,y_0)}$ denotes the $\dsl$-invariant probability measure on the closed orbit $\dsl(x_0,y_0)$.
\end{enumerate}
\end{proposition}

\begin{proof}
Let
\[
M:=G\times \SL(2,\R),\quad H:=\dsl\subseteq M.
\]
The subgroup \(\{(\Phi(\ups(t)),\ups(t)):t\in\R\}\subseteq H\) is the expanding horospherical subgroup of $(a(s),\de(s))\in H$. Let $L\subseteq M$ be the smallest closed connected subgroup containing $H$ such that $L(x_0,y_0)$ is closed. Since $H\cong \SL(2,\R)$ is generated by one-parameter unipotent subgroups, $L(x_0,y_0)$ carries a unique $L$-invariant probability measure $\mu_{L(x_0,y_0)}$; see \cite[Property-D]{MargulisProceedingsVarna86}, \cite[(2.3) Theorem]{Shah1991}. By \cite[Theorem~1.4]{Shah1996LimitDO},
\[
\lim_{s\to\infty}\mu_s^U(F)
=
\mu_{L(x_0,y_0)}(F)\int_{\R}w(t)\,dt.
\]

Let $p_1:M\to G$ and $p_2:M\to \SL(2,\R)$ be the coordinate projections. Since $p_1(H)=G$, we have $p_1(L)=G$. Set
\[
N:=L\cap (G\times\{e\}).
\]
Then $p_1(N)$ is a connected normal subgroup of the simple group $G$, so either $N=\{e\}$ or $N=G\times\{e\}$.

If $N=G\times\{e\}$, then $L$ contains $G\times\{e\}$ and $H$, hence \(L=(G\times\{e\})H=M\). Therefore, $\mu_{L(x_0,y_0)}=\mu_{G/\Gamma}\otimes\mu_{X_2}$, and \eqref{eq:full} follows.

If $N=\{e\}$, then $p_2|_L:L\to \SL(2,\R)$ is injective. Since $H\subseteq L$ and $p_2(H)=\SL(2,\R)$, it is also surjective. Hence $p_2|_L$ is an isomorphism. For each $g\in\SL(2,\R)$, the unique point of $L$ with second coordinate $g$ must be $(\Phi(g),g)\in H$, since $H\subseteq L$. Thus $L=H=\dsl$. Therefore $(x_0,y_0)$ is commensurable and $\mu_{L(x_0,y_0)}=\mu_{\dsl(x_0,y_0)}$, which gives \eqref{eq:L}.
\end{proof}

\medskip

\noindent\textbf{The cases $H_{1,0}$ and $H_{0,1}$.}
Recall
\[
H_{1,0}=
\left\{
\begin{bsmallmatrix}
\cos\theta&-\sin\theta&0\\
\sin\theta&\cos\theta&0\\
0&0&1
\end{bsmallmatrix}
:\theta\in\R
\right\},
\quad
H_{0,1}=
\left\{
\begin{bsmallmatrix}
1&0&0\\
0&\cosh t&\sinh t\\
0&\sinh t&\cosh t
\end{bsmallmatrix}
:t\in\R
\right\},
\]
and let
\[
h_{1,0}(t):=
\begin{bsmallmatrix}
\cos t&-\sin t\\
\sin t&\cos t
\end{bsmallmatrix},
\quad
h_{0,1}(t):=
\begin{bsmallmatrix}
\cosh t&\sinh t\\
\sinh t&\cosh t
\end{bsmallmatrix}.
\]

Let $w\in C_c(\R)$ be nonnegative, let $(x_0,y_0)\in G/\Gamma\times X_2$, and for each $s\in\R$, define
\begin{align*}
\mu_s^{1,0}(F)
&:=
\int_{\R}
F\!\left(
a(s)\begin{bsmallmatrix} h_{1,0}(t)&0\\ 0&1\end{bsmallmatrix}x_0,\,
\de(s)h_{1,0}(t)y_0
\right)w(t)\,dt,\\
\mu_s^{0,1}(F)
&:=
\int_{\R}
F\!\left(
a(s)\begin{bsmallmatrix} 1&0\\ 0&h_{0,1}(t)\end{bsmallmatrix}x_0,\,
(\de(s)h_{0,1}(t))^*y_0
\right)w(t)\,dt,
\end{align*}
for all $F\in C_c(G/\Gamma\times X_2)$.

\begin{proposition}\label{prop:FullEquidLowDimNonDeg}
For $(i,j)=(1,0),\,(0,1)$ and every $F\in C_c(G/\Gamma\times X_2)$,
\[
\lim_{s\to\infty}\mu_s^{i,j}(F)
=
\mu_{G/\Gamma}\otimes\mu_{X_2}(F)\int_{\R}w(t)\,dt.
\]
\end{proposition}
We observe that there exists an isogeny $\Phi:\SL(2,\R)\to G$ such that
\begin{align}
\Phi\begin{bsmallmatrix}
e^{s/2}&0\\
0&e^{-s/2}
\end{bsmallmatrix}
&=a(s),\nonumber\\
\Phi(h_{1,0}(\theta/2))
&=
\begin{bsmallmatrix}
h_{1,0}(\theta)&0\\
0&1
\end{bsmallmatrix},\label{eq:isogeny formulae}\\
\Phi(h_{0,1}(t/2))
&=
\begin{bsmallmatrix}
1&0\\
0&h_{0,1}(t)
\end{bsmallmatrix}.\nonumber
\end{align}
Using this isogeny, we reduce the proposition to an equidistribution statement in $\SL(2,\R)/\tilde\Ga\times\SL(2,\R)/\SL(2,\Z)$. Let $\widetilde\Gamma:=\Phi^{-1}(\Gamma)$ and let \(\bar\Phi:\SL(2,\R)/\widetilde\Gamma\to G/\Gamma\) be the induced map. Also let $\Psi(g):=(g^\top)^{-1}$; since $\Psi(\SL(2,\Z))=\SL(2,\Z)$, it induces an involution of $X_2$, still denoted by $\Psi$.

If $\tilde x_0\in \SL(2,\R)/\widetilde\Gamma$ projects to $x_0$, then by \eqref{eq:isogeny formulae},
\[
\mu_s^{1,0}(F)
=
\int_{\R}
F\bigl(\bar\Phi(\de(s)h_{1,0}(t/2)\tilde x_0),\de(s)h_{1,0}(t)y_0\bigr)\,w(t)\,dt.
\]
Similarly,
\[
\mu_s^{0,1}(F)
=
\int_{\R}
F\bigl(\bar\Phi(\de(s)h_{0,1}(t/2)\tilde x_0),\Psi(\de(s)h_{0,1}(t)\tilde y_0)\bigr)\,w(t)\,dt.
\]
Here $\tilde y_0=\Psi(y_0)$.

The required equidistribution is therefore a consequence of the following product statement.

\begin{proposition}\label{prop:equid of expanding circle in the product}
Let $\Gamma=\Gamma_1\times\Gamma_2\subseteq \SL(2,\R)\times\SL(2,\R)=:M$ be a lattice, and let $z_0\in M/\Gamma$. Let $I=[a,b]$, and let
\[
\psi(t):=(h(t/2),h(t)),
\]
where $h(t)$ is either $h_{1,0}(t)$ or $h_{0,1}(t)$.

Then, for every continuous nonnegative nonzero function $w:I\to\R$, the probability measures
\[
\eta_s(f):=
\frac{\int_I f\bigl((\de(s),\de(s))\psi(t)z_0\bigr)\,w(t)\,dt}
{\int_I w(t)\,dt},
\quad
f\in C_c(M/\Gamma),
\]
converge to the $M$-invariant probability measure on $M/\Gamma$.
\end{proposition}

We now prove Proposition \ref{prop:equid of expanding circle in the product}. Our main tool is  \cite[Theorem~1.6, Remark~1.7]{Yang_prod_of_SOn1}, which is stated as follows in a form suitable for our needs.

\begin{theorem}\label{thm:lei yang thm}
Let $\Gamma=\Gamma_1\times\Gamma_2\subseteq \SL(2,\R)\times\SL(2,\R)=:M$ be a lattice, and let $z_0\in M/\Gamma$. Let $\varphi=(\varphi_1,\varphi_2):[a,b]\to\R^2$ be analytic. Assume that there exists $s_0\in[a,b]$ such that
\[
(\varphi(t)-\varphi(s_0))^{-1}
:=
\left(
\frac{1}{\varphi_1(t)-\varphi_1(s_0)},
\frac{1}{\varphi_2(t)-\varphi_2(s_0)}
\right)
\]
is defined for almost every $t\in[a,b]$, and assume that
\begin{equation}\label{eq:Yang geometric condition}
\Bigl\{
(\varphi(s_1)-\varphi(s_0))^{-1}
-
(\varphi(s_2)-\varphi(s_0))^{-1}
:
s_1,s_2\in[a,b]
\Bigr\}
\end{equation}
is not contained in any line in $\R^2$. Then, for every subinterval $J\subseteq[a,b]$, every $f\in C_c(M/\Gamma)$, and every continuous function $w:J\to\R$,
\[
\lim_{s\to\infty}
\int_J
f\bigl((\de(s)\ups(\varphi_1(t)),\de(s)\ups(\varphi_2(t)))z_0\bigr)\,w(t)\,dt
=
\mu_{M/\Gamma}(f)\int_J w(t)\,dt,
\]
where $\ups(t)$ and $\de(s)$ are given by \eqref{eq:deltasandupss}.
\end{theorem}

We need the following lemma to verify the geometric condition \eqref{eq:Yang geometric condition} in our proof.

\begin{lemma}\label{lem:verifying Lei geometric condition}
For $\alpha_0,\alpha_1\in\R$, consider the curves
\[
\gamma_{\tan}(t):=
\left(
\frac{1}{\tan(t/2)-\alpha_1},\,
\frac{1}{\tan t-\alpha_0}
\right),\; \gamma_{\tanh}(t):=
\left(
\frac{1}{\tanh(t/2)-\alpha_1},\,
\frac{1}{\tanh t-\alpha_0}
\right)
\]
defined wherever the denominators are nonzero. Then neither image is contained in an affine line in $\R^2$.
\end{lemma}

\begin{proof}
We show the proof for $\gamma_{\tan}$; the proof for \(\gamma_{\tanh}\) is the same as $\tanh(t)=-i\tan(it)$.  

Suppose, toward a contradiction, that its image is contained in an affine line. Then there exist $A,B,C\in\R$, not all zero, such that
\[
\frac{A}{\tan(t/2)-\alpha_1}
+
\frac{B}{\tan t-\alpha_0}
=
C
\]
for all $t$ in some interval on which the curve is defined.

Set $x=\tan(t/2)$. Since \(
\tan t=\frac{2x}{1-x^2},
\) this becomes
\[
\frac{A}{x-\alpha_1}
+
B\,\frac{1-x^2}{\alpha_0x^2+2x-\alpha_0}
=
C.
\]
After clearing denominators,
\[
A(\alpha_0x^2+2x-\alpha_0)+B(1-x^2)(x-\alpha_1)
=
C(x-\alpha_1)(\alpha_0x^2+2x-\alpha_0).
\]
Comparing the coefficients of $x^3$ and of the constant term gives
\[
-B=C\alpha_0,
\quad
-B\alpha_1-A\alpha_0=C\alpha_0\alpha_1.
\]
If $\alpha_0=0$, then the first identity gives $B=0$, and the original relation reduces to
\(
\frac{A}{x-\alpha_1}=C,
\)
which is impossible on an interval unless $A=0$, contradiction. Thus $\alpha_0\neq0$. Substituting $C=-B/\alpha_0$ into the second identity yields $A\alpha_0=0$, hence $A=0$. Comparing the coefficient of $x^2$ now gives
\[
B\alpha_1=\frac{-B}{\alpha_0}(2-\alpha_0\alpha_1),
\]
that is,
\(\alpha_0\alpha_1=-2+\alpha_0\alpha_1,\) a contradiction.
\end{proof}

\begin{proof}[Proof of Proposition \ref{prop:equid of expanding circle in the product}]
We first treat \(h(t)=h_{1,0}(t)\). Let
\[
u_-(x):=\begin{bmatrix}
1&0\\
x&1
\end{bmatrix},
\quad
b(r):=\begin{bmatrix}
r&0\\
0&r^{-1}
\end{bmatrix}.
\]
In view of $u_-(\R) b(\R^\times)\ups(\R)$ being Zariski open and dense in $\SL(2,\R)$, for $t\notin \frac{\pi}{2}+\pi\Z$, we find that
\[
h(t)=u_-(\tan t)\,b(\cos t)\,\ups(-\tan t),
\]
and therefore
\begin{equation}\label{eq:circle decomposition}
\de(s)h(t)=u_-(e^{-s}\tan t)\,b(\cos t)\,\de(s)\ups(-\tan t).
\end{equation}

Let $K:=\text{supp}(w)\subseteq I$. The singular set
\[
\Sigma:=K\cap\Bigl(\Bigl(\frac{\pi}{2}+\pi\Z\Bigr)\cup(\pi+2\pi\Z)\Bigr)
\]
is finite; these are precisely the poles of $\tan t$ and $\tan(t/2)$ on $K$. Fix $\e>0$. Choose an open neighborhood $\mathcal U$ of $\Sigma$ in $I$ such that \(\int_{\mathcal U}w(t)\,dt<\e.\) Since $f$ is bounded,
\[
\left|
\int_{\mathcal U}f\bigl((\de(s)h(t/2),\de(s)h(t))z_0\bigr)\,w(t)\,dt
\right|
\leq
\|f\|_\infty\,\e.
\]

Write \(K\smallsetminus \mathcal U=\bigsqcup_{m=1}^N J_m\) as a finite disjoint union of compact intervals. On each $J_m$, the functions $\tan t$, $\tan(t/2)$, $\cos t$, and $\cos(t/2)$ are continuous and bounded. Fix such an interval $J=[c,d]$.

By uniform continuity of $f$, for every $\e>0$ there exists an identity neighborhood $O_\e\subseteq M$ such that \(|f(gz)-f(z)|\le \e
\) for all \(g\in O_\e,\ z\in M/\Gamma.\)

Since $\tan t$ and $\tan(t/2)$ are bounded on $J$, \eqref{eq:circle decomposition} implies that for all sufficiently large $s$,
\[
\bigl(u_-(e^{-s}\tan(t/2)),u_-(e^{-s}\tan t)\bigr)\in O_\e
\quad\forall\,t\in J.
\]
Next, partition $J$ into finitely many closed intervals $J_1',\dots,J_r'$ and choose $t_\ell\in J_\ell'$ so that
\[
\bigl(b(\cos(t/2)),b(\cos t)\bigr)\bigl(b(\cos(t_\ell/2)),b(\cos t_\ell)\bigr)^{-1}\in O_\e
\quad\forall\,t\in J_\ell'.
\]
Then, for $t\in J_\ell'$ and all sufficiently large $s$,
\[
f\bigl((\de(s)h(t/2),\de(s)h(t))z_0\bigr)
=
f_\ell\bigl((\de(s)\ups(-\tan(t/2)),\de(s)\ups(-\tan t))z_0\bigr)+O(\e),
\]
where
\[
f_\ell(z):=
f\bigl((b(\cos(t_\ell/2)),b(\cos t_\ell))z\bigr).
\]
Hence
\begin{align*}
  &\int_{J_\ell'}
f\bigl((\de(s)h(t/2),\de(s)h(t))z_0\bigr)\,w(t)\,dt
\\
&=
\int_{J_\ell'}
f_\ell\bigl((\de(s)\ups(-\tan(t/2)),\de(s)\ups(-\tan t))z_0\bigr)\,w(t)\,dt
+
O\!\left(\e\int_{J_\ell'}w(t)\,dt\right).  
\end{align*}

We now apply Theorem \ref{thm:lei yang thm} on each $J_\ell'$ with \(\varphi(t)=(-\tan(t/2),-\tan t).\) Fix $s_0\in J_\ell'$. Then $(\varphi(t)-\varphi(s_0))^{-1}$ is defined for all $t\in J_\ell'\smallsetminus\{s_0\}$. To verify \eqref{eq:Yang geometric condition}, suppose the set in \eqref{eq:Yang geometric condition} were contained in a line in $\R^2$. Fixing $s_2\in J_\ell'$, the curve
\[
t\mapsto
(\varphi(t)-\varphi(s_0))^{-1}
-
(\varphi(s_2)-\varphi(s_0))^{-1}
\]
would then be contained in an affine line. But this curve is of the form covered by Lemma \ref{lem:verifying Lei geometric condition}, with \(\alpha_1=-\tan(s_0/2)\) and \(\alpha_0=-\tan(s_0)\).
Hence \eqref{eq:Yang geometric condition} holds. Therefore, by Theorem \ref{thm:lei yang thm},
\[
\lim_{s\to\infty}
\int_{J_\ell'}
f_\ell\bigl((\de(s)\ups(-\tan(t/2)),\de(s)\ups(-\tan t))z_0\bigr)\,w(t)\,dt
=
\mu_{M/\Gamma}(f)\int_{J_\ell'}w(t)\,dt,
\]
where we used the $M$-invariance of $\mu_{M/\Gamma}$.

Summing over $\ell$ and then over $m$, and using the bound on $\mathcal U$, we obtain
\[
\limsup_{s\to\infty}
\left|
\int_I f\bigl((\de(s)h(t/2),\de(s)h(t))z_0\bigr)\,w(t)\,dt
-
\mu_{M/\Gamma}(f)\int_I w(t)\,dt
\right|
\ll
\e\bigl(\|f\|_\infty+\int_I w(t)\,dt\bigr).
\]
Letting $\e\to0$ proves the proposition in the trigonometric case.

The hyperbolic case $h(t)=h_{0,1}(t)$ is identical, using
\[
h(t)=u_-(\tanh t)\,b(\cosh t)\,\ups(\tanh t),
\quad
\de(s)h(t)=u_-(e^{-s}\tanh t)\,b(\cosh t)\,\de(s)\ups(\tanh t),
\]
and applying Theorem \ref{thm:lei yang thm} to \(\varphi(t)=(\tanh(t/2),\tanh t)\). The required geometric condition follows from Lemma \ref{lem:verifying Lei geometric condition} in exactly the same way.
\end{proof}

\section{Proving the main results}

In this section we combine the volume estimates from Section~\ref{sec:vol estiamtes} with the equidistribution results from Section~\ref{sec:equid results}.

The results from the previous sections apply most naturally to integrals of the form
\[
\frac{1}{\vol(G_T)}\int_{G_T}\varphi(g\Ga)\,f\bigl(g\cdot([\La_1],[\La_2])\bigr)\,dg
\]
when $([\La_1],[\La_2])\in X_{\Pstd}$ or $([\La_1],[\La_2])\in X_{\Pinfty}$; see
\eqref{eq:def of ver and hor vspaces} and \eqref{eq:def of Pinfty} to recall \(\Pstd\) and \(\Pinfty\).
For a general point of $X_{r,n+1}$, we reduce to one of these situations by translation as follows.

Let $([\tilde\La_1],[\tilde\La_2])\in X_{r,n+1}$. According to our conventions, either $\tilde\La_1$ is positive definite or degenerate. If $\tilde\La_1$ is positive definite, choose $g_0\in G$ so that
$g_0^{-1}\tilde\La_1\subseteq P^{(0)}_1=\Span_\R\{\mathbf e_1,\dots,\mathbf e_r\}$.
If $\tilde\La_1$ is degenerate, choose $g_0\in G$ so that
$g_0^{-1}\tilde\La_1\subseteq P^\infty_1=\Span_\R\{\mathbf e_1+\mathbf e_{n+1},\mathbf e_2,\dots,\mathbf e_r\}$.
Consider the skewed norm \(  \|g\|_{g_0}:=\|gg_0^{-1}\|\) and denote
\begin{equation*}
    G_T^{g_0}:=\{g\in G:\|g\|_{g_0}\le T\}.
\end{equation*}
Then $G_Tg_0=G_T^{g_0}$, and by right-invariance of Haar measure,
\begin{equation}\label{eq:reducing general integral by translating}
\frac{1}{\vol(G_T)}\int_{G_T}\varphi(g\Ga)f\bigl(g\cdot([\tilde\La_1],[\tilde\La_2])\bigr)\,dg
=
\frac{1}{\vol(G_T^{g_0})}\int_{G_T^{g_0}}\varphi(gg_0^{-1}\Ga)\,f\bigl(g\cdot([\La_1],[\La_2])\bigr)\,dg,
\end{equation}
where \(\)
\[
([\La_1],[\La_2]):=g_0^{-1}\cdot([\tilde\La_1],[\tilde\La_2]),
\]
$([\La_1],[\La_2])\in X_{\Pstd}$ in the positive-definite case, and
$([\La_1],[\La_2])\in X_{\Pinfty}$ in the degenerate case.

\subsection{Partition of unity}\label{sec:PartOfUnity}

Our next goal is to identify the weak-$*$ limits of the measures
$\mu_{x,y,k,T}$ from \eqref{eq:def of msr on XiXj Hij case}
and $\nu_{x,y,k,T}$ from \eqref{eq:def of msr on XiXj U case}.
As in \cite[Proof of Theorem~10.1]{Gorodnik2004DistributionOL},
the key point is a partition-of-unity argument.

\begin{lemma}\label{lem:bounding above and below bT by small O}
Let $H$ denote either $H_{i,j}$ or $U$.
For every $0<\e<1/2$, there exists a symmetric identity neighborhood
$O\subseteq H$ such that for all $h\in H$, $k\in K$, $L>0$, and $u_1,u_2,u_3\in O$,
\[
b_{T(1-\e),L-1}(k,hu_1)\subseteq b_{T,L}(k,hu_2)\subseteq b_{T(1+\e),L+1}(k,hu_3),
\]
where \(b_{T,L}(k,h)\subset\R\) are defined by \eqref{eq:definition of bTL}.
\end{lemma}

\begin{proof}
Choose $0<\delta<\e$ such that $(1+\delta)(1-\e)\le 1$.
Let $C\ge 1$ be such that $\|AB\|\le C\|A\|\,\|B\|$ for all matrices $A,B$.
Choose a symmetric identity neighborhood $O\subseteq H$ such that
\[
\|I-u\|\le \delta/C,\, \quad\forall\,u\in O^2.
\]
If $h_1,h_2\in hO$, then $h_2^{-1}h_1\in O^2$, hence for any $s\in\R$, 
\[
\|ka(s)h_1\|
\le \|ka(s)h_2\|+\|ka(s)h_2(h_2^{-1}h_1-I)\|
\le (1+\delta)\|ka(s)h_2\|.
\]
Interchanging $h_1$ and $h_2$ gives
\[
(1-\delta)\|ka(s)h_2\|\le \|ka(s)h_1\|\le (1+\delta)\|ka(s)h_2\|.
\]

Now let $s\in b_{T(1-\e),L-1}(k,hu_1)$.
Then
\[
s\ge \log(T(1-\e))-(L-1)\ge \log T-L,
\]
and
\[
\|ka(s)hu_2\|\le (1+\delta)\|ka(s)hu_1\|
\le (1+\delta)T(1-\e)\le T.
\]
Hence $s\in b_{T,L}(k,hu_2)$.
The second inclusion is proved in the same way.
\end{proof}

For $H=H_{i,j}$ with $i\neq 0$, or for $H=U$, define $c:K\to\R$ by
\[
c(k_0):=
\frac{\int_H \|k_0a(\infty)h\|^{-(n-1)}\,dh}
{\int_K\int_H \|ka(\infty)h\|^{-(n-1)}\,dh\,dk},\; \textup{for }k_0\in K.
\]
For $H=H_{0,n-1}$ define
\[
c(k_0):=
\frac{\int_H\Bigl(\|k_0a(\infty)h\|^{-(n-1)}+\|k_0a(-\infty)h\|^{-(n-1)}\Bigr)\,dh}
{\int_K\int_H\Bigl(\|ka(\infty)h\|^{-(n-1)}+\|ka(-\infty)h\|^{-(n-1)}\Bigr)\,dh\,dk},\; \textup{for }k_0\in K.
\]

\begin{proposition}\label{prop:convergence of inner integrals}
Let $n\ge 2$, let $i,j\ge 0$ with $i+j=n-1$, let $x_0\in G/\Ga$ and
$y_0\in X_{i+1}\times X_{j+1}$, and let $k\in K$.
Then the following hold.
\begin{enumerate}
    \item For every $F\in C_c(G/\Ga\times X_{i+1}\times X_{j+1})$,
    \[
    \lim_{T\to\infty}\mu_{x_0,y_0,k,T}(F)
    =
    c(k)\,\mu_{G/\Ga}\otimes\mu_{X_{i+1}}\otimes\mu_{X_{j+1}}(F).
    \]
    \item Suppose that $n\ge 3$, or that $n=2$ and $(x_0,y_0)$ is not commensurable
    in the sense of \Cref{def:commensurability}.
    Then for every $F\in C_c(G/\Ga\times X_{i+1}\times X_{j+1})$,
    \[
    \lim_{T\to\infty}\nu_{x_0,y_0,k,T}(F)
    =
    c(k)\,\mu_{G/\Ga}\otimes\mu_{X_{i+1}}\otimes\mu_{X_{j+1}}(F).
    \]
    \item Suppose that $n=2$ and that $(x_0,y_0)$ is commensurable.
    Then for every $F\in C_c(G/\Ga\times X_2)$,
    \[
    \lim_{T\to\infty}\nu_{x_0,y_0,k,T}(F)
    =
    c(k)\,\mu_{\SL^\De(2,\R)(x_0,y_0)}(F),
    \]
    where $\mu_{\SL^\De(2,\R)(x_0,y_0)}$ is the $\SL^\De(2,\R)$-invariant probability
    measure on $\SL^\De(2,\R)(x_0,y_0)$.
\end{enumerate}
\end{proposition}

\begin{proof}
We give the details for~(1); the proofs of~(2) and~(3) are identical, with the same partition-of-unity
argument and the corresponding $U$-equidistribution statement from Section~\ref{sec:equid results}.

By linearity it is enough to treat $F\ge 0$.
Fix $\e>0$.
By Corollary~\ref{cor:H component can be bounded by paying epsilon}, there exists $L>0$ such that
\begin{align*}
\mu_{x_0,y_0,k,T}(F)
&=
\frac{1}{\vol(G_T)}
\int_{H_{i,j}}\int_{b_T(k,h)}
F\bigl(a(s)hx_0,d_1(s)\hh y_{0,1},d_2(s)\vh^* y_{0,2}\bigr)\omega(s)\,ds\,dh\\
&=
\frac{1}{\vol(G_T)}
\int_{H_{i,j}}\int_{b_{T,L}(k,h)}
F\bigl(a(s)hx_0,d_1(s)\hh y_{0,1},d_2(s)\vh^* y_{0,2}\bigr)\omega(s)\,ds\,dh
+O(\e),
\end{align*}
where \(\omega(s)=\sinh(s)^i\cosh(s)^j.\)

Again by Corollary~\ref{cor:H component can be bounded by paying epsilon}, there exists $c_0>0$
such that $b_{T,L}(k,h)=\emptyset$ whenever $\|h\|\ge c_0e^L$.

Choose a symmetric identity neighborhood $O\subseteq H_{i,j}$ as in
Lemma~\ref{lem:bounding above and below bT by small O}.
Since
\[
H_{c_0e^L}:=\{h\in H_{i,j}:\|h\|\le c_0e^L\}
\]
is compact, we may cover it by finitely many translates
\begin{equation}\label{eq:cover of H compact range}
H_{c_0e^L}\subseteq \bigcup_{q=1}^N h_qO.
\end{equation}
Let $w_1,\dots,w_M\in C_c(H_{i,j})$ be a partition of unity subordinate to this cover.
Thus
\[
\sum_{l=1}^M w_l(h)=1 \quad\forall\,h\in H_{c_0e^L},
\]
and for each $l$ the support of $w_l$ is contained in some $h_{q(l)}O$.

Using this partition of unity, we write
\begin{align*}
&\int_{H_{i,j}}\int_{b_{T,L}(k,h)}
F\bigl(a(s)hx_0,d_1(s)\hh y_{0,1},d_2(s)\vh^* y_{0,2}\bigr)\omega(s)\,ds\,dh\\
&\quad=
\sum_{l=1}^M
\int_{h_{q(l)}O}w_l(h)\int_{b_{T,L}(k,h)}
F\bigl(a(s)hx_0,d_1(s)\hh y_{0,1},d_2(s)\vh^* y_{0,2}\bigr)\omega(s)\,ds\,dh.
\end{align*}
For fixed $l$, Lemma~\ref{lem:bounding above and below bT by small O} gives
\[
b_{T,L}(k,h)\subseteq b_{T(1+\e),L+1}(k,h_{q(l)})\quad\forall\,h\in h_{q(l)}O,
\]
and therefore
\begin{align*}
&\int_{h_{q(l)}O}w_l(h)\int_{b_{T,L}(k,h)}
F\bigl(a(s)hx_0,d_1(s)\hh y_{0,1},d_2(s)\vh^* y_{0,2}\bigr)\omega(s)\,ds\,dh\\
&\quad\le
\int_{h_{q(l)}O}w_l(h)\int_{b_{T(1+\e),L+1}(k,h_{q(l)})}
F\bigl(a(s)hx_0,d_1(s)\hh y_{0,1},d_2(s)\vh^* y_{0,2}\bigr)\omega(s)\,ds\,dh\\
&\quad=
\int_{b_{T(1+\e),L+1}(k,h_{q(l)})}\omega(s)\,ds\int_{h_{q(l)}O}
F\bigl(a(s)hx_0,d_1(s)\hh y_{0,1},d_2(s)\vh^* y_{0,2}\bigr)w_l(h)\,dh.
\end{align*}
We may now apply the expanding-$H_{i,j}$ equidistribution statements Proposition \ref{prop:HigherDimensionalExpendingPieces} and Proposition \ref{prop:FullEquidLowDimNonDeg} to the $h$-integral. For $T$ sufficiently large,
\begin{align*}
&\int_{b_{T(1+\e),L+1}(k,h_{q(l)})}\omega(s)\,ds\int_{h_{q(l)}O}
F\bigl(a(s)hx_0,d_1(s)\hh y_{0,1},d_2(s)\vh^* y_{0,2}\bigr)w_l(h)\,dh\\
&\quad=
(1+O(\e))
\mu_{G/\Ga}\otimes\mu_{X_{i+1}}\otimes\mu_{X_{j+1}}(F)\,
\int_{b_{T(1+\e),L+1}(k,h_{q(l)})}\omega(s)\,ds
\int_{H_{i,j}}w_l(h)\,dh.
\end{align*}
Using Lemma~\ref{lem:bounding above and below bT by small O} once more,
\begin{align*}
\int_{b_{T(1+\e),L+1}(k,h_{q(l)})}\omega(s)\,ds
\int_{H_{i,j}}w_l(h)\,dh\le
\int_{H_{i,j}} w_l(h)\int_{b_{T(1+\e)^2,L+2}(k,h)}\omega(s)\,ds\,dh.
\end{align*}
Summing over $l$ yields
\[
\mu_{x_0,y_0,k,T}(F)
\le
(1+O(\e))
\mu_{G/\Ga}\otimes\mu_{X_{i+1}}\otimes\mu_{X_{j+1}}(F)
\frac{\int_{H_{i,j}}\int_{b_{T(1+\e)^2,L+2}(k,h)}\omega(s)\,ds\,dh}
{\vol(G_T)}.
\]
The corresponding lower bound is obtained in the same way:
\[
\mu_{x_0,y_0,k,T}(F)
\ge
(1+O(\e))
\mu_{G/\Ga}\otimes\mu_{X_{i+1}}\otimes\mu_{X_{j+1}}(F)
\frac{\int_{H_{i,j}}\int_{b_{T(1-\e)^2,L-2}(k,h)}\omega(s)\,ds\,dh}
{\vol(G_T)}.
\]
By Lemma~\ref{lem:assymp of bT integral for gen case},
Corollary~\ref{cor:volume of G balls},
and Corollary~\ref{cor:H component can be bounded by paying epsilon},
both outer ratios converge to $c(k)+O(\e)$.
Since $\e>0$ is arbitrary, this proves~(1).
\end{proof}

\subsection{Proof of the main results}

We split the argument into two parts.
The proofs of \Cref{thm:non-degen orb,thm:degen orbits higher dim}
and of the non-special case of \Cref{thm:degen orbits low dim}
follow the same pattern.
The special case in \Cref{thm:degen orbits low dim} requires an additional argument to relate the limiting $\SL(2,\R)$-diagonal embedded measure to our to the limit of the lattice norm ball averages.

\subsubsection{The generic cases}

We give the details for \Cref{thm:non-degen orb} when $r\ge 2$.
The degenerate non-special case is entirely analogous, and the case $r=1$
is handled in the same way after keeping the extra $a(-\infty)$-term coming from the
$H_{0,n-1}$-decomposition.

Let $([\tilde\La_1],[\tilde\La_2])\in X$ and assume that $\tilde\La_1$ is positive definite.
Let $g_0\in G$ such that \(g_0^{-1}\tilde\La_1\subseteq P^{(0)}_1.\) Denote \(([\La_1],[\La_2]):=g_0^{-1}\cdot([\tilde\La_1],[\tilde\La_2])\) and let \(x_0:=g_0\Ga\). Choose $y_0\in X_{i+1}\times X_{j+1}$ with \(\rho^{(0)}(y_0)=([\La_1],[\La_2]).\)

Assume $i>0$.
Using \eqref{eq:reducing general integral by translating} and the Haar decomposition
\eqref{eq:Haar measure in KAHij for i nonzero}, we obtain
\begin{align}
&\frac{1}{\vol(G_T)}\int_{G_T}\varphi(g\Ga)\,f\bigl(g\cdot([\tilde\La_1],[\tilde\La_2])\bigr)\,dg \nonumber\\
&\quad=
\int_K
\frac{1}{\vol(G_T^{g_0})}
\int_{H_{i,j}}\int_{b_T^{g_0}(k,h)}
\varphi\bigl(ka(s)hx_0\bigr)\,
f\bigl(ka(s)h\cdot([\La_1],[\La_2])\bigr)\,
\omega(s)\,ds\,dh\,dk,
\label{eq:main proof nondeg}
\end{align}
where
\[
b_T^{g_0}(k,h):=\{s\ge 0:\|ka(s)hg_0^{-1}\|\le T\},
\quad
\omega(s)=\sinh(s)^i\cosh(s)^j.
\]

Apply Lemma~\ref{lem:replacing by averages in GmodD times expading in XiXj}
and Proposition~\ref{prop:convergence of inner integrals} to the skewed norm
$\|\cdot\|_{g_0}$.
The dominated convergence theorem then gives
\begin{equation}\label{eq:generic limit after K integration}
\lim_{T\to\infty}
\frac{1}{\vol(G_T)}\int_{G_T}\varphi(g\Ga)\,f\bigl(g\cdot([\tilde\La_1],[\tilde\La_2])\bigr)\,dg
=
\int_K c_{g_0}(k)\,(k\Psi)_*
\bigl(\mu_{G/\Ga}\otimes\mu_{X_{i+1}}\otimes\mu_{X_{j+1}}\bigr)(\varphi\otimes f)\,dk,
\end{equation}
where
\[
c_{g_0}(k_0):=
\frac{\int_{H_{i,j}}\|k_0a(\infty)hg_0^{-1}\|^{-(n-1)}\,dh}
{\int_K\int_{H_{i,j}}\|ka(\infty)hg_0^{-1}\|^{-(n-1)}\,dh\,dk}, \; \textup{for }k_0\in K.
\]
Since
\[
(k\Psi)_*
\bigl(\mu_{G/\Ga}\otimes\mu_{X_{i+1}}\otimes\mu_{X_{j+1}}\bigr)(\varphi\otimes f)
=
\mu_{G/\Ga}(\varphi)\,\mu_{X_{k\cdot\Pinfty}}(f),
\]
we obtain
\[
\lim_{T\to\infty}
\frac{1}{\vol(G_T)}\int_{G_T}\varphi(g\Ga)\,f\bigl(g\cdot([\tilde\La_1],[\tilde\La_2])\bigr)\,dg
=
\mu_{G/\Ga}(\varphi)\int_K \mu_{X_{k\cdot\Pinfty}}(f)\,c_{g_0}(k)\,dk.
\]
Finally, $\mu_{X_{k\kappa\cdot\Pinfty}}=\mu_{X_{k\cdot\Pinfty}}$ for every $\kappa\in K_{\Pinfty}$, so
\[
\int_K \mu_{X_{k\cdot\Pinfty}}(f)\,c_{g_0}(k)\,dk
=
\int_{K/K_{\Pinfty}}
\mu_{X_{k\cdot\Pinfty}}(f)
\left(\int_{K_{\Pinfty}}c_{g_0}(k\kappa)\,d\kappa\right)\,d(kK_{\Pinfty}).
\]
This is exactly the density appearing in \Cref{thm:non-degen orb}.
Applying \Cref{prop:reduction general method} finishes the proof.

\subsubsection{The special case in \Cref{thm:degen orbits low dim}}

Let $\Ga\subseteq G$ be a lattice, and let $\La$ be a degenerate $2$-lattice which is $\Ga$-special.
Choose $k_{\theta_0}\in K$ such that \(\La\subseteq k_{\theta_0}P^\infty\).

Using \eqref{eq:reducing general integral by translating} and the Haar decomposition
\eqref{eq:Haar measure in KAU}, we obtain
\begin{align}
&\frac{1}{\vol(G_T)}\int_{G_T}\varphi(g\Ga)\,f(g\cdot[\La])\,dg\nonumber\\
&=
\int_0^{2\pi}
\frac{1}{\vol(G_T^{k_{\theta_0}})}
\int_{\R}\int_{b_T^{\theta_0}(\theta,u(t))}
\varphi\bigl(k_\theta a(s)u(t) k_{\theta_0}^{-1}\Ga\bigr)\,
f\bigl(k_{\theta} a(s)u(t)k_{\theta_0}^{-1}\cdot[\La]\bigr)\,
e^s\,ds\,dt\,d\theta\nonumber\\
&=
\int_0^{2\pi}
\frac{1}{\vol(G_T^{k_{\theta_0}})}
\int_{\R}\int_{b_T^{\theta_0}(\theta,u(t))}
\varphi\bigl(k_\theta\Phi(\de(s)\ups(t)) k_{\theta_0}^{-1}\Ga\bigr)\,
f\bigl(k_{\theta-\theta_0}\cdot[\rho_{\theta_0}^{\infty}(\de(s)\ups(t))\La]\bigr)\,
e^s\,ds\,dt\,d\theta\nonumber,
\end{align}
where $\Phi:\SL(2,\R)\to G$ is the isogeny characterized by \eqref{eq:CharacterisationOfIsogeny}, where  recalling \eqref{eq:DefOfRhoInftyTheta},
\[
\rho_{\theta_0}^{\infty}(g):=k_{\theta_0}\rho^\infty(g)k_{\theta_0}^{-1},
\quad g\in \SL(2,\R),
\]
and where \(b_T^{\theta_0}(\theta,u(t))
:=
\{s\ge 0:\|k_\theta a(s)u(t)k_{\theta_0}^{-1}\|\le T\}\).

Since $\La$ is $\Ga$-special, the subgroup \(S:=\Ga_{\theta_0}^\Phi \cap \stab_{\theta_0}([\La])\) is a lattice in $\SL(2,\R)$; see Remark \ref{rem:StabIslatticeForSpecialOrbits}.
Set \(x_0:=k_{\theta_0}^{-1}\Ga\), and choose \(y_0\in \SL(2,\R)/\SL(2,\Z)\) such that \(\rho_{\theta_0}^\infty(y_0)=[\La]\).

Then \((x_0,y_0)\) is commensurable in the sense of \Cref{def:commensurability}. Hence,
Lemma~\ref{lem:replacing by averages in GmodD times expading in XiXj, unipotent case}
and Proposition~\ref{prop:convergence of inner integrals},
applied with respect to the skewed norm $\|\cdot\|_{k_{\theta_0}}$, yield
\begin{equation}\label{eq:special orbit limit on closed orbit rewritten}
\lim_{T\to\infty}
\frac{1}{\vol(G_T)}\int_{G_T}\varphi(g\Ga)\,f(g\cdot[\La])\,dg
=
\int_0^{2\pi}
w(\theta)\,
(k_\theta\Psi^\infty)_*\mu_Y(\varphi\otimes f)\,d\theta,
\end{equation}
where
\[Y:=\SL^\De(2,\R)(x_0,y_0)=\bigl\{(k_{\theta_0}\Phi(g)k_{\theta_0}^{-1}\Ga,[\rho_{\theta_0}^\infty(g)\La]):g\in\SL(2,\R)\bigr\},\]
where  $\mu_Y$ is the $\SL(2,\R)$-invariant probability on $Y$, and
\[
w(\vartheta):=
\frac{\int_\R \|k_\vartheta a(\infty)u(t)k_{\theta_0}^{-1}\|^{-1}\,dt}
{\int_0^{2\pi}\int_\R \|k_\theta a(\infty)u(t)k_{\theta_0}^{-1}\|^{-1}\,dt\,d\theta}\,.
\]

Identifying $Y$ with $\SL(2,\R)/S$, and writing $d\bar\eta$ for the $\SL(2,\R)$-invariant probability measure on $\SL(2,\R)/S$, we get
\begin{align*}
(k_\theta\Psi^\infty)_*\mu_Y(\varphi\otimes f)
&=
\int_{\SL(2,\R)/S}
\varphi\bigl(k_\theta\Phi(\bar\eta)k_{\theta_0}^{-1}\Ga\bigr)\,
f\bigl(k_\theta\cdot[\rho^\infty(\bar\eta)k_{\theta_0}^{-1}\La]\bigr)\,d\bar\eta\\
&=
\int_{\SL(2,\R)/S}
\varphi\bigl(k_{\theta_0}\Phi(\bar\eta)k_{\theta_0}^{-1}\Ga\bigr)\,
f\bigl(k_{\theta-\theta_0}\cdot[\rho_{\theta_0}^\infty(\kappa_{-(\theta-\theta_0)/2}\bar\eta)\La]\bigr)\,d\bar\eta,
\end{align*}
where we used the identity \(k_\vartheta=\Phi(\kappa_{\vartheta/2})\) and the left-invariance of \(d\bar\eta\).

Let \(S_1:=\Ga_{\theta_0}^\Phi,\) choose representatives \(\sigma_1,\dots,\sigma_m\in S_1\) for the finite quotient \(S_1/S\), and define
\begin{equation}\label{eq:SpecialOrbit in the final proof}
[\La_i]:=[\rho_{\theta_0}^\infty(\sigma_i)\La],\quad i=1,\dots,m.
\end{equation}
Then \(\{[\La_1],\dots,[\La_m]\}\) is precisely the \((\Ga,\La)\)-packet.

Consider the first-coordinate projection \(p:Y\to G/\Ga.\) Under the identifications \(Y\simeq \SL(2,\R)/S\) and \(G/\Ga\simeq \SL(2,\R)/S_1\), the map \(p\) is the natural quotient map \(\SL(2,\R)/S\to \SL(2,\R)/S_1\). Since \([S_1:S]=m\), it is an \(m\)-sheeted finite covering.

Choose a sufficiently small identity neighborhood \(U\subseteq G\) and a bump function
\(\varphi\in C_c(G/\Ga)\) supported on \(U\Ga/\Ga\) with \(\mu_{G/\Ga}(\varphi)=1\).
Shrinking \(U\) if necessary, 
\[
p^{-1}(U\Ga/\Ga)=Y_1\sqcup\cdots\sqcup Y_m,
\]
where each restriction \(p|_{Y_i}:Y_i\to U\Ga/\Ga\) is a homeomorphism, and we may arrange that \(\sigma_iS\in Y_i\).

Moreover, each sheet carries mass \(1/m\) with respect to \(\varphi\circ p\):
\begin{equation}\label{eq:sheet mass one over m}
\int_{Y_i}\varphi\circ p\,d\mu_Y=\frac1m,
\quad i=1,\dots,m.
\end{equation}
Indeed, after lifting to a sufficiently small neighborhood \(\widetilde U\subseteq \SL(2,\R)\), the sheets have the form
\[
\widetilde U\sigma_iS/S,\quad i=1,\dots,m,
\]
and each \(\widetilde U\sigma_i\) projects injectively to both
\(\SL(2,\R)/S\) and \(\SL(2,\R)/S_1\). Since Haar measure on \(\SL(2,\R)\) is right-invariant, the corresponding sheet integrals are equal. Summing over \(i\) gives
\[
\sum_{i=1}^m\int_{Y_i}\varphi\circ p\,d\mu_Y
=
\int_Y\varphi\circ p\,d\mu_Y
=
\mu_{G/\Ga}(\varphi)
=
1,
\]
which proves \eqref{eq:sheet mass one over m}.

Fix \(\e>0\). By uniform continuity of \(f\) and compactness of \(\theta\in[0,2\pi]\), after shrinking \(U\) further if necessary, we may assume that uniformly in \(\theta\) and uniformly for \(\bar\eta\in Y_i\),
\[
f\bigl(k_{\theta-\theta_0}\cdot[\rho_{\theta_0}^\infty(\kappa_{-(\theta-\theta_0)/2}\bar\eta)\La]\bigr)
=
f\bigl(k_{\theta-\theta_0}\cdot[\rho_{\theta_0}^\infty(\kappa_{-(\theta-\theta_0)/2}\sigma_i)\La]\bigr)
+O(\e).
\]
Therefore,
\begin{align*}
(k_\theta\Psi^\infty)_*\mu_Y(\varphi\otimes f)
&=
\sum_{i=1}^m
\int_{Y_i}
\varphi\bigl(k_{\theta_0}\Phi(\bar\eta)k_{\theta_0}^{-1}\Ga\bigr)\,
f\bigl(k_{\theta-\theta_0}\cdot[\rho_{\theta_0}^\infty(\kappa_{-\theta-\theta_0)/2}\bar\eta)\La]\bigr)\,d\bar\eta\\
&=
\sum_{i=1}^m
f\bigl(k_{\theta-\theta_0}\cdot[\rho_{\theta_0}^\infty(\kappa_{-\theta-\theta_0)/2}\sigma_i)\La]\bigr)
\int_{Y_i}\varphi\bigl(k_{\theta_0}\Phi(\bar\eta)k_{\theta_0}^{-1}\Ga\bigr)\,d\bar\eta
+O(\e)\\
&=
\frac1m\sum_{i=1}^m
f\bigl(k_{\theta-\theta_0}\cdot[\rho_{\theta_0}^\infty(\kappa_{-(\theta-\theta_0)/2}\sigma_i)\La]\bigr)
+O(\e).
\end{align*}
Substituting this into \eqref{eq:special orbit limit on closed orbit rewritten}, we obtain
\[
\lim_{T\to\infty}
\frac{1}{\vol(G_T)}\int_{G_T}\varphi(g\Ga)\,f(g\cdot[\La])\,dg
=
\frac1m\sum_{i=1}^m\int_0^{2\pi}
f\bigl(k_{\theta-\theta_0}\cdot[\rho_{\theta_0}^\infty(\kappa_{-(\theta-\theta_0)/2}\sigma_i)\La]\bigr)
\,w(\theta)\,d\theta
+O(\e).
\]
After the change of variables $\phi=\theta-\theta_0$ and by $2\pi$-periodicity, this is exactly the formula in \Cref{thm:degen orbits low dim}\eqref{enu:special case nondeg}, since $w(\phi+\theta_0)=w_{\theta_0}(\phi+\theta_0)$ by definition.

Since \(\e>0\) is arbitrary, \Cref{prop:reduction general method} now implies
\Cref{thm:degen orbits low dim} in the special case.


\begin{thebibliography}{EMSS16}

\bibitem[AES16]{AES_lat}
Menny Aka, Manfred Einsiedler, and Uri Shapira, \emph{Integer points on spheres and their orthogonal lattices}, Invent. Math. \textbf{206} (2016), 379–396.

\bibitem[BQ16]{BenoistQuintBook}
Yves Benoist and Jean-Fran\c{c}ois Quint, \emph{Random walks on reductive groups}, Ergebnisse der Mathematik und ihrer Grenzgebiete. 3. Folge, vol.~62, Springer, 2016.

\bibitem[EM93]{EskinMcMullen}
Alex Eskin and Curt McMullen, \emph{Mixing, counting, and equidistribution in {L}ie groups}, Duke Math. J. \textbf{71} (1993), no.~1, 181--209. \MR{1230290}

\bibitem[EMSS16]{emss}
Manfred Einsiedler, Shahar Mozes, Nimish Shah, and Uri Shapira, \emph{Equidistribution of primitive rational points on expanding horospheres}, Compos. Math. \textbf{152} (2016), 667–692.

\bibitem[GLS24]{gorodnik2022stationary}
Alexander Gorodnik, Jialun Li, and Cagri Sert, \emph{Stationary measures for $\text{SL}(2,\mathbb{R})$-actions on homogeneous bundles over flag varieties}, J. Reine Angew.\ Math. \textbf{2024} (2024), no.~814, 47--89.

\bibitem[GN12]{GorodnikNevo2012}
Alexander Gorodnik and Amos Nevo, \emph{Counting lattice points}, J. Reine Angew. Math. \textbf{2012} (2012), no.~663, 127--176.

\bibitem[GOS09]{shah_oh_gor_satake_2009}
Alexander Gorodnik, Hee Oh, and Nimish Shah, \emph{Integral points on symmetric varieties and {S}atake compactifications}, Amer. J. Math. \textbf{131} (2009), no.~1, 1--57. \MR{2488484}

\bibitem[GW04]{Gorodnik2004DistributionOL}
Alexander Gorodnik and Barak Weiss, \emph{Distribution of lattice orbits on homogeneous varieties}, Geom. Funct. Anal. \textbf{17} (2004), 58--115.

\bibitem[Maa59]{Maass59}
Hans Maass, \emph{\"{U}ber die {V}erteilung der zweidimensionalen {U}ntergitter in einem euklidischen {G}itter}, Math. Ann. \textbf{137} (1959), 319--327.

\bibitem[Mar88]{MargulisProceedingsVarna86}
Grigory~A. Margulis, \emph{Lie groups and ergodic theory}, Algebra Some Current Trends (Berlin, Heidelberg) (Luchezar~L. Avramov and Kerope~B. Tchakerian, eds.), Springer Berlin Heidelberg, 1988, pp.~130--146.

\bibitem[Moo66]{Moore1966}
Calvin~C. Moore, \emph{Ergodicity of flows on homogeneous spaces}, Amer.\ J.\ Math. \textbf{88} (1966), no.~1, 154--178.

\bibitem[Roe56]{Roelcke56}
Walter Roelcke, \emph{\"{U}ber die {V}erteilung der {K}lassen eigentlich assoziierter zweireihiger {M}atrizen, die sich durch eine positiv-definite {M}atrix darstellen lassen}, Math. Ann. \textbf{131} (1956), 260--277.

\bibitem[Sch84]{Schlichtkrull_book}
Henrik Schlichtkrull, \emph{Hyperfunctions and harmonic analysis on symmetric spaces}, Progress in Mathematics, vol.~49, Birkhäuser Boston, Inc., Boston, MA, 1984. \MR{757178}

\bibitem[Sch98]{Schmidt98}
Wolfgang~M. Schmidt, \emph{The distribution of sublattices of {${\bf Z}^m$}}, Monatsh. Math. \textbf{125} (1998), no.~1, 37--81.

\bibitem[Sha91]{Shah1991}
Nimish Shah, \emph{Uniformly distributed orbits of certain flows on homogeneous spaces}, Math.\ Ann. \textbf{289} (1991), 315--334.

\bibitem[Sha96]{Shah1996LimitDO}
Nimish~A. Shah, \emph{Limit distributions of expanding translates of certain orbits on homogeneous spaces}, Proc. Indian Acad. Sci. Math. Sci. \textbf{106} (1996), 105--125.

\bibitem[SS17]{Sargent2017DynamicsOT}
Oliver Sargent and Uri Shapira, \emph{Dynamics on the space of 2-lattices in 3-space}, Geom.\ Funct.\ Anal. \textbf{29} (2017), 890--948.

\bibitem[Yan22]{Yang_prod_of_SOn1}
Lei Yang, \emph{Equidistribution of expanding translates of curves in homogeneous spaces with the action of {$({\rm SO}(n,1))^k$}}, Acta Math. Sin. (Engl. Ser.) \textbf{38} (2022), no.~1, 205--224. \MR{4372213}

\end{thebibliography}

\end{document}